%% file: statification.tex
\documentclass[12pt]{amsart}
\usepackage[utf8]{inputenc}

\usepackage{amsmath, amsfonts, amssymb, amsthm, mathrsfs, graphicx,float, mathtools}
\usepackage{array}
\usepackage{comment}
\usepackage[shortlabels]{enumitem}
\usepackage{booktabs,longtable}
\usepackage[numbers]{natbib}
\usepackage[margin=1in]{geometry}
\usepackage{stmaryrd} 
\usepackage{xcolor}
\usepackage{tikz-cd}
\usepackage{tikz}
\usepackage[loadshadowlibrary]{todonotes}

\usepackage{yfonts}

\graphicspath{{graphics/}}

\input{macros.tex}

\title{A stratification of moduli of arbitrarily singular curves}
\author{Sebastian Bozlee, Christopher Guevara, and David Smyth}
\date{\today} 

\begin{document}

\begin{abstract}
We introduce a new moduli stack $\EE_{g,n}$ of ``equinormalized curves,'' closely related to the moduli stack of all reduced,
connected algebraic curves. We construct a stratification $\bigsqcup_\Gamma \EE_\Gamma$ of $\EE_{g,n}$ indexed by generalized dual graphs and prove
that each stratum $\EE_{\Gamma}$ is a fiber bundle over a finite quotient of a product of $\moduli_{g,n}$s. The fibers are moduli schemes
parametrizing subalgebras of a fixed algebra, and are in principle explicitly computable as locally closed subschemes of products of
Grassmannians. We thus obtain a remarkably explicit geometric description of moduli of reduced curves with arbitrary singularities.
\end{abstract}

\maketitle

\tableofcontents

\section{Introduction}

One of the most useful features of the moduli stack of stable pointed curves $\Moduli_{g,n}$ is its stratification
into curves of fixed topological type indexed by dual graphs. Dual graphs and the associated stratification have played an essential role in the study of moduli stacks of algebraic curves since the first papers on stable curves \cite{deligne_mumford,knudsen_projectivity_II}, appearing for example in the Picard group of $\Moduli_{g,n}$ \cite{arbarello_cornalba_pic}, its Chow ring \cite{keel_intersection_theory_m0n}, connections to tropical geometry \cite{chan_tropical_intro, ccuw}, and alternate compactifications of $\moduli_{g,n}$ \cite{smyth_zstable}.

The purpose of this paper is to construct a similar stratification
for moduli of connected, reduced pointed curves with arbitrary isolated singularities.
Our strata $\EE_\Gamma$ are indexed by ``combinatorial types'' $\Gamma$ generalizing dual graphs. These are bipartite graphs encoding the combinatorial data of which singularities meet which components, as well
as other discrete invariants which we describe more precisely below.

Na\"ively, one might hope to directly stratify the moduli stack $\UU_{g,n}$ of all connected, reduced pointed curves with arbitrary isolated singularities.

\begin{definition} (\cite[Appendix B]{smyth_zstable}) \label{def:U_gn}
For non-negative integers $g, n$, we define $\UU_{g,n}$ (or $\UU_{g}$, if $n = 0$) to be the stack whose fiber over a scheme $S$ is the category of flat, proper, finitely presented morphisms of algebraic spaces $\pi : C \to S$ with connected, reduced, pure one-dimensional geometric fibers of arithmetic genus $g$, together with $n$ sections $p_1, \ldots, p_n : S \to C$.
\end{definition}

However, stratifying $\UU_{g,n}$ would require analyzing the deformation spaces of arbitrary curve singularities, a daunting prospect. We will work instead on a stack $\EE_{g,n}$ parametrizing curves $C$ in $\UU_{g,n}$ together with a normalization $\nu : \tilde{C} \to C$ and a lift of its markings to $\tilde{C}$.

\begin{definition} \label{def:E_gn}
Let $g \geq 0, n \geq 0$ be integers. Given a scheme $S$, a \emphbf{family of $n$-pointed equinormalized curves of genus $g$ over $S$}
consists of
\begin{enumerate}
  \item A smooth and proper morphism of algebraic spaces $\tilde{\pi} : \tilde{C} \to S$, with pure one-dimensional geometric fibers, together with $n$ sections $p_1, \ldots, p_n : S \to \tilde{C}$;
  \item A family of curves $\pi : C \to S$ in $\UU_{g}(S)$;
  \item An affine morphism $\nu : \tilde{C} \to C$ over $S$,
\end{enumerate}
such that
\begin{enumerate}
  \item $\nu^\sharp : \OO_{C} \to \nu_*\OO_{\tilde{C}}$ is injective,
  \item the $\OO_C$ module $\Delta \coloneqq \nu_*\OO_{\tilde{C}}/\OO_C$ is supported on a subscheme of $C$ finite over $S$,
  \item $\pi_*\Delta$ is locally free of finite rank on $S$
\end{enumerate}
A family of equinormalized curves over the spectrum of an algebraically closed field will simply be called a \emphbf{normalized curve}.

An isomorphism of such families is defined in the obvious way. Define the \emphbf{stack of equinormalized curves} $\EE_{g,n}$ to be the stack over $\Sch$ whose fiber over $S$ is the category of families of $n$-marked equinormalized curves of genus $g$ over $S$.
\end{definition} 

Occasionally we will consider marked points or subschemes of $\tilde{C}$ corresponding to the preimages of singularities of $C$. We will always refer to these as \emph{branch markings} or \emph{branch subschemes} to distinguish them from the ordinary marked points considered above.

\medskip

There is a natural morphism of fibered categories $\EE_{g,n} \to \UU_{g,n}$ taking
\[
  (\nu : \tilde{C} \to C, p_1, \ldots, p_n) \mapsto (C, \nu(p_1),\ldots, \nu(p_n))
\]

The stack $\EE_{g,n}$ has essentially the same geometric points as $\UU_{g,n}$ in the sense of the proposition below. The point is that a singular curve extends to a normalized curve in a unique way by the universal property of normalization.

\begin{proposition}[see Proposition \ref{prop:equinormalized_to_ugn}]
Let $k$ be an algebraically closed field. The restriction of $\EE_{g,n}(k) \to \UU_{g,n}(k)$ to the locus with smooth markings is a bijection on isomorphism classes of $k$-points, and indeed is an equivalence of categories. If we also consider the locus where markings can collide with singularities, the map $\EE_{g,n}(k) \to \UU_{g,n}(k)$ is finite-to-one on isomorphism classes of $k$ points. 
\end{proposition}

However, the topology is markedly different.

\begin{example}
The preimage of $\Moduli_{g,n} \subseteq \UU_{g,n}$ in $\EE_{g,n}$ is naturally isomorphic to the disjoint union of the usual strata in $\Moduli_{g,n}$. Intuitively, the normalization of one stable curve can be continuously deformed into the normalization of another
stable curve only if the two stable curves have the same dual graph. See Example \ref{ex:stable_equinormalized_curves}.
\end{example}

Accordingly, we conceive of $\EE_{g,n}$ as a partial stratification of $\UU_{g,n}$, although this is not literally the case in positive characteristic, see Example \ref{ex:egn_not_stratification_positive_char}.

Even so, there is further stratification to do in order to arrive at a stratification by combinatorial type analogous to that by dual graph. For example, singularities may collide or split apart, which we will take to change the underlying combinatorial type.

\begin{example} \label{ex:cusp_is_limit_of_nodes}
Consider the family of curves $C$ over $\AA^1 = \Spec \ZZ[t]$ cut out by the equation $y^2z = x(x - tz)^2$ in $\PP^2$. The fibers over $t \neq 0$ are nodal cubics and the special fiber over $t = 0$ is a cuspidal cubic. The family $C$ can be enhanced to a family of equinormalized curves
\[
\begin{tikzcd}
  \tilde{C} = \PP^1_{\AA^1_t} \ar[r,"\nu"] \ar[rd, "\tilde{\pi}"] & C \ar[d, "\pi"] \\
    & \AA^1_t,
\end{tikzcd}
\]
where the map $\nu$ glues the point $\sqrt{t}$ of $\PP^1$ to the point $-\sqrt{t}$ when $t \neq 0$, and crimps the tangent vector at the origin to a cusp when $t = 0$.
Explicitly, if we write $\tilde{C} = \Proj \ZZ[u,v,t]$ where $u,v$ have degree one and $t$ has degree 0, then we may set $\nu : \tilde{C} \to C$ to be the map induced by the morphism of graded rings $\ZZ[x,y,z,t] /(y^2z - x(x-tz)^2) \to \ZZ[u,v,t]$ taking $x \mapsto u^2v$, $y \mapsto u(u^2-tv^2)$, and $z \mapsto u^3.$

We will take the curves over $D(t)$---the nodal cubics---and the curve over $V(t)$---a cuspidal cubic---to have different combinatorial types.
\end{example}

\bigskip

We now state our main theorems, leaving some details for when we have established more background.

\begin{customthm}{I}[see Theorem \ref{thm:algebraicity} and Theorem \ref{thm:stratification_by_gamma}]
For each combinatorial type $\Gamma$, there is an algebraic stack $\EE_{\Gamma}$ over $\QQ$ parametrizing equinormalized curves of type $\Gamma$. 
\end{customthm}

The algebraic stacks $\EE_{\Gamma}$ form a locally closed stratification of $\EE_{g,n}$
in the following sense.

\begin{customthm}{II}[see Proposition \ref{prop:stratification_by_delta_deltaprime} and Theorem \ref{thm:stratification_by_gamma}]
For any $\QQ$-scheme $S$ and morphism $f : S \to \EE_{g,n}$ there is a decomposition $\sqcup_{\Gamma} S_{\Gamma} \to S$ of $S$ into disjoint locally closed subschemes, compatible with base change, such that $f|_{S_{\Gamma}}$ factors through $\EE_\Gamma$. 
\end{customthm}

We reserve a proof of the algebraicity of the full ambient stack $\EE_{g,n}$ for future work as it requires a technical argument somewhat orthogonal to the main ideas of this paper. However, we do show algebraicity of a coarse stratification of $\EE_{g,n}$ in which only a pair of numerical invariants are fixed rather than the full data of a combinatorial type.

\begin{customthm}{III}[see Theorem \ref{thm:algebraicity} and Theorem \ref{thm:stratification_by_gamma}]
For each $\delta, c$:
\begin{itemize}
    \item There is an algebraic stack $\EE_{g,n}^{\delta,c}$ parametrizing equinormalized curves of total delta invariant $\delta$ and total rank of the conductor locus $c$.
    \item There is a moduli stack $\moduli^\delta_{g,n}$ parametrizing certain smooth, possibly disconnected $n$-pointed curves. The forgetful morphism $\EE_{g,n}^{\delta,c} \to \moduli^\delta_{g,n}$ taking $(\nu : \tilde{C} \to C, p_1, \ldots, p_n)$ to $(\tilde{C}, p_1, \ldots, p_n)$ is separated and representable by schemes.
    \item For any scheme $S$ and morphism $f : S \to \EE_{g,n}$ there is a decomposition $\sqcup_{\delta,c} S^{\delta,c} \to S$ of $S$ into disjoint locally closed subschemes, compatible with base change, such that $f|_{S^{\delta,c}}$ factors through $\EE^{\delta,c}_{g,n}$.
    \item When $S$ is a $\QQ$-scheme, the decomposition $\sqcup_{\delta,c} S^{\delta,c}$ is refined by the decomposition $\sqcup_{\Gamma} S_{\Gamma}.$
\end{itemize}
\end{customthm}

Our last main result gives a geometric description of $\EE_{\Gamma}.$ We recall that the stratum of $\Moduli_{g,n}$ parametrizing stable curves $C$ with dual graph $\Gamma$ is isomorphic to a finite quotient of a product of $\moduli_{g,n}$s by taking $C$ to its pointed normalization $\tilde{C}$.
Analogously, there is a map $\EE_{\Gamma} \to \moduli_\Gamma$
where $\moduli_\Gamma$ is a finite quotient of a product of $\moduli_{g,n}$s (restricted to $\Spec \QQ$), taking an equinormalized curve $\nu : \tilde{C} \to C$ to $\tilde{C}$ along with the unordered preimages of $C$'s singularities.

However, this map is not an isomorphism. Since $C$ may have more complicated singularities than nodes, it is necessary to provide some further information to recover $\nu : \tilde{C} \to C$ from $\tilde{C}$, namely that of a sheaf of subalgebras of $\OO_{\tilde{C}}$. In fact, each stratum $\EE_\Gamma$ forms a fiber bundle over $\moduli_{\Gamma}$.

We state this as our last main result.

\begin{customthm}{IV}[see Theorem \ref{thm:e_gamma_fiber_bundle}] \label{thm:e_gamma_fiber_bundle_intro}
The morphism of $\QQ$-algebraic stacks $\EE_{\Gamma} \to \moduli_\Gamma$ is a fiber bundle in the \'etale topology with fiber an explicitly computable moduli scheme parametrizing certain subalgebras of a fixed algebra.
\end{customthm}

Therefore, \emph{by combining the well-understood geometry of moduli of smooth curves with the geometry of moduli of subalgebras, we open up a new, tractable approach to understanding moduli of singular curves}.

\medskip

The restriction to characteristic zero in Theorem I, II, and IV is due to technical issues related to stratifying the Hilbert scheme of points of a smooth curve, see Example \ref{ex:positive_characteristic_bad}. The majority of our foundational work is independent of characteristic, and we suspect that similar results could be achieved over $\FF_p$ in which the $\EE_\Gamma \to \moduli_{\Gamma}$ are fppf-bundles. 

\subsection{Strategy}
The central observation motivating this paper is that a singular curve $C$ can be recovered from its normalization $\tilde{C}$ by specifying a subalgebra of $\OO_{\tilde{C}}$. Our strategy is to develop the theory of moduli of such subalgebras
with the necessary finesse to apply it to moduli stacks of algebraic curves.

Moduli schemes, called \emphbf{territories}, parametrizing subalgebras of a fixed algebra were originally considered by Ishii in \cite{ishii_moduli_subrings}. One quickly runs into a major difficulty when applying her construction to curves: roughly speaking, families of subalgebras of 1-dimensional algebras over non-reduced schemes are far too numerous for their moduli to be representable by a scheme of finite type \cite[pg 507, remark (ii)]{ishii_moduli_subrings}. In \cite{ishii_moduli_subrings}, this difficulty is avoided by restricting the moduli functor to reduced schemes. This has the unfortunate side effect of throwing out any non-reduced structure that territories might rightfully have.

\bigskip

Fortunately, a good scheme structure can be recovered with some more technique. One can show that it suffices to parametrize subalgebras of a finite rank quotient $\OO_{\tilde{C}} / \mathscr{I}$ where $\mathscr{I}$ is a sheaf of ideals satisfying $\nu_*\OO_{\tilde{C}} \supseteq \OO_C \supseteq \mathscr{I}$. Accordingly, in Section \ref{ssec:territories_algebras}, we present a construction of a genuine moduli scheme parametrizing subsheaves of algebras of locally free finite rank sheaves of algebras, which we again call a territory. This generalizes the main construction of \cite{ishii_moduli_subrings} which considered specifically the subalgebras of a $k$-algebra. Our approach involves a direct description of the equations cutting out territories inside of Grassmannians, which we find simpler than Ishii's original argument using Noetherian induction.

Next, one is led to consider what ideal $\mathscr{I}$ to mod out by. The most natural choice is the largest possible choice: the conductor ideal $\cond_{\OO_{\tilde{C}}/\OO_C} = \Ann_{\OO_C}(\nu_*\OO_{\tilde{C}})$. Once again there is a major difficulty: conductor ideals do not commute with base change in general. In Section \ref{ssec:territories_conductors} we make the key observation that if the ``total conductance'' $\dim_{k(x)} ( \nu_*\OO_{\tilde{C}} / \cond_{\OO_{\tilde{C}}/\OO_C} )$ is locally constant for $x \in S$, then conductors do commute with base change. In such a situation, $\nu_*\OO_{\tilde{C}} / \cond_{\OO_{\tilde{C}}/\OO_C}$ is locally free, and we call its rank the \emphbf{conductance} $c$ of $\tilde{C}$ over $C$. Therefore, by passing to an appropriate locally closed stratification of the territory parametrizing subalgebras of fixed conductance, we can ensure that conductor ideals do commute with base change. This gives us the technical groundwork to prove algebraicity of $\EE_{g,n}^{\delta,c}$.

\bigskip

Section \ref{sec:territories_singularities} studies the territories of the specific algebras needed for our application to curves, those of the form $A_{\vec{c}} = \ZZ[t_1]/t_1^{c_1} \times \cdots \times \ZZ[t_m]/t_m^{c_m}$ where $\vec{c} = (c_1, \ldots, c_m)$. The main result is the decomposition of the territory of $A_{\vec{c}}$ into a disjoint union of products of ``territories of curve singularities.'' In subsection \ref{ssec:crimping}, we prove that the crimping spaces of van der Wyck \cite{fred_thesis} are orbits of the action of a subgroup of $\Aut(A_{\vec{c}})$ on the territory of a curve singularity.

\bigskip

In Section \ref{sec:territories_schemes}, we promote the construction of territories from algebras to schemes with a specified sheaf of ideals $\mathscr{I}$. Descent allows us to quickly generalize to representable morphisms of algebraic stacks.

\bigskip

In Section \ref{sec:equinormalized_curves}, we consider the stack of equinormalized curves in more detail. In light of our observations on conductance, we introduce the substack $\EE^{\delta,c}_{g,n}$ of $\EE_{g,n}$ parametrizing equinormalized curves of fixed total delta invariant and conductance.
We then use territories of stacks to construct $\EE^{\delta,c}_{g,n}$ as an algebraic stack representable over a suitable moduli stack of pointed curves.

\bigskip

In Section \ref{sec:hilbert_stratification}, we consider the problem of stratifying the Hilbert scheme of points of a curve according to how many points collide. This is an essential step in the construction of our stratification of $\EE_{\Gamma}$s. In order for this stratification to behave well, we introduce a characteristic 0 hypothesis.

\bigskip

In Section \ref{sec:combinatorial_type}, we introduce combinatorial types, our notion of generalized dual graph. The \emphbf{combinatorial type} $\Gamma(\nu)$ of an equinormalized curve $\nu : \tilde{C} \to C$ is a bipartite graph with a vertex for each singularity of $C$,
a vertex for each irreducible component of $C$, and an edge for each branch of a singularity from the singularity to the component to which it belongs, as well as additional data recording
locations of marked points, certain discrete invariants of singularities, and genera of components. The invariant of combinatorial type is sufficiently fine to ensure that the
morphism of Theorem IV is a fiber bundle.

We draw combinatorial types with circles for components, squares for singularities, half-edges for distinguished points,
and labels denote the markings, genera, and branch conductances. See Figure \ref{fig:combinatorial_type_example}.

\begin{figure}[h]
\begin{center}
\includegraphics[width=3.5in]{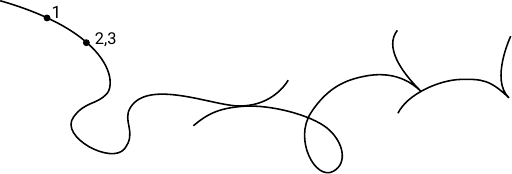}

\bigskip

\begin{tikzpicture}[scale=1.5]
\draw (0,0) --node[above]{\tiny$3$} (1,0) --node[above]{\tiny$3$} (2,0) --node[above]{\tiny$2$} (3,0) --node[above]{\tiny$1$} (4,0) -- node[above]{\tiny$2$} (5,0);
\draw (-.5,.4) node[left]{\tiny$1$} -- (0,0);
\draw (-.5,-.4) node[left]{\tiny$2,3$} -- (0,0);
\draw (2,0) to[out=-60,in=60]node[right]{\tiny$1$} (2,-1) (2,0) to[out=-120,in=120] node[left]{\tiny$1$} (2,-1);
\draw[fill=white]
  (0,0) circle (5pt)
  (1,0) +(-.2,-.2) rectangle +(.2,.2)
  (2,0) circle (5pt)
  (3,0) +(-.2,-.2) rectangle +(.2,.2)
  (4,0) circle (5pt)
  (5,0) +(-.2,-.2) rectangle +(.2,.2)
  (2,-1) +(-.2,-.2) rectangle +(.2,.2);
\node at (0,0) {\tiny$1$};
\node at (1,0) {\tiny$2$};
\node at (2,0) {\tiny$0$};
\node at (3,0) {\tiny$1$};
\node at (4,0) {\tiny$0$};
\node at (5,0) {\tiny$1$};
\node at (2,-1) {\tiny$0$};
\end{tikzpicture}
\end{center}

\caption{A sketch of a singular pointed curve and the combinatorial type of its normalization map. The singularities, from left to right, are an $A_5$ singularity $y^2 = x^6$ of genus 2, a node, a cusp meeting a smooth branch transversely,
and a cusp.}
\label{fig:combinatorial_type_example}
\end{figure}

\bigskip

In Section \ref{sec:stratification}, we construct the algebraic stacks $\EE_\Gamma$ and $\moduli_\Gamma$, prove that the $\EE_\Gamma$s stratify
$\EE^{\delta,c}_{g,n}$, and prove that $\EE_\Gamma \to \moduli_\Gamma$ is a fiber bundle. The strategy is to first construct a territory bundle over a moduli stack
of smooth curves with both ordered markings and ordered branch markings. Once this is done, we then forget the ordering of branch markings, and analyze the result.

\bigskip

We conclude the paper with an extended example in Section \ref{sec:territory_colliding_sings}, studying the stratification of $\EE^{2,4}_{2,0}$. Although it is located at the end of the main text, we encourage the reader to refer to it as early as possible.

\bigskip

We also include an index of notation at the end of the paper.

\subsection{Relation to other work}

There has been considerable interest in alternate compactifications of the moduli stack of curves \cite{schubert_pseudostable, hassett_weighted_curves, smyth_zstable, fedorchuk_smyth, battistella_m_stable, bkn_qstable, blankers_bozlee}, and in the moduli stacks appearing in the log minimal model program for $\Moduli_{g,n}$ \cite{hassett_hyeon, alper_fedorchuk_smyth_flip2}. Moduli of curves with singularities worse than nodes have also been used to construct smooth modular compactifications of moduli stacks of stable maps in genus 1 and 2 \cite{rsw, battistella_carocci_smooth_maps, zheng2021moduli}, with applications to realizability of tropical curves \cite{rsw2} and smoothability of stable maps \cite{battistella_carocci_plane_curves}. The stratification of the usual compactification $\Moduli_{g,n}$ by dual graph has been a central tool in the study of moduli of curves since \cite{knudsen_mumford_projectivity_I}. Therefore it is desirable to stratify the alternate modular compactifications of $\moduli_{g,n}$ as well in order to better access their geometry. Our results are a first step in this direction.

In the papers \cite{bkn_qstable, blankers_bozlee}, a stratification of Gorenstein pointed curves of genus one was a key ingredient in obtaining their classifications of modular compactifications of $\moduli_{1,n}$ by Gorenstein curves. The stratification presented in this paper is considerably more general. We therefore anticipate applications to the study of alternate modular compactifications of $\moduli_{g,n}$.

Moduli stacks of curves with singularities worse than nodes were used to determine the integral Chow ring of $\Moduli_{2,1}$ in \cite{dilorenzo_pernice_vistoli} and the almost-integral Chow ring of $\Moduli_{3}$ in the series of papers \cite{pernice_paper1, pernice_paper2,pernice_paper3,pernice_paper4}. More recently, the papers \cite{battistella_dilorenzo_wall_crossing, battistella_dilorenzo_all_elliptic_chow} compute the integral Chow rings of all modular compactifications of $\moduli_{1,n}$ by Gorenstein curves, including many curves with singularities worse than nodes.  In light of these results, it seems plausible that moduli stacks of curves with worse
than nodal singularities will have further applications to the intersection theory of $\Moduli_{g,n}$ and related moduli stacks. Our results promise a way to approach the Chow rings of strata within such moduli stacks of worse-than-nodal curves. In particular, the territories of singularities with certain small numerical invariants are simply $\AA^1$ and $\GG_m$, which make the corresponding fiber bundles $\EE_\Gamma \to \moduli_{\Gamma}$ into $\AA^1$-bundles or $\GG_m$-bundles, respectively. These sorts of spaces should have approachable Chow rings relative to those of the base.

A natural question related to our stratification is: which strata $\EE_{\Gamma}$ intersect the closure of other strata? This amounts to the question of which singularities are formed when singularities are permitted to collide, which has been studied in the case
of plane curves in \cite{kerner_on_collisions, kerner_equinormalizable_deformations}. Our framework provides a promising setting within which to study such questions for abstract (not necessarily planar) families of curves: see
Section \ref{ssec:curves_normalized_by_p1}. We direct the reader to \cite{kerner_equinormalizable_deformations} for a review of the related literature.

Simultaneous normalization of families of curves was studied in the complex analytic setting in \cite{teissier} and in the algebraic setting in \cite{chiang-hsieh_lipman, greuel_pfister_fiberwise_normalization}. The main result is that a flat family of curves over a normal base admits a simultaneous normalization if and only if the delta invariant of its fibers is constant. These papers can be seen as providing a criterion for lifting families in $\UU_{g,n}$ to families in $\EE_{g,n}$.

Ishii \cite{ishii_moduli_subrings} introduced territories for the purpose of studying singularities with fixed normalization, proved that territories of local rings are connected, and described a stratification of territories of unibranch singularities by semigroup. In a subsequent paper, she constructed moduli spaces of curves with unibranch singularities of fixed $\delta$-invariant \cite{ishii_moduli_equisingular_curves}. Our spaces $\EE_{\Gamma}$ may be regarded as a generalization of these spaces. Hamilton \cite{hamilton_complete_subalgebras} studied subalgebras of $\CC\llbracket t \rrbracket$, and gave an explicit algorithm for computing defining equations for moduli of unibranch singularities with fixed semigroup. Bavula \cite{Bavula} constructed a variety parametrizing the subalgebras of $\CC[t]$ containing some power of the ideal $(t)$ and described the automorphism groups of these subalgebras. Guevara \cite{chris_thesis} studied the connected subalgebras of $\CC\llbracket t_1 \rrbracket \times \cdots \times \CC\llbracket t_m \rrbracket$, producing a recursive algorithm for computing them, and an ``Isom-Hilb'' stratification of associated territories. These works can be viewed as giving algorithms for determining (in certain special cases) the defining equations of the territories discussed in Section \ref{sec:territories_singularities}.

Our territories of singularities are related to the crimping spaces of van der Wyck's thesis \cite{fred_thesis}. These crimping spaces parametrize the possible ways to normalize a fixed isomorphism class of singularity. In comparison, the isomorphism class of a singularity may vary inside of a fixed territory of a curve singularity. See Theorem \ref{thm:crimp_vs_ter} for a more precise comparison.

\subsection{Acknowledgements}

We are grateful for conversations with Jarod Alper, Dori Bejleri, Luca Battistella, Aise Johan de Jong, Gavril Farkas, Maksym Fedorchuk, Leo Herr, David Holmes, Nathan Ilten, Jake Levinson, Han-Bom Moon, Patrick McFaddin, Ian Morrison, Anand Patel, Michele Pernice, Dhruv Ranganathan, Dave Swinarski, Montserrat Teixidor i Bigas, and Jonathan Wise. We also thank the anonymous referee for many suggestions that have improved the paper and helped to uncover issues in positive characteristic.

\subsection{Tool and computational resource disclosure}

Macaulay2 \cite{M2} was used for several computer algebra computations identified in the text. ChatGPT (OpenAI) and Claude (Anthropic) were used during the preparation of the manuscript for literature discovery, consistency and error checking, and editorial assistance, including the index of notation. All outputs were reviewed and edited by the authors, who take full responsibility for the mathematical content and the final text.

\section{Moduli of subsheaves of algebras}
\label{sec:territories_algebras}

\subsection{Territories of sheaves of algebras of fixed corank}
\label{ssec:territories_algebras}

It will be important for us to consider the possible extensions of a family of smooth curves $\tilde{\pi} : \tilde{C} \to S$
to a family of equinormalized curves $\nu : \tilde{C} \to C$. Affine locally, this amounts to choosing a subalgebra $\OO_C$ of $\OO_{\tilde{C}}$. 

In this section we construct an extension of Ishii's moduli functors of subalgebras \cite{ishii_moduli_subrings} from families of subalgebras of a $k$-algebra to families of subalgebras of a quasicoherent sheaf of algebras. Ishii considered essentially the same functor in \cite[Section 1]{ishii_moduli_equisingular_curves}, however, in order to give a self-contained exposition and in order to set up some generalizations
to our setting we provide our own construction.

For the remainder of this section, let $S$ be a scheme and $\mathscr{A}$ a quasicoherent sheaf of $\OO_S$-algebras. For brevity, given an $S$-scheme $f : T \to S$, we may write $f^*\mathscr{A}$ or $\mathscr{A}_T$ (if $f$ is understood) for $\mathscr{A} \otimes_{\OO_S} \OO_T$.
Let $\delta$ be a non-negative integer. We define the following functor on $S$-schemes. 

\begin{definition} \label{def:territory}
Given an $S$-scheme $f: T \to S$, a \emphbf{family of subalgebras of $\mathscr{A}$ of corank $\delta$} on $T$ is a quasi-coherent $\OO_T$-subalgebra $\mathscr{B}$ of $\mathscr{A}_T$ such that the quotient $\OO_T$-module $\mathscr{A}_T / \mathscr{B}$ is locally free of rank $\delta$.

We define a functor $F^\delta_{\mathscr{A}} : (\Sch/S)^{op} \to \Set$ by:
\begin{enumerate}
    \item If $f : T \to S$ is an $S$-scheme, $F^\delta_{\mathscr{A}}(T \to S)$ is the set of families of subalgebras of $\mathscr{A}$ of corank $\delta$ on $T$.
    \item If $g : T \to T'$ is a morphism of $S$-schemes, then $F^\delta_{\mathscr{A}}(g) : F^\delta_{\mathscr{A}}(T') \to F^\delta_{\mathscr{A}}(T)$ is defined by taking a family of subalgebras to its pullback.
\end{enumerate}

Given such an algebra $\mathscr{B}$, we write $\Delta$ for the quotient $\mathscr{A}_T/\mathscr{B}$. Observe that $\Delta$ may be regarded
as either a $\mathscr{B}$-module, or as an $\OO_T$-module.

\end{definition}

There are several associated kinds of functoriality.

\begin{lemma} \label{lem:ter_base_change}
Let $T \to S$ be a morphism. Then there is a canonical isomorphism $F^\delta_{\mathscr{A}} \times_S T \cong F^\delta_{\mathscr{A}_T}$.
\end{lemma}

It is natural to consider subalgebras of subalgebras. When $\mathscr{A}$ is finite locally free, the subalgebras of a fixed subalgebra form
a closed subfunctor.

\begin{lemma} \label{lem:subalgebra_of_subalgebra}
Let $\alpha + \beta = \delta$ be an integer partition of $\delta.$
Let $\mathscr{B} \in F^\beta_{\mathscr{A}}(S)$. Then $F^{\alpha}_{\mathscr{B}}$ is a subfunctor of $F^\delta_{\mathscr{A}}$. If $\mathscr{A}$ is in addition a locally free $\OO_S$-algebra of finite rank, then $F^{\alpha}_{\mathscr{B}}$ is a closed subfunctor of $F^\delta_{\mathscr{A}}$.
\end{lemma}
\begin{proof}
We first show that $F^\alpha_{\mathscr{B}}$ is a subfunctor of $F^\delta_{\mathscr{A}}$. 
Let $T \overset{f}{\to} S$ be an $S$-scheme and $\mathscr{C} \in F^\alpha_{\mathscr{B}}(T \overset{f}{\to} S)$. We want to show that
$\mathscr{C}$ is a family of subalgebras of $f^*\mathscr{A}$ of corank $\delta$.
Let $\Delta_{\mathscr{A}/\mathscr{B}}, \Delta_{\mathscr{B}/\mathscr{C}}, \Delta_{\mathscr{A}/\mathscr{C}}$ be the respective cokernels of $\mathscr{B} \to \mathscr{A}$, $\mathscr{C} \to f^*\mathscr{B}$, and $\mathscr{C} \to f^*\mathscr{B} \to f^*\mathscr{A}$. With the aid of the snake lemma we can form a diagram
\[
\begin{tikzcd}
  & & 0 \ar[d] & 0 \ar[d] & \\
  0 \ar[r] & \mathscr{C} \ar[r] \ar[d, equals] & f^*\mathscr{B} \ar[r] \ar[d] & \Delta_{\mathscr{B}/\mathscr{C}} \ar[r] \ar[d] & 0 \\
  0 \ar[r] & \mathscr{C} \ar[r] & f^*\mathscr{A} \ar[d] \ar[r] & \Delta_{\mathscr{A}/\mathscr{C}} \ar[r] \ar[d] & 0 \\
  & & f^*\Delta_{\mathscr{A}/\mathscr{B}} \ar[r, equals] \ar[d] & f^*\Delta_{\mathscr{A}/\mathscr{B}} \ar[d] &  \\
  & & 0 & 0 & 
\end{tikzcd}
\]
Since $\Delta_{\mathscr{A}/\mathscr{B}}$ is locally free of rank $\beta$ and $\Delta_{\mathscr{B}/\mathscr{C}}$ is locally free of rank $\alpha$, it follows that $\Delta_{\mathscr{A}/\mathscr{C}}$ is locally free of rank $\delta = \alpha + \beta$, as required.

It is clear that the inclusion $F^\alpha_{\mathscr{B}} \to F^\delta_{\mathscr{A}}$ is natural.

Now suppose that $\mathscr{A}$ is locally free of finite rank. To check that $F^\alpha_{\mathscr{B}}$ is a closed subfunctor of $F^\delta_{\mathscr{A}}$ we must show that for each $S$-scheme $T \overset{f}{\to} S$ and each $\mathscr{C} \in F^\delta_{\mathscr{A}}(T \overset{f}{\to} S)$, there is a closed subscheme $Z$ of $T$ with the universal property that for all $S$-morphisms $g : W \to T$, $g^*\mathscr{C} \subseteq (f \circ g)^*\mathscr{B}$ if and only if $g$ factors through $Z$.

Since $f^*\mathscr{A}$ and $\Delta_{\mathscr{A}/\mathscr{C}}$ are locally free of finite rank, the sequence $\mathscr{C} \to f^*\mathscr{A} \to \Delta_{\mathscr{A}/\mathscr{C}}$ admits a local splitting. It follows $\mathscr{C}$ is locally finitely generated. Checking at points, $\mathscr{C}$ has constant rank, so $\mathscr{C}$ is locally free of finite rank. Consider the diagram
\[
\begin{tikzcd}
  & & \mathscr{C} \ar[d,hook] \ar[rd, "\phi"] & & \\
  0 \ar[r] & f^*\mathscr{B} \ar[r] & f^*\mathscr{A} \ar[r] & \Delta_{\mathscr{A}/\mathscr{B}} \ar[r] & 0
\end{tikzcd}
\]
The composite $\mathscr{\phi}$ is a map of $\OO_T$-modules of finite rank, so it has a well-defined vanishing locus $Z$ in $T$. Now, if $g : W \to T$ is a morphism of $S$-schemes, then $g^*\mathscr{C}$ factors through $(g\circ f)^*\mathscr{B}$ if and only if $g^*\phi$ vanishes, if and only if $g$ factors through $Z$.
\end{proof}

When $\mathscr{A}$ is locally free of finite rank, $F^{\delta}_{\mathscr{A}}$ admits a representing scheme.

\begin{definition}
Suppose that $\mathscr{A}$ is locally free of fixed rank $n$. The \emphbf{Grassmannian} is the functor
$\Grass(d, \mathscr{A}) : (\Sch/S)^{op} \to \Set$ defined by:
\begin{enumerate}
  \item If $f : T \to S$ is an $S$-scheme, $\Grass(d, \mathscr{A})(T)$ is the set of quasicoherent subsheaves $\mathscr{E}$ of $\mathscr{A}$ such that the quotient $\mathscr{A} / \mathscr{E}$ is locally free of rank $n - d$.
  \item If $g : T \to T'$ is a morphism of $S$-schemes, the induced map $\Grass(d, \mathscr{A})(T') \to \Grass(d,\mathscr{A})(T)$ is induced by pullback.
\end{enumerate}
\end{definition}

It is well-known that $\Grass(d, \mathscr{A})$ is representable by a locally projective scheme over $S$. (For example, one may apply Zariski descent to \cite[Tag 089R]{stacks-project}.)

\begin{theorem} \label{thm:territory_representable}
Suppose $\mathscr{A}$ is locally free of fixed rank $n$ over $S$. Then the functor on $S$-schemes $F^{\delta}_{\mathscr{A}}$ is representable by a closed subscheme $\Ter^\delta_{\mathscr{A}}$ of $\Grass(n - \delta, \mathscr{A})$.
\end{theorem}

\begin{proof}
Let $G = \Grass(n - \delta, \mathscr{A})$. Let $\mathscr{U} \in G(G)$ be the universal subsheaf and let $Q = \mathscr{A}_G / \mathscr{U}$. The multiplication map $\mathscr{A}_G \otimes_{\OO_G} \mathscr{A}_G \to \mathscr{A}_G$ yields a multiplication $\mathscr{U} \otimes_{\OO_G} \mathscr{U} \to \mathscr{A}_G$. Compose with the quotient map to obtain a morphism $\alpha : \mathscr{U} \otimes_{\OO_G} \mathscr{U} \to Q$.

Since $\mathscr{A}_G$ and $Q$ are locally free of ranks $n$ and $\delta$, $\mathscr{U}$ is locally free of rank $n - \delta$. Since $\mathscr{U}$ and $Q$ are locally free of finite rank, we get a well-defined vanishing locus $V(\alpha)$ of $\alpha$. This locus has the universal property that $T \to G$ factors through $V(\alpha)$ if and only if $\alpha|_T = 0$ if and only if $\mathscr{U}|_T$ is a multiplicatively closed subsheaf of $\mathscr{A}|_T$.

Next, let $\beta : \OO_G \to Q$ be the composite of the multiplication by $1$ map $\OO_G \to \mathscr{A}_G$ with the canonical projection $\mathscr{A}_G \to Q$. There is a well-defined vanishing locus $V(\beta)$ with the universal property that $T \to G$ factors through $V(\beta)$ if and only if $\mathscr{U}|_T$ contains $\OO_T \cdot 1$. Altogether, $V(\alpha) \cap V(\beta)$ cuts out the $\OO_T$-subalgebras of $\mathscr{A}|_T$, as desired.
\end{proof}

Of course, the algebras $\mathscr{A}$ we want to consider --- the structure sheaves of smooth curves over $S$ --- are not locally free of finite rank.
We can remedy this by passing to finite locally free quotients of $\mathscr{A}$. By considering lots of such quotients, we can approximate 
a ``territory of a completion.'' Let's begin by setting definitions.

\begin{definition} \label{def:completed_territory}
If $\mathscr{I}$ is a quasicoherent sheaf of ideals in $\mathscr{A}$ such that $\mathscr{A}/\mathscr{I}$ is flat over $S$, set $F^{\delta}_{(\mathscr{A},\mathscr{I})}$ to be the subfunctor of $F^{\delta}_{\mathscr{A}}$ where $\mathscr{B}$ contains $\mathscr{I}_T$. If $\mathscr{A}/\mathscr{I}^r$ is locally free for all $r$, then denote by $F^{\delta}_{(\mathscr{A},\mathscr{I}^\infty)}$ the sheafification of the union $\bigcup_{r} F^{\delta}_{(\mathscr{A},\mathscr{I}^r)}$ in the Zariski topology.
\end{definition}

\begin{theorem} \label{thm:contains_I_vs_subalg_of_A/I}
Let $\mathscr{I}$ be a quasicoherent sheaf of ideals in $\mathscr{A}$ such that $\mathscr{A}/\mathscr{I}$ is flat over $S$. Then the
functors $F^{\delta}_{\mathscr{A}/\mathscr{I}}$ and $F^{\delta}_{(\mathscr{A},\mathscr{I})}$ are isomorphic via the natural transformation on $T \overset{f}{\to} S$-points
\[
  \mathscr{C} \subseteq f^*(\mathscr{A}/\mathscr{I}) \quad \mapsto \quad \pi_T^{-1}(\mathscr{C}) \subseteq f^*\mathscr{A},
\]
where $\pi_T : f^*\mathscr{A} \to f^*(\mathscr{A}/\mathscr{I})$ is the pullback of the canonical projection $\pi : \mathscr{A} \to \mathscr{A}/\mathscr{I}$ to $T$.
\end{theorem}

\begin{proof}
Let $\pi_T: \mathscr{A}_T \rightarrow (\mathscr{A/I})_T$ be the pullback of the projection $\pi: \mathscr{A} \rightarrow \mathscr{A/I}$. Given $\mathscr{C} \in F^\delta_{\mathscr{A/I}}(T)$, consider the $\OO_T$-subalgebra $\pi_T^{-1}(\mathscr{C}) \subset \mathscr{A}_T$. Clearly $\pi_T^{-1}(\mathscr{C})$ contains $\ker \pi_T = \pi_T^{-1}(0)$.

Since $\mathscr{A/I}$ is flat over $S$, $(\mathscr{A/I})_T$ is flat over $T$, and $(\mathscr{A/I})_T/\mathscr{C}$, being a locally free $\OO_T$-module of rank $\delta$, is also flat. Hence $\mathscr{C}$ is flat. The flatness of $\mathscr{A/I}$ implies $\ker \pi_T = \mathscr{I}_T$.

Now given a family $\mathscr{C} \in F^\delta_{\mathscr{A/I}}(T)$, consider $\pi_T^{-1}(\mathscr{C}) \subset \mathscr{A}_T$. Clearly this is an $\OO_T$-subalgebra containing $\ker \pi_T$. Consider the following diagram:
\[
\begin{tikzcd}
& \mathscr{I}_T \ar[r,equals] \ar[d] & \mathscr{I}_T \ar[d] & & \\
0 \ar[r] & \pi_T^{-1}(\mathscr{C}) \ar[r] \ar[d] & \mathscr{A}_T \ar[r] \ar[d,"\pi_T"] & \mathscr{A}_T/\pi_T^{-1}(\mathscr{C}) \ar[r] \ar[d,"\varphi"] & 0 \\
0 \ar[r] & \mathscr{C} \ar[r] \ar[d] & (\mathscr{A/I})_T \ar[r] \ar[d] & (\mathscr{A/I})_T/\mathscr{C} \ar[r] & 0 \\
& 0 & 0 & &
\end{tikzcd}
\]
By the snake lemma, $\ker \varphi = 0$, and by the commutativity of the rightmost square, $\varphi$ is surjective, so $\mathscr{A}_T/\pi_T^{-1}(\mathscr{C})$ is a locally free $\OO_T$-module of rank $\delta$. Since $\pi_T^{-1}(\mathscr{C})$ contains $\mathscr{I}_T$, we have $\pi_T^{-1}(\mathscr{C}) \in F^{\delta}_{(\mathscr{A},\mathscr{I})}(T)$.

We then define functions $\eta_T: F^\delta_{\mathscr{A/I}}(T) \rightarrow F^{\delta}_{(\mathscr{A},\mathscr{I})}(T)$ for each $S$-scheme $T$, given by $\eta_T(\mathscr{C}) = \pi_T^{-1}(\mathscr{C})$. Note that $\eta_T$ is a bijection, derived from the one-to-one correspondence between subrings of $\mathscr{A/I}$ and subrings of $\mathscr{A}$ containing $\mathscr{I}$; the corank of the quotient $\OO_T$-module is preserved by the snake lemma above. If 
\[
\begin{tikzcd}
  T' \ar[rr,"h"] \ar[dr,"f"'] & & T \ar[dl,"f'"] \\
  & S
\end{tikzcd}
\]
is a morphism of $S$-schemes, it is clear that $h^*(\pi_T^{-1}(\mathscr{C})) = \pi_{T'}^{-1}(h^*\mathscr{C})$, so that the collection of maps $\eta_T$ indeed define a natural isomorphism between these functors. 
\end{proof}

\begin{corollary} \label{cor:subalg_with_supp_representable}
If $\mathscr{I}$ is a sheaf of ideals of $\mathscr{A}$ such that $\mathscr{A}/\mathscr{I}$ is locally free of finite rank, then $F^{\delta}_{(\mathscr{A},\mathscr{I})}$ is representable in the category of $S$-schemes by the $S$-scheme $\Ter^\delta_{\mathscr{A}/\mathscr{I}}$.
\end{corollary}

Our next proposition, generalizing \cite[Proposition 2]{ishii_moduli_subrings}, shows that the functor $F^{\delta}_{(\mathscr{A},\mathscr{I}^\infty)}$ is representable after restricting to \emphbf{reduced} $S$-schemes, and in this sense there is a well-behaved territory of a completion. However, $F^{\delta}_{(\mathscr{A},\mathscr{I}^\infty)}$ is poorly behaved on non-reduced schemes, see \cite[remark (ii)]{ishii_moduli_subrings}. Intuitively, many of the $S$-points of $F^{\delta}_{(\mathscr{A},\mathscr{I}^\infty)}$ for $S$ non-reduced come not from singularities varying, but from singularities moving around or splitting apart infinitesimally, which becomes easier to do as the power on $\mathscr{I}$ gets larger. Ishii's remark (ii) suggests that representing all of these possibilities is impossible with a finite type scheme. Rather than giving up representability over general schemes, we work for the rest of the paper with
more carefully chosen ideals $\mathscr{I}$, preventing singularities from moving or splitting infinitesimally and allowing us to use Theorem \ref{thm:territory_representable} for representability. The following proposition is included only for completeness.

\begin{proposition}
Let $\mathscr{I}$ be a quasicoherent sheaf of ideals in $\mathscr{A}$ such that $\mathscr{A} / \mathscr{I}^r$ is locally free of finite rank for all $r$. Then the restriction of the functor $F^{\delta}_{(\mathscr{A},\mathscr{I}^\infty)}$ to reduced $S$-schemes is represented by the $S$-scheme $(\Ter^{\delta}_{\mathscr{A}/\mathscr{I}^{2\delta}})_{red}$.
\end{proposition}

\begin{proof}
It is a general fact that for an $S$-scheme $X$ the restriction of the functor $\Hom(-, X)$ to reduced $S$-schemes is representable by the reduced scheme $X_{red}$. Then by Corollary \ref{cor:subalg_with_supp_representable} the scheme $(\Ter^{\delta}_{\mathscr{A}/\mathscr{I}^{2\delta}})_{red}$ represents $F^{\delta}_{(\mathscr{A},\mathscr{I}^{2\delta})}$ on the category of reduced $S$-schemes. This is a subfunctor of $F^{\delta}_{(\mathscr{A}\mathscr{I}^\infty)}$, so it suffices to show for all reduced schemes $T \overset{f}{\to} S$ that
$F^{\delta}_{(\mathscr{A},\mathscr{I}^{2\delta})}(T \overset{f}{\to} S) \supseteq F^{\delta}_{(\mathscr{A},\mathscr{I}^\infty)}(T \overset{f}{\to} S)$.

Let $\mathscr{B} \in F^{\delta}_{(\mathscr{A},\mathscr{I}^\infty)}(T \overset{f}{\to} S)$. It suffices to show $\mathscr{B} \in F^{\delta}_{(\mathscr{A},\mathscr{I}^{2\delta})}$ affine locally on $T$, so we may assume that $\mathscr{B} \in F^{\delta}_{(\mathscr{A},\mathscr{I}^r)}(T)$ for some $r$ where $T = \Spec R$ is affine, $\mathscr{A}|_T \cong \widetilde{A}$ for an $R$-algebra $A$, $\mathscr{I}^{2\delta} = \widetilde{J}$ where $J$ is an ideal of $A$, and $\mathscr{A}|_T/\mathscr{B} \cong \widetilde{R^\delta}$ is free as an $\OO_T$-module. We know by \cite[Proposition 2]{ishii_moduli_subrings} that after any base change to any algebraically closed field $k$, $\mathscr{B} \otimes_{\OO_T} k$ contains $\mathscr{I}^{2\delta} \otimes_{\OO_S} k$. Equivalently, if $\phi$ is the composite $\phi : J \to A \to R^\delta$, the map $\phi \otimes_R k : J \otimes_R k \to k^\delta$ is zero for all maps $R \to k$ to an algebraically closed field.

Now let $\pp$ be any prime ideal of $R$ and $\ol{k(\pp)}$ an algebraic closure of its residue field. We have a commutative diagram
\[
  \begin{tikzcd}[column sep=large]
    J/\pp J \ar[r, "\phi\, \otimes R/\pp"] \ar[d] & (R/\pp)^\delta \ar[d] \\
    J \otimes_R \ol{k(\pp)} \ar[r, "\phi\, \otimes \ol{k(\pp)}"] & (\ol{k(\pp)})^\delta.
  \end{tikzcd}
\]
The lower horizontal arrow is zero since $\ol{k(\pp)}$ is algebraically closed, and the right vertical arrow is injective. Therefore the top horizontal arrow is also zero.

Next, we have a diagram
\[
  \begin{tikzcd}[column sep=huge]
    J \ar[r, "\phi"] \ar[d] & R^\delta \ar[d] \\
    \prod\limits_{\pp \in \Spec R} J/\pp J \ar[r, "\prod_\pp \phi\, \otimes R/\pp"] & \prod\limits_{\pp \in \Spec R} (R/\pp)^\delta.
  \end{tikzcd}
\]
The lower horizontal arrow is zero by the previous step. The right vertical arrow has kernel $\left(\bigcap\limits_{\pp \in \Spec R} \pp\right)^{\times \delta}$, which vanishes since $R$ is reduced. Then the right vertical arrow is injective, and therefore the top horizontal arrow is zero. We conclude that $\mathscr{I}^{2\delta}|_T \subseteq \mathscr{B}$, as desired.
\end{proof}

It is convenient to introduce a lemma now which will be useful when we discuss territories of curve singularities below.

\begin{lemma} \label{lem:product_of_territory_to_territory_of_product}
Suppose $\delta_1, \ldots, \delta_p$ are positive integers with sum $\delta$ and $\mathscr{A}_1, \ldots, \mathscr{A}_p$ are locally free $\OO_S$-algebras of fixed finite rank. Then the natural transformation
\begin{align*}
  F^{\delta_1}_{\mathscr{A}_1} \times \cdots \times F^{\delta_p}_{\mathscr{A}_p} &\to F^{\delta}_{\mathscr{A}_1 \times \cdots \times \mathscr{A}_p} \\
  (\mathscr{B}_1, \ldots, \mathscr{B}_p) &\mapsto \mathscr{B}_1 \times \cdots \times \mathscr{B}_p
\end{align*}
induces a closed immersion $\Ter^{\delta_1}_{\mathscr{A}_1} \times \cdots \times \Ter^{\delta_p}_{\mathscr{A}_p} \to \Ter^{\delta}_{\mathscr{A}_1 \times \cdots \times \mathscr{A}_p}$.
\end{lemma}
\begin{proof}
It is clear that this is a well-defined natural transformation. The induced map of schemes is proper, since the source and target are proper. It is also a monomorphism since the map on $T$-points $(\mathscr{B}_1, \ldots, \mathscr{B}_p) \mapsto \mathscr{B}_1 \times \cdots \times \mathscr{B}_p$ is injective. The result follows since proper monomorphisms are closed immersions \cite[Tag 04XV]{stacks-project}.
\end{proof}

\subsection{Conductors and territories of fixed conductance}
\label{ssec:territories_conductors}

Now we move towards a better choice of ideal: the conductor ideal. Conductor ideals are important invariants of finite extensions of algebras. For example, the conductor ideal coincides with the relative dualizing sheaf of a normalization map of curves $\nu : \tilde{C} \to C$ and the dimension of the quotient by the conductor can be used to test for Gorenstein curves, see Lemma \ref{lem:conductor_bounds_and_gorenstein}. However conductor ideals have the disadvantage of not commuting with base change in general, see Example \ref{ex:cond_no_commute}. This makes it difficult to write down moduli functors involving the conductor. In a similar context in \cite{fred_thesis}, this difficulty was avoided by keeping track of the conductor ideal only up to radical, as the radical of the conductor does behave well with base change.

We present a different solution in this section: by fixing the rank $c$ of $\mathscr{A}$ modulo the conductor, the conductor ideal itself can be made to commute with base change. Motivated by this, we construct a locally closed stratification of $\Ter^{\delta}_{\mathscr{A}}$ by $c$.

We begin by recalling some facts about conductors.

\begin{definition} \label{def:conductor}
Given a subsheaf of algebras $\mathscr{B} \subseteq \mathscr{A}_T$ over $T$ of corank $\delta$, the \emphbf{conductor ideal of $\mathscr{B}$ inside $\mathscr{A}_T$} is the sheaf of ideals
\[
  \cond_{\mathscr{A}_T/\mathscr{B}} = \AAnn_{\mathscr{B}}(\Delta).
\]

We denote by $\Delta'$ the sheaf of $\mathscr{B}$-modules $\mathscr{B} / \cond_{\mathscr{A}_T/\mathscr{B}}$ and denote its rank by $\delta'$.
\end{definition}

\begin{lemma} \label{lem:affine_conductor_largest_ideal}
The conductor $\cond_{\mathscr{A}_T/\mathscr{B}}$ is an ideal of $\mathscr{A}_T$ and it is characterized by the property that it is
the largest ideal of $\mathscr{A}_T$ also contained in $\mathscr{B}$.
\end{lemma}
\begin{proof}
Omitted.
\end{proof}

In other words, given $\mathscr{B} \in \Ter^{\delta}_{\mathscr{A}}(S)$, the conductor ideal $\cond_{\mathscr{A}_T/\mathscr{B}}$ is the largest sheaf of ideals $\mathscr{I}$ of $\mathscr{A}$ such that
$\mathscr{B} \in F^{\delta}_{(\mathscr{A},\mathscr{I})}(S) \cong \Ter^{\delta}_{\mathscr{A}/\mathscr{I}}.$ When thinking of $\cond_{\mathscr{A}_T/\mathscr{B}}$ as an ideal sheaf of $\mathscr{A}_T$ instead of $\mathscr{B}$,
we may write $\Cond_{\mathscr{A}_T/\mathscr{B}}$ instead of $\cond_{\mathscr{A}_T/\mathscr{B}}.$

\medskip

We give an example to illustrate how the conductor may fail to commute with base change in general.

\begin{example} \label{ex:cond_no_commute}
Let $k$ be a field and $R = k[t]$. Let $A = R[x]/x^2 \times R[y]/y^2$, and let $B$ be the $R$-subalgebra
\[
  B = \{ a(t)(1,1) + b(t)(tx, y) \mid a(t), b(t) \in R \} \subseteq A.
\]
We will show that the conductor $\cond_{A/B}$ is zero, but the conductor of a base change is nonzero.

Suppose $a(t)(1,1) + b(t)(tx,y) \in \cond_{A/B}$. Then in particular
\[
  (a(t)(1,1) + b(t)(tx, y)) \cdot (0, 1) = (0, a(t) + b(t)y) \in B,
\]
which implies $a(t) = b(t) = 0$. Therefore $\cond_{A/B} = 0$.

Now consider the base change $- \otimes_R R/t$. One computes that $B \otimes_R R/t$ is identified with
the $k$-subalgebra $B_0 \coloneqq k(1,1) + k(0,y)$ of $A_0 \coloneqq k[x]/x^2 \times k[y]/y^2 \cong A \otimes_R R/t$.

Consider an arbitrary element $(a+bx, c+dy) \in A_0$. Then
\[
  (0,y) \cdot (a + bx, c + dy) = (0, cy) \in B_0,
\]
so $\cond_{A_0/B_0}$ contains $(0,y)$. It follows
\[
  \cond_{A/B} \otimes_R R/t = 0 \neq \cond_{A_0/B_0}.
\]
\end{example}

Notice that the conductor jumps up in size when specializing to $t = 0.$ In order to control the behavior of the conductor, it is useful to view it as the kernel of a homomorphism.

\begin{lemma}
The conductor is the kernel of the map of $\OO_S$-modules
\[
  \mathscr{B} \to \EEnd_{\OO_S}(\Delta)
\]
taking an element $f \in \mathscr{B}$ to the multiplication by $f$ map $\Delta \to \Delta$.
In particular, the image is isomorphic to $\Delta'$.
\end{lemma}
\begin{proof}
Omitted.
\end{proof}

\begin{definition} \label{def:conductance}
The \emph{conductance at a point $t \in T$} of $\mathscr{B}$ inside $\mathscr{A}_T$ is the dimension of $(\mathscr{A}_T / \Cond_{\mathscr{A}_T /\mathscr{B}}) \otimes_{\OO_T} k(t)$ as a $k(t)$-vector space.

We say that $\mathscr{B}$ has \emph{conductance $c$} inside of $\mathscr{A}_T$ if $\mathscr{A}_T / \Cond_{\mathscr{A}_T / \mathscr{B}}$ is locally
free of rank $c$.
\end{definition}

\begin{definition} \label{def:ter_delta_c}
Let $\mathscr{A}$ be a locally free sheaf of $\OO_S$-algebras of rank $n$. 
Let $F^{\delta,c}_{\mathscr{A}}$ be the subfunctor of $F^{\delta}_{\mathscr{A}}$ determined by the condition
that $\mathscr{B} \in F^{\delta}_{\mathscr{A}}(T)$ belongs to $F^{\delta,c}_{\mathscr{A}}(T)$ if and only if
the cokernel of the map
\[
  \mathscr{B} \to \EEnd_{\OO_T}(\Delta)
\]
is a locally free $\OO_T$-module of rank $\delta^2 - c + \delta$.
\end{definition}

\begin{remark}
The subfunctor is well-defined since formation of the map
  \[
    \mathscr{B} \to \EEnd_{\OO_T}(\Delta)
  \]
commutes with pullback and cokernels are preserved by tensor product.
\end{remark}

As promised, we deduce the following essential property.

\begin{proposition} \label{prop:conductor_commutes_base_change}
Suppose $\mathscr{B} \in F^{\delta,c}_{\mathscr{A}}(T \to S)$. Then the formation of $\Delta'$ and $\cond_{\mathscr{A}_T/\mathscr{B}}$ commute with pullback in $T$, $\Delta'$ is locally free of rank $\delta' = c - \delta$ over $T$, and $\mathscr{B}$ has conductance $c$ inside $\mathscr{A}_T$.
\end{proposition}
\begin{proof}
We have a short exact sequence on $T$
\[
  0 \to \Delta' \to \EEnd_{\OO_T}(\Delta) \to Q \to 0.
\]
By assumption, $\EEnd_{\OO_T}(\Delta)$ is locally free of rank $\delta^2$ and $Q$ is locally free of rank $\delta^2 - c + \delta$.
It follows $\Delta'$ is locally free of rank $\delta'$ and $\mathscr{A}_T / \Cond_{\mathscr{A}_T / \mathscr{B}}$ is locally free of rank $c$.

Now consider a $S$-morphism $f : T' \to T$. Since $Q$ is $T$-flat,
\[
  0 \to f^*(\Delta') \to f^*\EEnd_{\OO_T}(\Delta) \to f^*Q \to 0
\]
is exact. Moreover, since formation of the map $\mathscr{B} \to \EEnd_{\OO_S}(\Delta)$ commutes with pullback, we have a commutative diagram with exact rows
\[
\begin{tikzcd}
  0 \ar[r] & \cond_{\mathscr{A}_{T'}/\mathscr{B}_{T'}} \ar[r] \ar[d] & \mathscr{B}_{T'} \ar[r] \ar[d] & \EEnd_{\OO_{T'}}(f^*\Delta) \ar[r] \ar[dash, d, "\sim"] & f^*Q \ar[r] \ar[d, equal] & 0 \\
  & 0 \ar[r] & f^*(\Delta') \ar[r] & f^*\EEnd_{\OO_T}(\Delta) \ar[r] & f^*Q \ar[r] & 0.
\end{tikzcd}
\]
It follows that $f^*(\Delta') \cong \mathscr{B}_{T'} / \cond_{\mathscr{A}_{T'}/\mathscr{B}_{T'}}$, so formation of $\Delta'$ commutes with base change. Moreover, since $\Delta'$ is $T$-flat and the map
$\mathscr{B}_{T'} \to f^*(\Delta')$ is the pullback of $\mathscr{B} \to \Delta'$, we have $\cond_{\mathscr{A}_{T'}/\mathscr{B}_{T'}} \cong f^*\cond_{\mathscr{A}_T/\mathscr{B}}$, so formation of the conductor
commutes with base change as well.
\end{proof}

\begin{theorem} \label{thm:delta_prime_stratification_affine}
Suppose $\mathscr{A}$ is locally free of rank $n$. Then the functors $F^{\delta,c}_{\mathscr{A}}$
are represented by disjoint, locally closed subschemes $\Ter^{\delta,c}_{\mathscr{A}}$ of $\Ter^{\delta}_{\mathscr{A}}$, with union
$\Ter^{\delta}_{\mathscr{A}}$.
\end{theorem}
\begin{proof}
Let $\mathscr{U} \subseteq \mathscr{A}_{\Ter^{\delta}_{\mathscr{A}}}$ be the universal subalgebra.
The subschemes $\Ter^{\delta,c}_{\mathscr{A}}$ are precisely the strata induced by the Fitting ideals of
$\EEnd_{\Ter^\delta_{\mathscr{A}}}(\Delta)/\mathscr{U}$, see \cite[Tag 05P8]{stacks-project}.
\end{proof}

\begin{example} \label{ex:genus_two_unibranch_territory}
Consider the territory $\Ter^2_{\QQ[t]/(t^4)}$. We may think of this as parametrizing the ways of inserting a genus 2 unibranch singularity on $\AA^1$ such that the conductor of the singularity contains $(t^4)$. We compute in Section \ref{ssec:territory_computations} that
\[
  \Ter^2_{\QQ[t]/(t^4)} \cong \Proj \QQ[a,b,c]/(a^3, 2a^2b, a^2c - 2ab^2),
\]
with reduction $\PP^1 = V(a)$. An $S$-point corresponds to the algebra $\OO_S \cdot 1 + \OO_S \cdot (at + bt^2 + ct^3) \subseteq \OO_S[t]/(t^4)$.

There are only two non-empty strata: $\Ter^{2,4}_{\QQ[t]/t^4}$, which is open, and $\Ter^{2,3}_{\QQ[t]/t^4}$, which is closed.
If $k$ is a field over $\QQ$, then a $k$-point $B = k \cdot 1 + k \cdot (bt^2 + ct^3) \subseteq k[t]/t^4$ factors through $\Ter^{2,3}_{\QQ[t]/(t^4)}$ if and only if $B$ contains $t^3$, i.e.,
$b = 0$. Thus, the open stratum $\Ter^{2,4}_{\QQ[t]/(t^4)}$ is the complement $D(b) \cong \AA^1$ of the point at infinity. It parametrizes various normalizations
of the ramphoid cusp, a Gorenstein singularity.

We perform a local computation in $U \coloneqq D(c) = \Spec \QQ[a, b]/(a^3,2a^2b, a^2 - 2ab^2)$ in order to find the scheme structure on $\Ter^{2,3}_{\QQ[t]/(t^4)}$. Let $B = \OO_U \cdot 1 + \OO_U \cdot (at + bt^2 + t^3)$ be the universal subalgebra of $\OO_U[t]/t^4$ over $U$. We need to compute the cokernel $Q$ of $B \to \EEnd_{\OO_{U}}(\Delta)$ then find the Fitting ideal cutting out the rank $\delta^2 - c + \delta = 3$ locus. Note that $\Delta$ has a basis given by the equivalence classes $\ol{t}$ and $\ol{t^2}$ of $t$ and $t^2$ modulo $B$. Multiplication by 1 acts as the identity on $\Delta$ (of course)
while the action of $at + bt^2 + t^3$ is given by
\begin{align*}
  (at + bt^2 + t^3) \cdot \ol{t} &= \ol{at^2 + bt^3} = -ab\ol{t} + (a - b^2)\ol{t^2} \\
  (at + bt^2 + t^3) \cdot \ol{t^2} &= \ol{at^3} = -a^2 \ol{t} - ab \ol{t^2}.
\end{align*}
Using this, we compute
\begin{align*}
  Q &\cong \frac{\mathrm{Mat}_{2 \times 2}(\OO_U)}{\OO_U \begin{bmatrix} 1 & 0 \\ 0 & 1\end{bmatrix} + \OO_U \begin{bmatrix} -ab & -a^2 \\ a - b^2 & -ab \end{bmatrix}} \\
   &\cong \OO_U^3 / \OO_U \cdot (0, -a^2, a - b^2).
\end{align*}
The appropriate Fitting ideal is generated by the entries of the vector $(0,-a^2,a - b^2)$, namely $a^2$ and $a - b^2$. Adding these to the defining ideal $(a^3, 2a^2b, a^2 - 2ab^2)$ of $U$, we obtain
\[
  (a^3, 2a^2b, a^2 - 2ab^2, a^2, a - b^2) = (a^2, a - b^2, ab^2) = (a - b^2, b^4).
\]
Setting $\epsilon = b$ (so that $a = \epsilon^2$), we conclude that
\[
  \Ter^{2,3}_{\QQ[t]/t^4} \cong \QQ[\epsilon]/(\epsilon^4).
\]
The universal subalgebra of $\QQ[\epsilon, t]/(\epsilon^4, t^4)$ of conductance 3 is spanned by $1$ and $\epsilon^2t + \epsilon t^2 + t^3$. 
The conductor ideal of the universal subalgebra is the ideal of $\QQ[\epsilon, t]/(\epsilon^4, t^4)$ generated by $(\epsilon^3 + \epsilon^2t + \epsilon t^2 + t^3)$. The unique $k$-point of $\Ter^{2,3}_{\QQ[t]/(t^4)}$, for $k$ an algebraically closed field, corresponds to the unique non-Gorenstein unibranch singularity of genus 2, $k[t^3,t^4,t^5]$.
\end{example}

\section{Territories for curve singularities}
\label{sec:territories_singularities}

We now turn our attention to the case of curve singularities. We will construct a scheme $\Tersing(g,\vec{c})$ parametrizing curve singularities with fixed invariants within their seminormalization
and $\Tergl(\mathcal{P},g,\vec{c})$ parametrizing several curve singularities with fixed invariants within their normalization. We begin by introducing some basic invariants of curve singularities.

\begin{definition} \label{def:sing_basics}
Let $C$ be a reduced algebraic curve over an algebraically closed field $k$, $p$ a singularity of $C$, and $\nu : \tilde{C} \to C$ the normalization of $C$ at $p$. The \emphbf{$\delta$-invariant} of $p$ is
\[
  \delta(p) = \dim_k(\nu_*\OO_{\tilde{C}} / \OO_C).
\]

The \emphbf{number of branches} of $p$ is
\[
  m(p) = \# \nu^{-1}(p).
\]

The \emphbf{genus} of $p$ is
\[
  g(p) = \delta(p) - (m(p) - 1).
\]
This can be interpreted as the number of conditions for a function on $\tilde{C}$ to descend to one on $C$ apart from the $m(p) - 1$ ``obvious'' relations $f(p_1) = f(p_2) = \cdots = f(p_{m(p)})$, where $\nu^{-1}(p) = \{ p_1, \ldots, p_{m(p)} \}$.
More abstractly, this number can be interpreted as the ``contribution of $p$ to the arithmetic genus of $C$.''

The \emphbf{conductor} of $p$ is the $\OO_{C}$-ideal sheaf
\[
  \cond(p) \coloneqq \AAnn_{\OO_C}(\Delta_p).
\]
It is not hard to check that $\cond(p)$ is also closed under multiplication by $\nu_*\OO_{\tilde{C}}$, so by pulling back appropriately to $\tilde{C}$, we have a sheaf of $\OO_{\tilde{C}}$ ideals
\[
  \Cond(p) \coloneqq \nu^{-1}\cond(p) \otimes_{\nu^{-1}\nu_*\OO_{\tilde{C}}} \OO_{\tilde{C}}.
\]
satisfying $\nu_*\Cond(p) = \cond(p)$. We will also call $\Cond(p)$ the \emphbf{conductor of $p$} and rely on context to distinguish between $\cond(p)$ and $\Cond(p)$.

The \emphbf{conductance} of $p$ is the quantity
\[
  c(p) = \dim_k(\OO_{\tilde{C}} / \Cond(p)).
\]
The \emphbf{branch conductance} $c(b)$ at a point $b \in \nu^{-1}(p)$ is the dimension of the stalk of $\OO_{\tilde{C}} / \Cond(p)$ at $b \in \nu^{-1}(p)$. The branch conductances and conductance satisfy
the obvious relation $c(p) = \sum_{b \in \nu^{-1}(p)} c(b)$.
\end{definition}

Conductor ideals have special significance for reduced curves: it is well-known that for a normalized curve $\nu : \tilde{C} \to C$, the conductor $\Cond_{\tilde{C}/C}$ is isomorphic to the relative dualizing sheaf
$\omega_{\nu}$. The following lemma suggests that we may regard the invariant $c$ as a measure of how close $C$ is to being Gorenstein at $p$.

\begin{lemma} \label{lem:conductor_bounds_and_gorenstein}
Let $C$ be a proper, reduced curve over an algebraically closed field $k$. Let $p$ be a singularity of $C$. Then $\delta(p) < c(p) \leq 2\delta(p),$
and $C$ is Gorenstein at $p$ if and only if $c(p) = 2\delta(p).$
\end{lemma}
\begin{proof}
This is \cite[Chapter VIII, Proposition 1.16]{altman_kleiman}.
\end{proof}

Observe that $\Gamma(\tilde{C}, \OO_{\tilde{C}} / \Cond(p)) \cong \prod_{i = 1}^m k[t_i]/ (t_i^{c_i})$, where $c_1, \ldots, c_m$ are the branch conductances, and $\Gamma(C, \OO_C / \cond(p))$ is a $\delta(p)$-codimensional subalgebra. In order to parametrize such subalgebras of $\OO_{\tilde{C}} / \Cond(p)$, it suffices to consider the territories of the following algebras.

\begin{definition} \label{def:algebra_for_curve_sing}
Let $c_1,\ldots, c_m$ be positive integers. Write $\vec{c} = (c_1, \ldots, c_m)$.
Let $A_{\vec{c}}$ denote the ring
\[
  A_{\vec{c}} = \prod_{i=1}^m \ZZ[t_i] / (t_i^{c_i}).
\]
Let $A^+_{\vec{c}}$ denote the subring
\[
  A^+_{\vec{c}} \coloneqq \{ (f_1(t_1), \ldots, f_m(t_m)) \in A_{\vec{c}} \mid f_i(0) = f_j(0) \text{ for all }i, j \}
\]
\end{definition}

We view $\Spec A_{\vec{c}}$ as the infinitesimal germ of the normalization of a curve singularity with conductances $c_1, \ldots, c_m$ and $\Spec A^+_{\vec{c}}$ as the infinitesimal germ of its seminormalization.
Indeed, it is not hard to show that $A^+_{\vec{c}}$ is the pushout of the diagram
\[
  \begin{tikzcd}
    \bigsqcup_{i = 1}^m \Spec \ZZ \ar[r] \ar[d] & \Spec A_{\vec{c}} \\
    \Spec \ZZ
  \end{tikzcd}
\]
in which the rightward arrow is the inclusion of the $m$ sections of $\Spec \ZZ[t_i] / (t_i^{c_i}) \to \Spec \ZZ$ induced by $t_i = 0$ for $i = 1, \ldots, m$. That is, fiber by fiber, $\Spec A^+_{\vec{c}}$ is obtained by gluing together the closed points of $\Spec A_{\vec{c}}$.

\medskip

Suppose given a non-negative integer $g$ and positive integers $c_1,\ldots, c_m$ such that $c_1 \geq \cdots \geq c_m$ and $g + m - 1 < c_1 + \cdots + c_m \leq 2(g + m - 1)$. Write $\delta = g + m -1$
and $c = c_1 + \cdots + c_m$.
Now, by Lemma \ref{lem:subalgebra_of_subalgebra}, since $A^+_{\vec{c}}$ is a corank $m - 1$ subalgebra of $A_{\vec{c}}$, we have a natural closed immersion
\[
  \Ter^g_{A^+_{\vec{c}}} \subseteq \Ter^{\delta}_{A_{\vec{c}}}.
\]
On the other hand, we have an immersion
\[
  \Ter^{\delta,c}_{A_{\vec{c}}} \subseteq \Ter^{\delta}_{A_{\vec{c}}}.
\]
This is an open immersion: since $c$ is equal to the rank of $A_{\vec{c}}$, it is cut out by the minimal rank locus of $\mathrm{coker} \{ \mathscr{B} \to \Hom_{A_{\vec{c}}}(\Delta,\Delta) \}$, see Theorem \ref{thm:delta_prime_stratification_affine}.

\begin{definition} \label{def:territory_of_singularity}
With notation as in the paragraph above, we define
the \emphbf{territory $\Tersing(g,\vec{c})$ of singularities of type $(g,\vec{c})$} to be the scheme
\[
  \Tersing(g,\vec{c}) \coloneqq \Ter^{\delta, c}_{A_{\vec{c}}} \cap \Ter^{g}_{A^{+}_{\vec{c}}}.
\]
\end{definition}

Among all of the subalgebras parametrized by $\Ter^{\delta, c}_{A_{\vec{c}}}$, the subscheme $\Tersing(g,\vec{c})$ parametrizes those in which all points $t_i = 0$ of $\Spec A_{\vec{c}}$ are glued together. Hence, $\Tersing(g,\vec{c})$ parametrizes the curve singularities with $m$ branches, genus $g$, and conductances $\vec{c}$. We remark that for any $T$-point $B \in \Tersing(g,\vec{c})(T)$, the conductor ideal $\cond_{A_{\vec{c}}|_T / B}$ is fixed: it is always the zero ideal.

\begin{remark}
Surprisingly, fixing the conductance within $A^+_{\vec{c}}$ is not equivalent to fixing it within $A_{\vec{c}}$. For example, let $S = \Spec \CC[x]$ and consider the family of subalgebras of $A^+_{(2,2)}$ over $S$ given by
\[
  \mathscr{B} = \OO_S \cdot 1 + \OO_S \cdot (xt_1,t_2).
\]
For all $x \in \CC$, the algebra $\CC \cdot 1 + \CC \cdot (xt_1, t_2)$ has conductor ideal $( (xt_1, t_2) )$ as a subalgebra of $A^+_{(2,2)} \otimes \CC$. Thus, the conductance within $A^+_{(2,2)}$ is always $2$. In contrast, the conductance as a subalgebra of $A_{(2,2)}$ varies: for $x \in \CC - \{ 0 \}$,
the algebra $\CC \cdot 1 + \CC \cdot (x t_1, t_2)$ has conductor ideal $0$ in $A_{(2,2)}$, while for $x = 0$ it has conductor ideal $((0,t_2))$. That is, the conductance within $A_{(2,2)}$ is generically 4, but drops to 3 at $x = 0$.

Fixing conductance within $A_{\vec{c}}$ is the correct choice for our purposes since we will be working with conductor ideals of normalizations below.
\end{remark}

\begin{remark} We comment on a subtle point related to the assumption that $c = \mathrm{rank}(A_{\vec{c}})$.
If $\vec{c}$ and $\vec{d}$ are vectors of conductances such that $c_i \leq d_i$ for all $i$, there is a natural inclusion $\Ter^{\delta,c}_{A_{\vec{c}}} \to \Ter^{\delta,c}_{A_{\vec{d}}}$ induced by taking preimages of subalgebras along the quotient map $A_{\vec{d}} \to A_{\vec{c}}$. This map is in general a nilpotent thickening. For instance, we saw in Example \ref{ex:genus_two_unibranch_territory} that $\Ter^{2,3}_{\ZZ[t]/t^4} \otimes \QQ \cong \Spec \QQ[\epsilon]/\epsilon^4$. On the other hand, $\Ter^{2,3}_{\ZZ[t]/t^3} \otimes \QQ \cong \Spec \QQ$, since its only $S$-point is the subalgebra $\OO_S \cdot 1$ of $\OO_S[t]/t^3$. Intuitively, the conductor ideal associated to points of the larger territory is allowed to vary within $V(t^4)$, so the parametrized singularities can move around or begin to smooth out infinitesimally. In the smaller territory, the conductor ideal is restricted to be the zero ideal. The appropriate territory of curve singularities is $\Tersing(2,(3)) = \Ter^{2,3}_{\ZZ[t]/t^3}$, the smaller territory.
\end{remark}

\bigskip

We will spend the rest of this section showing that, modulo an infinitesimal thickening, $\Ter^{\delta,c}_{A_{\vec{c}}}$ breaks up into a disjoint union of products of territories of singularities coming from the various ways to glue points together into singularities. We first introduce gluing data, combinatorial data that track which points of $A_{\vec{c}}$ are glued together and what the 
genera of the resulting singularities are.

\begin{definition}
If $\vec{c} = (c_1, \ldots, c_m)$ is a tuple of positive integers and $P = \{i_1, \ldots, i_p\}$ is a subset of $\{ 1, \ldots, m \}$, with $i_1 < \cdots < i_p$, denote by $\vec{c}|_P$ the tuple $(c_{i_1}, \ldots, c_{i_p})$.
\end{definition}

\begin{definition} \label{def:gluing_data}
An $m$-branch \emphbf{gluing datum} of corank $\delta$ is a tuple $(\calP, g)$ where
\begin{enumerate}
  \item $\calP$ is a partition of the set $\{1, \ldots, m \}$
  \item $g : \calP \to \ZZ_{\geq 0}$ is a function taking each part $P$ of $\calP$ to its genus $g(P)$
  \item $\sum_{P \in \calP} (g(P) + |P| - 1) = \delta$.
\end{enumerate}
\end{definition}

\bigskip

Let us introduce notation for the $k$-points of $\Spec A_{\vec{c}}$.

\begin{definition}
Let $k$ be a field, $\vec{c} = (c_1, \ldots, c_m)$ a tuple of positive integers. Write $A_{\vec{c},k}$ for the ring $A_{\vec{c}} \otimes k$. For each $i = 1, \ldots, m$, we denote by $\qq_i$ the prime ideal of $A_{\vec{c},k}$ generated by $(1, \ldots,1, t_i, 1, \ldots, 1)$. 
\end{definition}

In particular, for $A_{(1,\ldots,1),k} \cong k^m$, $\qq_i$ is the prime ideal generated by the tuple $(1, \ldots, 1, 0, 1, \ldots, 1)$ whose $i$th entry is 0.
Next, we prove a lemma on the subalgebras of $k^m$.

\begin{lemma} \label{lem:subalges_of_kn}
Let $m$ be a positive integer. Given a subset $S \subseteq \{ 1, \ldots, m \}$, write $1_S$ for the element of $k^m$ whose $i$th entry is 1 if $i \in S$ and 0 otherwise.
Then
\begin{enumerate}
  \item The $k$-subalgebras of $k^m$ are precisely the subsets of the form
\[
  B_{\mathcal{P}} = \Span \{ 1_P \mid P \in \mathcal{P} \}
\]
for some set partition $\mathcal{P}$ of $\{1, \ldots, m \}$.

  \item If $\mathcal{P}$ is a partition of $\{1, \ldots, m \},$ the parts $P$ of $\mathcal{P}$ are in bijection with the points of $\Spec B_{\mathcal{P}}$ via
  \[
     P \mapsto \mathfrak{q}_P.
  \]
  where $\mathfrak{q}_P$ is the prime ideal $1_{P^c} \cdot B_{\mathcal{P}}$ of $B_{\mathcal{P}}$.

  \item For each partition $\mathcal{P}$ of $\{1, \ldots, m\}$, the map
  \[
    \Spec k^m = \{ \mathfrak{q}_{1}, \ldots, \mathfrak{q}_{m} \} \to \Spec B_{\mathcal{P}} = \{ \mathfrak{q}_P \mid P \in \mathcal{P} \}
  \]
  induced by the inclusion $B_{\mathcal{P}} \to k^m$ sends $\mathfrak{q}_{i}$ to $\mathfrak{q}_P$ where $P$ is the part of the partition containing $i$.
\end{enumerate} 
\end{lemma}
\begin{proof}
We prove (i) and leave the proofs of (ii) and (iii) as exercises for the interested reader. It is clear that each $B_{\mathcal{P}}$ is a $k$-subalgebra of $k^m$. Conversely, suppose $B$ is a subalgebra of $k^m$. Say that a nonempty subset $S \subseteq \{1, \ldots, m \}$ is a \textit{level set} of $B$ if there exist elements $b = (b_1, \ldots, b_m) \in B, x \in k$ such that $S = \{ i \mid b_i = x \}$. Let $\mathcal{S}$ be the set of level sets of $B$.

We claim that if $S \in \mathcal{S}$, then $1_S \in B$. To see this, let $b \in B$ and write
\[
  b = x_1 1_{S_1} + x_2 1_{S_2} + \cdots + x_l 1_{S_l}
\]
where $x_1, x_2, \ldots, x_l$ are the distinct values of $b$. Then for each $i = 1, \ldots, l$,
\[
  1_{S_i} = \frac{\prod_{j \neq i} (b - x_j \cdot (1, \ldots, 1))}{ \prod_{j \neq i} (x_i - x_j)} \in B,
\]
establishing the claim. It follows that $\Span \{ 1_S \mid S \in \mathcal{S} \} \subseteq B$. The other containment is clear, as each element is a linear combination of the characteristic functions of its level sets. Thus, $B$ is the span of its level sets.

Now observe that $\mathcal{S}$ contains $\{1, \ldots, m \}$, since $1 \in B$; $\mathcal{S}$ is closed under intersections, since $1_S \cdot 1_T = 1_{S \cap T}$; and $\mathcal{S}$ is closed under complements, since $1_{S^c} = 1 - 1_S.$ For each $i \in \{1, \ldots, m\}$ let $S_i = \bigcap_{i \in S \in \mathcal{S}} S$ be the intersection of the level sets of $B$ that contain $i$. Since level sets are closed under intersection, each $S_i$ is a level set. I claim that the $S_i$'s form a set partition of $\{1, \ldots, m\}$. By construction, the union of the $S_i$'s is $\{1, \ldots, m\}$. Now suppose $i, j \in \{1, \ldots, m\}$. If $i \not\in S_j$, then $S_j^c$ is a level set containing $i$, so $S_i \subseteq S_j^c$, which implies $S_i \cap S_j = \emptyset$. Symmetrically, if $j \not\in S_i$, $S_i \cap S_j = \emptyset.$ If both $i \in S_j$ and $j \in S_i$, then every level set containing $i$ also contains $j$ and vice versa, whence $S_i = S_j$. Thus, for each $i, j \in \{1, \ldots, m \}$, either $S_i = S_j$ or $S_i \cap S_j = \emptyset.$ Therefore the $S_i$'s form a partition $\mathcal{P}$ of $\{1, \ldots, m \}$.

Clearly, $\Span \{ 1_P \mid P \in \mathcal{P} \} \subseteq \Span \{ 1_S \mid S \in \mathcal{S} \}= B$. For the other containment, suppose $S$ is a level set. Then $S = \bigcup_{i \in S} S_i$, so $1_S = \sum_{P} 1_P$ where the sum is over the $P \in \mathcal{P}$ such that $P \subseteq S$. It follows $\Span \{ 1_P \mid P \in \mathcal{P} \} = B$, as desired.
\end{proof}

\begin{theorem} \label{thm:decomposition_into_territories_of_sings_no_conductance}
There is an infinitesimal thickening
\[
  \coprod_{(\mathcal{P},g)} \left( \prod_{P \in \mathcal{P}} \Ter^{g(P)}_{A^+_{\vec{c}|_P}} \right) \to \Ter^\delta_{A_{\vec{c}}}.
\]
where the disjoint union varies over the $m$-branch gluing data $(\mathcal{P},g)$ of corank $\delta$.

Moreover, if $k$ is a field and $B$ is a $k$-point of the summand indexed by $(\mathcal{P},g)$, then the induced map $\Spec A_{\vec{c},k} \to \Spec B$ identifies points $\mathfrak{q}_i$ and $\mathfrak{q}_j$ if and only if $i, j$ belong to
 the same part of the partition $\mathcal{P}.$
\end{theorem}

\begin{proof}
We begin by constructing the map. Fix a choice of $(\mathcal{P},g)$.
Observe that
\[
  B_{\vec{c},\mathcal{P}} \coloneqq \{ (f_1(t_1), \ldots, f_m(t_m)) \in A_{\vec{c}} \mid P \in \mathcal{P} \text{ and }i,j \in P \implies f_i(0) = f_j(0) \}
\]
is a subalgebra of $A_{\vec{c}}$ of corank $m - |\mathcal{P}|$. Observe that there is an isomorphism
\[
  \phi_{(\mathcal{P},g)} : \prod_{P \in \mathcal{P}} A^+_{\vec{c}|_P} \cong B_{\vec{c},\mathcal{P}}.
\]
Define a closed immersion $\pi_{(\mathcal{P},g)} : \prod_{P \in \mathcal{P}} \Ter^{g(P)}_{A^+_{\vec{c}|_P}} \to \Ter^{\delta}_{A_{\vec{c}}}$ by the composition
\[
  \prod_{P \in \mathcal{P}} \Ter^{g(P)}_{A^+_{\vec{c}|_P}} \to \Ter^{\sum_{P \in \mathcal{P}} g(P)}_{\prod_{P \in \mathcal{P}}A^+_{\vec{c}|_P}} \to \Ter^{\sum_{P \in \mathcal{P}} g(P)}_{B_{\vec{c},\mathcal{P}}} \to \Ter^{\delta}_{A_{\vec{c}}},
\]
where the first map is that of Lemma \ref{lem:product_of_territory_to_territory_of_product}, the second is the isomorphism induced by $\phi_{(\mathcal{P},g)}$, and the third is the map of Lemma \ref{lem:subalgebra_of_subalgebra}.

Now letting $(\mathcal{P},g)$ vary, observe that the images of the various closed immersions $\pi_{(\mathcal{P},g)}$ are disjoint, so the induced map $\pi : \coprod_{(\mathcal{P},g)} \left( \prod_{P \in \mathcal{P}} \Ter^{g(P)}_{A^+_{\vec{c}|_P}} \right) \to \Ter^\delta_{A_{\vec{c}}}$ is also a closed immersion.

In order to show that the map is an infinitesimal thickening, it suffices to show that it induces a surjection on $k$-points for each field $k$.

Suppose $B \subseteq A_{\vec{c},k}$ is a $k$-subalgebra of corank $\delta$. By Lemma \ref{lem:subalges_of_kn}, $B \cap \left(\prod_{i=1}^m k \cdot 1\right)$ is a span of orthogonal idempotents $B_{\mathcal{P}} = \Span \{ 1_P \mid P \in \mathcal{P} \}$ for some partition $\mathcal{P}$ of $\{1, \ldots, m\}$.
Then $B$ decomposes as a direct product $\prod_{P \in \mathcal{P}} B \cdot 1_P$. Observe that $B \cdot 1_P \subseteq A_{\vec{c}} \cdot 1_P$, which is naturally isomorphic to $A_{\vec{c}|_P}$ by renumbering the $t_i$'s in $P$. Because $B \cap \left(\prod_{i=1}^m k \cdot 1\right) = B_{\mathcal{P}}$, the subalgebra $B \cdot 1_P$ factors through $A^+_{\vec{c}|_P}$, with some codimension $g(P)$. This gives us a point of $\Ter^{g(P)}_{A^+_{\vec{c}|_P}}$ for each $P \in \mathcal{P}$. The codimensions of $B \cdot 1_P$ in $A_{\vec{c}} \cdot 1_P$ must sum to the codimension of $B$ in $A_{\vec{c}}$, so $\sum_{P \in \mathcal{P}} (g(P) + |P| - 1) = \delta$. Altogether, we have a partition $\mathcal{P}$ of $\{1, \ldots, m\}$, a function $g : \mathcal{P} \to \ZZ_{\geq 0}$ satisfying $\sum_{P \in \mathcal{P}} (g(P) + |P| - 1) = \delta$, and a point $(B \cdot 1_P)_{P \in \mathcal{P}}$ of $\prod_{P \in \mathcal{P}} \Ter^{g(P)}_{A^+_{\vec{c}|_P}}$. Finally, under the closed immersion, $(B \cdot 1_P)_{P \in \mathcal{P}} \mapsto \prod_{P \in \calP} B \cdot 1_P = B$. We conclude that the closed immersion is a surjection on $k$-points, hence an infinitesimal thickening.

We conclude by showing the claim about identified points. If $B \subseteq A_{\vec{c}}$ belongs to a summand indexed by $(\mathcal{P}, g)$, then $B \cap \left(\prod_{i=1}^m k \cdot 1\right) = B_{\mathcal{P}}$. Therefore the induced map $\Spec A_{\vec{c}} \to \Spec B$ has as its reduction the morphism $\Spec k^m \to \Spec B_{\mathcal{P}}$ induced by the inclusion. The result follows by Lemma \ref{lem:subalges_of_kn}, part (iii).
\end{proof}

\begin{remark} The computation of $\Ter^{2}_{A_{(2,2)}}$ in Section \ref{ssec:territory_computations} shows that the preceding theorem cannot be strengthened to an isomorphism.
\end{remark}

\begin{corollary} \label{cor:decomposition_into_ter_of_sings}
There is an infinitesimal thickening
\[
  \coprod_{(\mathcal{P},g)} \prod_{P \in \calP} \Tersing(g(P),\vec{c}|_P) \to \Ter^{\delta,c}_{A_{\vec{c}}}.
\]
where the disjoint union varies over the $m$-branch gluing data $(\mathcal{P},g)$ of corank $\delta$.
\end{corollary}
\begin{proof}
The domain and codomain are open subschemes of the domain and codomain of Theorem \ref{thm:decomposition_into_territories_of_sings_no_conductance}, so it suffices to check that the $k$-points of the domain map into the $k$-points of the codomain. This is obvious, since a $k$-point $B \subseteq A_{\vec{c},k}$ of $\Ter^{\delta}_{A_{\vec{c}}}$ belongs to $\Ter^{\delta,c}_{A_{\vec{c}}}$ if and only if $\Cond_{A_{\vec{c},k}/B} = 0$ if and only if $B$ contains no $t_i^{c_i - 1}$.
\end{proof}

\begin{definition} \label{def:territory_gluing_datum}
The \emphbf{territory of the gluing datum} $(\calP, g)$, denoted by $\Tergl(\mathcal{P},g,\vec{c})$ is the union of the connected components of $\Ter^{\delta,c}_{A_{\vec{c}}}$ that meet the image of $\prod_{P \in \calP} \Tersing(g(P),\vec{c}|_P)$ under the map of Corollary \ref{cor:decomposition_into_ter_of_sings}.
\end{definition}

Note that $\Ter^{\delta,c}_{A_{\vec{c}}} = \coprod_{(\mathcal{P},g)} \Tergl(\mathcal{P},g,\vec{c})$ is a disjoint union decomposition.

\begin{remark}
Ishii's connectedness theorem \cite[Corollary 2]{ishii_moduli_subrings} implies that $\Ter^{\delta}_{A^+_{\vec{c}}}$ is geometrically connected for any choice of $\vec{c}$. In particular, the decomposition of Theorem \ref{thm:decomposition_into_territories_of_sings_no_conductance} is the decomposition of $\Ter^{\delta}_{A_{\vec{c}}}$ into its connected components. Ishii's proof works by finding a sequence of lines connecting any two points of the territory. These lines may pass through singularities with reduced conductance, so her result does \textit{not} imply connectedness of the spaces $\Tersing(g,\vec{c})$.

In fact, they need not be connected. Consider $\Tersing(3,(6)) = \Ter^{3,6}_{A_{(6)}}$. Computations in Macaulay2 show that $\Ter^3_{A_{(6)}}$ consists of two irreducible components, $Z_1$ and $Z_2$.
The component $Z_1$ generically parametrizes Gorenstein subalgebras of the form $\OO_S \cdot 1 + \OO_S (t^2 + at^3 + bt^5) + \OO_S(t^4 + 2at^5)$, while $Z_2$ generically parametrizes Gorenstein subalgebras of the form $\OO_S \cdot 1 + \OO_S (t^3 + at^5) + \OO_S (t^4 + bt^5)$. (We remark that in the classification of Gorenstein singularities of genus 3 by Battistella \cite{battistella_genus_3_sings}, the singularities in differential stratum $\mathcal{H}(4)^{ev}$ live in $Z_1$ while the singularities in $\mathcal{H}(4)^{odd}$ live in $Z_2$.) The closed points of the intersection $Z_1 \cap Z_2$ parametrize subalgebras with a basis of the form $1, at^3 + bt^4, t^5$, all of which have either conductance 4 or 5. Thus, passing to the open subset $\Ter^{3,6}_{A_{(6)}}$ removes this intersection, and we find that $Z_1$ and $Z_2$ restrict to distinct connected components.

Ishii also constructs a stratification of $(\Ter^\delta_{k[t]/t^{2\delta}})_{red}$ by associated numerical semigroup \cite[Section 3]{ishii_moduli_subrings}. In the case of $\Ter^{2}_{A_{(6)}}$, the strata and associated subalgebras are:
\begin{center}
\begin{tabular}{|c|c|c|} \hline
  Stratum & Subalgebras & Conductance \\ \hline
  $Z_{\langle 2,7 \rangle}$ & $k \cdot 1 + k (t^2 + at^3 + bt^5) + k (t^4 + 2at^5)$ & 6 \\
  $Z_{\langle 3,4 \rangle}$ & $k \cdot 1 + k(t^3 + at^5) + k (t^4 + bt^5)$ & 6\\
  $Z_{\langle 3,5,7,11 \rangle}$ & $k \cdot 1 + k(t^3 + at^4) + k t^5$ & 5\\
  $Z_{\langle 4,5,6,7 \rangle}$ & $k \cdot 1 + kt^4 + kt^5$ & 4 \\ \hline
\end{tabular}
\end{center}

This refines our stratification by conductance, as both $Z_{\langle 2,7 \rangle}$ and $Z_{\langle 3,4\rangle}$ consist of subalgebras of conductance 6, while $Z_{\langle 3,5,7,11 \rangle}$ consists of subalgebras of conductance 5, and $Z_{\langle 4,5,6,7 \rangle}$ the lone subalgebra of conductance $4$. Ishii does not construct a stratification corresponding to curve singularities with multiple branches.

This same example shows that a fixed territory of curve singularity may parametrize curve singularities with distinct topological types, since both the $E_6$ singularity $k\llbracket t^3, t^4 \rrbracket$ and the $A_6$ singularity $k \llbracket t^2, t^7 \rrbracket$ are among those parametrized by $\Ter^{3,6}_{k[t]/t^6}$.
\end{remark}

\subsection{Relationship to crimping spaces} \label{ssec:crimping}
Van der Wyck's thesis \cite{fred_thesis} constructs ``crimping spaces'' parametrizing the normalizations of a fixed curve singularity class.
These attempt to solve a similar problem to the territories of curve singularities constructed here. (In fact, our definition of $\EE_{g,n}$ is motivated by that of the stack of 1d algebras with resolution of \cite[Definition 1.15]{fred_thesis}.)
Accordingly, in this section, we explain the relationship between crimping spaces and territories of curve singularities.  Our main results do not depend on this section, so readers interested in those results may wish to skip ahead. For simplicity, we work over an algebraically closed field $k$. For this section only, we will also write $\AA^n$ for $\Spec k[x_1, \ldots, x_n]$ and $\GG_m$ for $\Spec k[t,t^{-1}]$.

Let $B$ be a 1-dimensional $k$-algebra singular at a unique point $p$. Let $A$ be its normalization. Write $J = \sqrt{\Cond_{A/B}}$, where the radical is taken as an ideal of $A$. Let $N$ be any positive integer such that $J^N \subseteq B$. (Equivalently by Lemma \ref{lem:affine_conductor_largest_ideal}, we may ask for $N$ such that $J^N \subseteq \Cond_{A/B}$.) We remark that since $k$ is a field the cokernel of $B/\cond_{A/B} \to \EEnd_{k}(\Delta)$ is automatically locally free, so
formation of the conductor commutes with base change.

\begin{definition} \label{def:crimping_scheme}
We denote by $\Aut^{[N]}_{k \to (A,J)}$ the group $k$-scheme representing the functor
\[
  \Aut^{[N]}_{k \to (A,J)}(S) =  \left\{ \phi : A/J^N|_S \to A/J^N|_S \, \, \middle| \, \, \begin{minipage}[c]{6cm} $\phi$ is an $\OO_S$-algebra isomorphism and $\phi(J/J^N|_S) = J/J^N|_S$ \end{minipage} \right\}.
\]
We denote by $\Aut^{[N]}_{k \to B \to (A,J)}$ the closed (but not necessarily normal) subgroup $k$-scheme of $\Aut^{[N]}_{k \to (A,J)}$ representing the functor
\[
  \Aut^{[N]}_{k \to B \to (A,J)}(S) = \{ \phi \in \Aut^{[N]}_{k \to (A,J)}(S) \mid \phi(B/J^N|_S) = B/J^N|_S \}.
\]
The \emphbf{crimping scheme} of the singularity $p$ is the fppf sheafification of the functor quotient \cite[Definition 1.71]{fred_thesis}
\[
  \Cr = \Aut^{[N]}_{k \to (A,J)} / \Aut^{[N]}_{k \to B \to (A,J)},
\]
which is a scheme by \cite[1.73]{fred_thesis}.
\end{definition}

The automorphism group $\Aut^{[N]}_{k \to (A,J)}$ is permitted to exchange branches of $p$ with distinct conductances; we will want to remove this possibility for our comparison to territories.

\begin{definition} \label{def:strict_crimping_scheme}
We denote by $\Aut^{[N]}_{k \to (A,J), 0}$ the closed subgroup scheme of $\Aut^{[N]}_{k \to (A,J)}$ cut out by the condition that $\phi((\cond_{A/B}/J^N)|_S) = (\cond_{A/B}/J^N)|_S$, and denote by $\Aut^{[N]}_{k \to B \to (A,J), 0}$ the subgroup $\Aut^{[N]}_{k \to (A,J), 0} \cap \Aut^{[N]}_{k \to B \to (A,J)}$. We define the \emphbf{strict crimping scheme} to be the scheme representing the quotient
\[
  \Cr_0 = \Aut^{[N]}_{k \to (A,J), 0} / \Aut^{[N]}_{k \to B \to (A,J), 0}.
\]
\end{definition}

We remark that since $\Aut^{[N]}_{k \to (A,J), 0}$ has finite index in $\Aut^{[N]}_{k \to (A,J)}$, $\Cr_0$ is a closed subscheme of $\Cr$, and $\Cr$ is a union of finitely many translates of $\Cr_0$ under $\Aut^{[N]}_{k \to (A,J)}$.

\begin{example}[Crimping schemes vs strict crimping schemes]
Suppose $A = k[t_1] \times k[t_2]$ and $B = \{ (f(t_1), g(t_2)) \mid f(0) = g(0), g'(0) = 0 \}$, i.e., $B$ has the singularity obtained by gluing
a cusp to a transverse smooth branch. In this case the conductor is $\cond_{A/B} = ((t_1,0), (0, t_2^2))$ (i.e., the branch conductances are $1$ and $2$), its radical is $J = ((t_1,0), (0, t_2))$, and we may take $N = 2$. We will see that the crimping scheme is $\ZZ / 2\ZZ$, while the strict crimping scheme is a point, $\Spec k$.

Observe that
\[
  \Aut^{[2]}_{k \to (A, J)} = \Aut_k(k[t_1]/(t_1^2) \times k[t_2]/(t_2^2)) \cong (\GG_m \times \GG_m) \rtimes (\ZZ/2\ZZ),
\]
where the factors of $\GG_m$ scale $t_1$ and $t_2$ and the non-trivial element of $\ZZ/2\ZZ$ interchanges the coordinates of $k[t_1]/(t_1^2) \times k[t_2]/(t_2^2)$. This non-trivial element neither preserves $B / J^2$ nor preserves $\cond_{A/B}/J^2$, so
\begin{align*}
  \Aut^{[2]}_{k \to (A,J), 0} &= \\
  \Aut^{[2]}_{k \to B \to (A,J)} &= \\
  \Aut^{[2]}_{k \to B \to (A,J), 0} &= \GG_m \times \GG_m.
\end{align*}
It follows that $\Cr \cong \ZZ / 2\ZZ$ and $\Cr_0 \cong \Spec k$, as claimed.
\end{example}

In order to set up a comparison with territories of singularities, let us suppose that $p$ has genus $g$ and branch conductances $c_1 \geq \cdots \geq c_m$. Write $A_{\vec{c},k}$ for $A_{\vec{c}} \otimes k$ and $A^+_{\vec{c}, k}$ for $A^+_{\vec{c}} \otimes k$. Identify $A/\cond_{A/B}$ with $A_{\vec{c}, k}$
and write $\pi : A/J^N \to A_{\vec{c},k}$ for the quotient map induced by the fact that $\cond_{A/B} \supseteq J^N$; its kernel
is $\cond_{A/B} / J^N = \cond_{(A/J^N)/(B/J^N)}$.

\begin{lemma} \label{lem:aut_unibranch_A_c}
For any integer $d \geq 2$, there is an isomorphism of schemes (not preserving the group structure)
\[
  \Aut_k(k[t]/(t^{d})) \longrightarrow \GG_m \times \AA^{d - 2}.
\]
given by taking $\phi : k[t]/t^d \to k[t]/t^d$ to the tuple $(a_1, \ldots, a_{d-1})$ such that
\[
  \phi(t) = a_1t + \cdots + a_{d-1}t^{d-1}.
\]

For $d = 1$, $\Aut_k(k[t]/t)$ is trivial.
\end{lemma}
\begin{proof}
Omitted.
\end{proof}

\begin{lemma} \label{lem:aut_multibranch_A_c}
For any non-negative integers $c_1 \geq \cdots \geq c_m$,
\[
  \Aut_k(A_{\vec{c},k}) \cong \prod_{i=1}^m \Aut(k[t_i]/(t_i^{c_i})) \rtimes \mathfrak{S}_{\vec{c}}
\]
where $\mathfrak{S}_{\vec{c}}$ is the subgroup of the symmetric group on $m$ letters determined by the condition $\sigma \in \mathfrak{S}_{\vec{c}}$ if and only if $c_{\sigma(i)} = c_i$ for all $i$. In particular,
\[
  \Aut^{[N]}_{k \to (A,J),0} \cong \prod_{i=1}^m \Aut(k[t_i]/(t_i^{N})) \rtimes \mathfrak{S}_{\vec{c}}
\]
\end{lemma}
\begin{proof}
Omitted.
\end{proof}

\begin{lemma} \label{lem:auts_to_auts_mod_conductor}
We have the following:
\begin{enumerate}
  \item[(A)] There is a natural surjective morphism of group schemes
    \[
      \xi : \Aut^{[N]}_{k \to (A,J), 0} \longrightarrow \Aut_k(A_{\vec{c},k})
    \]
    given by taking an isomorphism $\phi : A/J^N \to A/J^N$ to the map $\xi(\phi) : A^+_{\vec{c}} \to A^+_{\vec{c}}$ defined by $\pi(a) \mapsto \pi(\phi(a))$.
  \item[(B)] There is a natural action of $\Aut_k(A_{\vec{c},k})$ on $\Tersing(g, \vec{c}) \otimes k$ given on $S$-points by
    $\phi \cdot C = \phi(C)$.
  \item[(C)] The subgroup $\Aut^{[N]}_{k \to B \to (A,J),0}$ of $\Aut^{[N]}_{k \to (A,J), 0}$ is $\xi^{-1}(\Stab_{\Aut_k(A_{\vec{c},k})}(B / \cond_{A/B}))$.
\end{enumerate}
\end{lemma}

\begin{proof}
Claim (C) is clear once (A) and (B) are established. For simplicity, we will give the proof for (A) on $k$-points. The generalization to
$S$-points is immediate.

To show that $\xi$ is well-defined, we must show that
\begin{enumerate}
  \item $\pi(a) = \pi(b) \implies \pi(\phi(a)) = \pi(\phi(b))$;
  \item $\xi(\phi)$ preserves sums, products, and 1;
  \item $\xi(\phi)$ is bijective.
\end{enumerate}

Item (i) holds since $\phi(\cond_{(A/J^N)/(B/J^N)}) = \cond_{(A/J^N)/(B/J^N)}$. Item (ii) clearly holds. Since $\phi$ is surjective, $\xi(\phi)$ is surjective. The kernel of $\xi(\phi)$ is clearly trivial, so $\xi(\phi)$
is an automorphism.

To see that $\xi$ is surjective, observe that we may identify $A/J^N$ with $A_{\vec{N},k}$ where $\vec{N}$ is the $m$-vector $(N, \ldots, N)$ and we may identify $\pi$ with the natural quotient $A_{\vec{N},k} \to A_{\vec{c},k}$.
Then in light of Lemmas \ref{lem:aut_unibranch_A_c} and \ref{lem:aut_multibranch_A_c},
we see that the map $\xi$ is simply given by projecting away some factors of $\AA^1$ or $\GG_m$, so $\xi$ is surjective.

To show (B), it suffices to check that $\xi(\phi)$ takes $A^+_{\vec{c},k}$ into itself.
Write $\mathrm{ev}_j$ for the morphism $\mathrm{ev}_j : A_{\vec{c},k} \to k$ given by taking $(f_i(t_i)) \mapsto f_j(0)$. These are
the quotient maps associated to the maximal ideals of $A_{\vec{c},k}$, so the automorphism $\xi(\phi)$ induces a permutation of
the evaluation maps, i.e., there is a permutation $\sigma \in S_m$ such that $\xi(\phi) \circ \mathrm{ev}_j = \mathrm{ev}_{\sigma(j)}$ for all $j$.
Since $A^+_{\vec{c},k}$ is the subset of $A_{\vec{c},k}$ cut out by $\mathrm{ev}_{i} = \mathrm{ev}_{j}$ for all $i, j$, we conclude $A^+_{\vec{c},k}$
is taken into itself. Then (B) follows.
\end{proof}

Next, we conclude that strict crimping spaces are the orbits in $\Tersing(g,\vec{c})$ under the action of $\Aut_k(A_{\vec{c}})$.

\begin{theorem} \label{thm:crimp_vs_ter}
There is a locally closed immersion
\[
  f : \Cr_0 \to \Tersing(g, \vec{c}) \otimes k
\]
identifying the strict crimping space of $p$ with the orbit of $B/\cond_{A/B}$ under the action of $\Aut_k(A_{\vec{c},k})$ on $\Tersing(g, \vec{c}).$
\end{theorem}
\begin{proof}
By Lemma \ref{lem:auts_to_auts_mod_conductor},
\begin{align*}
  \Cr_0 &= \Aut^{[N]}_{k \to (A,J),0} / \Aut^{[N]}_{k \to B \to (A,J), 0} \\
        &\cong \Aut_k(A_{\vec{c},k}) / \Stab(B/\cond_{A/B}).
\end{align*}
As is usual for orbits, the natural map $f : \Cr_0 \to \Tersing(g,\vec{c})$ is induced by taking $\phi \in \Aut_k(A_{\vec{c},k})$ to $\phi(B/\cond_{A/B})$. This is a locally closed immersion by \cite[Proposition 1.65]{milne_algebraic_groups}.
\end{proof}

\begin{example} \label{ex:crimping_all_2s}
Let $m$ be an integer such that $m \geq 1$ and let $g$ be an integer such that $1 \leq g \leq m$. Write $\vec{2}$ for the $m$-vector $(2,\ldots,2)$. We will compute $\Tersing(g, \vec{2})$ and consider the embedded strict crimping spaces
as orbits under $\Aut_k(A_{\vec{2},k})$.

Observe that any corank $g$ subalgebra of $A^+_{\vec{2},k}$ is of the form $k(1,\ldots, 1) + V$ where $V$ is a codimension $g$ subspace of $\mm = (t_1, \ldots, t_m)$. On the other hand, any vector subspace of $\mm$ in $A^+_{\vec{2},k}$ is closed under multiplication since $\mm$ is square zero. Accordingly,
\[
  \Ter^{g}_{A^+_{\vec{2},k}} \cong \Grass_{k}(m - g, \mm) \cong \Grass_{k}(m - g, m),
\]
where a codimension-$g$ subspace $V \subseteq \mm$ corresponds to the algebra $k(1,\ldots, 1) + V \subseteq A^+_{\vec{c},k}$, which corresponds in turn to the curve singularity modelled by
\[
  k(1,\ldots,1) + V + (t_1^2,\ldots, t_m^2) \subseteq k[t_1] \times \cdots k[t_m].
\]
A local computation shows that the locus $\Tersing(g, \vec{2}) \subseteq \Ter^g_{A^+_{\vec{2}, k}}$ with the desired conductance corresponds to the open locus in $\Grass_{k}(m - g, m)$ parametrizing the $(m - g)$-planes
not containing any coordinate axis. (The locus of planes that do contain a coordinate axis has codimension $g$.) This removes the curve singularities that decompose into a transverse union of a smooth branch with another singularity.

The group $\Aut_k(A_{\vec{2},k})$ is isomorphic to $(\GG_m)^{m} \rtimes S_m$. The action by the subgroup $(\GG_m)^{m}$ is by scaling the coordinates, while the action of the subgroup $S_m$ is by permuting the coordinates.

For $g = 1$, the action of $(\GG_m)^m$ on
\[
  \Tersing(1, \vec{2}) \cong \PP^{m-1} - V(x_0) - \cdots - V(x_{m-1}) \cong \GG_m^{m-1}
\]
is induced by the usual torus action on $\Grass(m - 1, m) \cong \PP^{m-1}$, which is of course transitive. We see that there is one isomorphism class of Gorenstein, genus 1 singularity for each number of branches. This agrees with Smyth's classification of Gorenstein, genus 1 singularities \cite[Appendix A]{smyth_mstable}. In the case $m = 1$, the territory possesses one point corresponding to the subalgebra $k \cdot 1 \subseteq k[t]/t^2$ corresponding in turn to a cusp $k[t^2,t^3]$. For $m = 2$, corresponding to $u \in \GG_m$ we have subalgebras $k \cdot 1 + k(t_1 + u t_2) \subseteq A_{(2,2),k}$, all corresponding to tacnodes. For $m \geq 3$, the point $(u_2, \ldots, u_m) \in \GG_m^{m-1}$ corresponds to the subalgebra $k \cdot 1 + k(t_1 + u_2 t_2 + \cdots + u_m t_m) \subseteq A_{\vec{2},k}$, all corresponding to the elliptic $m$-fold point, isomorphic to $m$ general lines through the origin of $\AA^{m-1}$.

For larger $g$, $m$, we can compute the dimension of orbits and stabilizers by concentrating on the action by $(\GG_m)^m,$ since $S_m$ is discrete.

Concretely, consider the standard affine patch $\Spec k[a_{i,j}]_{1 \leq i \leq g, 1 \leq j \leq m - g}$ of the Grassmannian, where a point $(a_{i,j})$ corresponds to the subspace with basis given by the columns of the matrix
\[
  \begin{bmatrix}
      \mathbb{I}_{m - g} \\
      [ a_{i,j}]_{1 \leq i \leq g, 1 \leq j \leq m - g}.
  \end{bmatrix}
\]
This patch is $(\GG_m)^m$-invariant, and the action of $(\GG_m)^{m}$ is given by
\[
  (u_l)_{l = 1}^m \cdot (a_{i,j})_{1 \leq i \leq g, 1 \leq j \leq m - g} = (u_{g - m + i}u_j^{-1}a_{i,j}).
\]
That is, the first $m - g$ coordinates of $(\GG_m)^m$ act with weight $-1$ on columns, and the remaining $g$ coordinates of $(\GG_m)^m$ act with weight $1$ on rows.

Observe that for any $g$ and $m$, $\Aut_k(A_{\vec{2},k})$ has dimension $m$, while $\Tersing(g, \vec{2})$ has dimension $g(m - g)$. Since the stabilizer of any point always contains the diagonal subgroup of $(\GG_m)^m$, orbits have dimension at most $m - 1$.
Already for $g = 2$ and $m = 4$, since $m - 1 = 3$ and $g(m-g) = 4$, there are no strict crimping spaces open in $\Tersing(2, (2,2,2,2)) \otimes k.$ This shows that strict crimping spaces, while locally closed, do not yield a locally finite stratification of territories.

The singularities considered here for $g \geq 2$ are never Gorenstein, so do not appear in Battistella's classifications of Gorenstein singularities for genus $2, 3$ \cite[Section 2]{battistella_m_stable}, \cite{battistella_genus_3_sings}, and are never planar.

We remark that the orbits of $(\GG_m)^m$ of $\Grass(m - g, m)$ and the associated quotient have been studied classically and have connections to hypergeometric functions, $K$-theory, combinatorial constructions of characteristic classes, and equivalence classes of configurations of points in projective space (see for example \cite{kapranov_chow_quotients,gelfand_goresky_macpherson_serganova,gelfand_macpherson,gelfand_serganova}). In particular, the ``Chow quotient'' $\Grass(2, m) /\!/ (\GG_m)^m$ is isomorphic to $\ol{M}_{0,m}$ \cite[Theorem 4.1.8]{kapranov_chow_quotients}.
\end{example}

\section{Territories of schemes concentrated on subschemes}
\label{sec:territories_schemes}

In this section, we extend the notion of territories of sheaves of algebras to the case of schemes, anticipating our application to equinormalized families of curves. We give an explicit moduli functor, then prove its representability by a territory of a finite locally free algebra. Generalizations to representable morphisms of algebraic stacks follow easily.

\begin{definition} \label{def:global_territory}
Suppose that $\tilde{\pi} : \tilde{X} \to S$ is a locally quasiprojective family of schemes. Let $\iota : Z \to \tilde{X}$ be a closed subscheme such that $\tilde{\pi} \circ \iota: Z \to S$ is finite locally free of fixed rank. Write $\mathscr{I}_{Z/\tilde{X}}$ for the sheaf of ideals of $Z$ in $\tilde{X}$. Define a functor $F^\delta_{Z/\tilde{X}/S}$ that assigns to an $S$-scheme $T$ the set of isomorphism classes of diagrams
\[
  \begin{tikzcd}
    \tilde{X}_T \ar[rr, "\nu"] \ar[dr, "\tilde{\pi}_T"'] &  & X \ar[dl, "\pi"] \\
     & T
  \end{tikzcd}
\]
such that
\begin{enumerate}
  \item $\nu$ is a finite morphism of schemes,
  \item the associated map of sheaves $\nu^\sharp : \OO_{X} \to \nu_*\OO_{\tilde{X}_T}$ is injective,
  \item $\nu^\sharp(\OO_{X})$ contains $\nu_*\mathscr{I}_{Z_T/X}$,
  \item letting $\Delta = (\nu_*\OO_{\tilde{X}_T}) / \nu^\sharp(\OO_{X})$, the $\OO_T$-module $\pi_*\Delta$
  is locally free of rank $\delta$.
\end{enumerate}
Restriction morphisms $F^\delta_{Z/\tilde{X}/S}(T \to S) \to F^\delta_{Z/\tilde{X}/S}(U \to S)$ are given by pullback.
\end{definition}

\begin{remark}
If $f : T \to S$ is an $S$-scheme, $F^\delta_{Z/\tilde{X}/S} \times_S T \cong F^\delta_{Z_T/\tilde{X}_T/T}$.
\end{remark}

It follows from the definition that the family $X \to T$ is obtained by modifying the scheme $\tilde{X}$ only within $Z$. Intuitively, $\tilde{X} \to X$ may crimp jets or glue points of $Z$.

\begin{lemma} \label{lem:the_nature_of_nu}
Let $\nu : \tilde{X}_T \to X$ be a $T$-morphism as in Definition \ref{def:global_territory}. Then
\begin{enumerate}
    \item $\nu$ is surjective and scheme-theoretically surjective;
    \item $\nu$ restricts to an isomorphism $\tilde{X}_T - Z_T \to X - \nu(Z_T)$; and
    \item $\nu^{-1}(\nu(Z_T)) = Z_T$ as a closed subscheme of $\tilde{X}_T$.
\end{enumerate}
\end{lemma}
\begin{proof}
Scheme-theoretic surjectivity of $\nu$ follows from part (ii) of Definition \ref{def:global_territory}. Since $\nu$ is finite, $\nu$ is also surjective. This proves (i).

Next we prove (iii). Let us assume $\nu^\sharp$ is an inclusion to ease the notation. Since $\nu$ is proper, $\nu(Z_T)$ is closed in $X$. Since $\nu_*\mathscr{I}_{Z_T/\tilde{X}_T} \subseteq \OO_X$, the ideal sheaf of $\nu(Z_T)$ in $X$ is $\nu_*(\mathscr{I}_{Z_T/\tilde{X}_T})$. Then the ideal sheaf of $\nu^{-1}(\nu(Z_T))$ is the image of the map $\nu^*(\nu_*\mathscr{I}_{Z_T/\tilde{X}_T}) \to \OO_{\tilde{X}_T}$ obtained by pulling back the inclusion $\nu_*\mathscr{I}_{Z_T/\tilde{X}_T} \to \OO_X$.

Let $U \subseteq X$ be an open affine subscheme of $X$. Write $B = \OO_{X}(U), A = \OO_{\tilde{X}_T}(\nu^{-1}(U))$, and $I = \mathscr{I}_{Z_T/\tilde{X}_T}(U)$. On $U$, the map $\nu^*(\nu_*\mathscr{I}_{Z_T/\tilde{X}_T}) \to \OO_{\tilde{X}_T}$ becomes the morphism $I \otimes_B A \to A$ given by $i \otimes a \mapsto ia$. Since $I$ is an ideal of $A$, the image of $I \otimes_B A \to A$ is $I$. It follows $\nu^{-1}(\nu(Z_T)) = Z_T.$

Finally we prove (ii). Consider the sequence
\begin{equation} \label{eq:fundamental_sequence}
   0 \to \OO_X \to \nu_*\OO_{\tilde{X}_T} \to \Delta \to 0
\end{equation}
Let $V = X - \nu(Z_T)$ and $\tilde{V} = \nu^{-1}(V) = \tilde{X}_T - \nu^{-1}(\nu(Z_T)) = \tilde{X}_T - Z_T$. Tensoring with $\OO_V$,
\[
  0 \to \OO_V \to \nu_*\OO_{\tilde{V}} \to \Delta \otimes \OO_V \to 0
\]
is exact. By part (iii) of Definition \ref{def:global_territory}, $\Delta$ is supported on $\nu(Z_T)$, so $\Delta \otimes \OO_V = 0$. Then $\OO_V \to \nu_*\OO_{\tilde{V}}$
is an isomorphism. Since $\nu$ is affine, we conclude that (ii) holds.
\end{proof}

We do not explicitly impose any hypotheses on $X \to T$ in Definition \ref{def:global_territory}, yet $X$ generally has nice properties when $\tilde{X}$ has nice properties.

\begin{proposition} \label{prop:X_to_T_is_nice}
Let $\nu : \tilde{X}_T \to X$ be a $T$-morphism as in Definition \ref{def:global_territory}. Then:
\begin{enumerate}
  \item the map $\pi : X \to T$ is separated;
  \item if $\tilde{\pi} : \tilde{X} \to S$ is proper, then $\pi : X \to T$ is proper;
  \item if $\tilde{\pi} : \tilde{X} \to S$ is flat, then $\pi : X \to T$ is flat.
\end{enumerate}
\end{proposition}
\begin{proof}
First we show that $\pi$ is separated. Consider the diagram
\[
\begin{tikzcd}
  \tilde{X}_T \ar[d, "\Delta_{\tilde{X}_T/T}"] \ar[r, "\nu"] & X \ar[d, "\Delta_{X/T}"] \\
  \tilde{X}_T \times_T \tilde{X}_T \ar[r, "\nu \times \nu"] & X \times_T X.
\end{tikzcd}
\]
The top and bottom arrows are finite and the left arrow is a closed embedding. Following the arrows, $\tilde{X}_T \to \tilde{X}_T \times_T \tilde{X}_T \to X \times_T X$, we see that the image of $\tilde{X}_T$ is closed in $X \times_T X$.
Since $\nu$ is surjective, $\nu(\tilde{X}_T) = X$. It follows that the image of $X$ in $X \times_T X$ is also closed, whence $X$ is separated over $T$.

Now, since $\pi : X \to T$ is separated, if $\tilde{\pi} : \tilde{X}_T \to T$ is proper, then the image $\nu(\tilde{X}_T) = X \to T$ is also proper. This gives us (ii).

Finally, we show $\pi : X \to T$ is flat. Part (ii) of Lemma \ref{lem:the_nature_of_nu} implies that $X \to T$ is flat away from $\nu(Z_T)$, so it suffices to show flatness on an open neighborhood of $\nu(Z_T)$. Since flatness is \'etale local we may assume $S$ and $T$ are affine and $\tilde{\pi}$ is projective. Again passing to an \'etale local cover if necessary, since $Z$ is finite and $\tilde{X}$ is projective, we can find an affine open subset $\tilde{U} \subseteq \tilde{X}$ such that $Z \subseteq \tilde{U}$ and $\tilde{\pi}(\tilde{U}) = S$ by taking the complement of a suitable hyperplane. Let $U = \nu(\tilde{U} \times T)$. Since $\nu$ is finite surjective and $\tilde{U} \times T$ is affine, $U$ is affine \cite[Tag 01ZT]{stacks-project}. Now, $U$ contains the support of $\Delta$ and $R^1\pi_*\OO_U = 0$, since $U$ is affine over the affine $T$. Applying $\pi_*(-|_U)$ to sequence \eqref{eq:fundamental_sequence}, we have
\[
  0 \to \pi_*\OO_U \to \tilde{\pi}_*\OO_{\tilde{U}\times T} \to \pi_*\Delta \to 0.
\]
Since $\pi_*\Delta$ is finite locally free, it is flat. The middle $\OO_T$-module $\tilde{\pi}_*\OO_{\tilde{U}\times T}$ is flat since $\tilde{X}_T$ is flat over $T$ and $\tilde{U} \times T$ is affine over $T$.
Since the middle and last terms are flat $\OO_T$-modules, $\pi_*\OO_U$ is flat too. Therefore, $X \to T$ is flat.
\end{proof}

\begin{theorem} \label{thm:territory_of_scheme_representable}
With notation as above, the functor $F^\delta_{Z/\tilde{X}/S}$ is representable by the scheme $\Ter^{\delta}_{\tilde{\pi}_*\OO_Z}$.
\end{theorem}
\begin{proof}
The point is the same as Theorem \ref{thm:contains_I_vs_subalg_of_A/I}, but with schemes. We will show that there is a natural isomorphism $F^\delta_{Z/\tilde{X}/S} \to F^\delta_{\tilde{\pi}_*\OO_Z}$.

Let us describe the map $F^\delta_{Z/\tilde{X}/S}(T \to S) \to F^\delta_{\tilde{\pi}_*\OO_Z}(T \to S)$. Since $F^\delta_{Z/\tilde{X}/S} \times_S T \cong F^\delta_{Z_T/\tilde{X}_T/T}$ for any $T \to S$, we may assume for ease of notation that $T = S$.

Let
\[
  \begin{tikzcd}
    \tilde{X} \ar[rr, "\nu"] \ar[dr, "\tilde{\pi}"'] &  & X \ar[dl, "\pi"] \\
     & S
  \end{tikzcd}
\]
be an element of $F^\delta_{Z/\tilde{X}/S}(S)$. Without loss of generality, we may take $\nu^\sharp$ to be an inclusion map, so that $\OO_{X} \subseteq \nu_*\OO_{\tilde{X}}$.
By the universal property of quotients, $\OO_{\nu(Z)} = \OO_{X} / \nu_*\mathscr{I}_{Z/\tilde{X}}$ is naturally a sub-sheaf of algebras of $\nu_*\OO_Z$ and
\[
  \nu_*\OO_{\tilde{X}}/\OO_{X} \cong \nu_*\OO_{Z} / \OO_{\nu(Z)}.
\]
Then since $\pi_*(\nu_*\OO_{\tilde{X}}/\OO_{X})$ is a locally free $\OO_S$-module of rank $\delta$, so is $\pi_*(\nu_*\OO_Z / \OO_{\nu(Z)})$. Now, $\nu_*\OO_Z$ and $\OO_{\nu(Z)}$ are both supported on $\nu(Z)$. Since $\nu$ is finite and $X$ is separated,
$\nu(Z)$ is proper over $S$. Since $\nu(Z)$ is proper and quasi-finite over $S$, it is finite, hence affine. Now $\pi_*{\OO_{\nu(Z)}}$ is a subsheaf of $\OO_S$-algebras of $\tilde{\pi}_*\OO_Z$ and, since $R^1\pi_*(\OO_{\nu(Z)}) = 0$,
\[
  \pi_*(\nu_*\OO_Z / \OO_{\nu(Z)}) \cong \tilde{\pi}_*\OO_Z / \pi_*\OO_{\nu(Z)}
\]
is a locally free $\OO_S$-module of rank $\delta$. Therefore,
\[
  \begin{tikzcd}
    \tilde{X} \ar[rr, "\nu"] \ar[dr, "\tilde{\pi}"'] &  & X \ar[dl, "\pi"] \\
     & S
  \end{tikzcd}
  \quad \mapsto
  \pi_*(\OO_{\nu(Z)})
\]
is a well-defined map $F^\delta_{Z/\tilde{X}/S}(S) \to F^\delta_{\tilde{\pi}_*\OO_Z}$. This map commutes with base change $T \to S$ since the maps $\nu$ and $\nu(Z) \to S$ are affine \cite[Tag 02KG]{stacks-project}.

Next, we construct an inverse map $F^\delta_{\tilde{\pi}_*\OO_Z}(S) \overset{\sim}{\to} F^\delta_{Z/\tilde{X}/S}(S)$. Suppose given a subsheaf of $\OO_S$-algebras $\mathscr{B}$ of $\tilde{\pi}_*\OO_Z$ of corank $\delta$.
Since $Z \to S$ is finite and $\tilde{X} \to S$ is locally quasiprojective, every fiber of $Z \to \Spec_{S}(\mathscr{B})$ admits an affine open neighborhood in $\tilde{X}$ by taking the complement of an appropriate hyperplane in an ambient projective space.
By \cite[Tag 0E25]{stacks-project}, a pushout $X$ exists:
\[
  \begin{tikzcd}
    Z \ar[d, "f"] \ar[r, "j"] & \tilde{X} \ar[d,"\nu"] \\
    \Spec_{S}(\mathscr{B}) \ar[r, "i"] & X
  \end{tikzcd}
\]
$\nu : \tilde{X} \to X$ is affine, and $\OO_{X} = i_*\OO_{\Spec \mathscr{B}} \times_{c_*\OO_Z} \nu_*\OO_{\tilde{X}}$ satisfies the desired properties, where $c = i \circ f = \nu \circ j.$
\end{proof}

\begin{definition} \label{def:territory_of_scheme}
Accordingly, we will write $\Ter^{\delta}_{Z/\tilde{X}/S}$ for the scheme, unique up to unique isomorphism, representing the functor $F^\delta_{Z/\tilde{X}/S}$. We call this the $\delta$-territory of $\tilde{X}$ concentrated on $Z$.
\end{definition}

\begin{corollary} \label{cor:global_territory_finite_presentation}
Let $\nu : \tilde{X}_T \to X$ be a $T$-morphism as in Definition \ref{def:global_territory}. If $\tilde{\pi} : \tilde{X} \to S$ is proper, then $\pi : X \to T$ is of finite presentation.
\end{corollary}
\begin{proof}
To avoid overloading subscripts, write $\mathcal{T} = \Ter^\delta_{Z/\tilde{X}/S}$. Let
\[
\begin{tikzcd}
  \tilde{X}_{\mathcal{T}} \ar[rr, "\nu_{\mathcal{T}}"] \ar[rd, "\tilde{\pi}"] & & \mathcal{X} \ar[dl, "\pi_{\mathcal{T}}"] \\
  & \mathcal{T}
\end{tikzcd}
\]
be the universal element. Since $\tilde{\pi} : \tilde{X} \to S$ is proper, $\pi_{\mathcal{T}} : \mathcal{X} \to \mathcal{T}$ is proper. Since $\mathcal{T}$ is locally projective, it is locally noetherian.
Since $\pi_{\mathcal{T}} : \mathcal{X} \to \mathcal{T}$ is of finite type and separated over a locally noetherian base, it is of finite presentation. Since $\pi : X \to T$ is pulled back from $\pi_{\mathcal{T}}$,
the conclusion follows.
\end{proof}

By faithfully flat descent, we obtain the following corollaries.

\begin{corollary}
If $\tilde{X} \to S$ is instead assumed to be an \'etale locally quasi-projective family of algebraic spaces (for example, an arbitrary family of proper curves), then the same functor $F^\delta_{Z/\tilde{X}/S}$ is represented by an algebraic space.
\end{corollary}

\begin{corollary} \label{cor:territory_of_stack}
If $\mathcal{\tilde{X}} \to \mathcal{S}$ is an \'etale locally quasi-projective, representable morphism of algebraic stacks, then the category fibered
in groupoids $\mathcal{F}^{\delta}_{\calZ/\mathcal{\tilde{X}}/\mathcal{S}}$ whose fiber over $T$ is the category of pairs
\[
  (f : T \to \mathcal{S}, \xi)
\]
where $\xi \in \Ter^{\delta}_{\calZ|_{T}/\mathcal{\tilde{X}}|_T/T}(T)$, is represented by an algebraic stack $\Ter^{\delta}_{\calZ/\mathcal{\tilde{X}}/\mathcal{S}}$ with a morphism $\Ter^{\delta}_{\calZ/\mathcal{\tilde{X}}/\mathcal{S}} \to \mathcal{S}$
representable by algebraic spaces.

Moreover, the conclusions of Proposition \ref{prop:X_to_T_is_nice} and Corollary \ref{cor:global_territory_finite_presentation} hold in this generality.
\end{corollary}

Similarly to the case of algebras, we have representing stacks $\Ter^{\delta,c}_{\calZ/\mathcal{\tilde{X}}/\mathcal{S}}$ in which we additionally impose the condition that the cokernel of $\pi_*\OO_{\nu(\calZ)} \to \EEnd_{\mathcal{S}}(\pi_*\Delta)$ is locally free of rank $\delta^2 - c + \delta$. Theorem \ref{thm:delta_prime_stratification_affine} yields the following corollary.

\begin{corollary} \label{cor:stacky_conductor_stratification}
The spaces $\Ter^{\delta,c}_{\calZ/\mathcal{\tilde{X}}/\mathcal{S}}$ are disjoint locally closed substacks of $\Ter^\delta_{\calZ/\mathcal{\tilde{X}}/\mathcal{S}}$ with union $\Ter^\delta_{\calZ/\mathcal{\tilde{X}}/\mathcal{S}}$.
\end{corollary}

\section{Equinormalized curves of fixed conductance}
\label{sec:equinormalized_curves}

In this section, we define locally closed substacks $\EE^{\delta,c}_{g,n}$ coarsely stratifying $\EE_{g,n}$ and give an explicit construction of each $\EE^{\delta,c}_{g,n}$ as an algebraic stack in terms of territories.

\subsection{Definitions}

We begin by translating conductors and related notions to the setting of families of equinormalized curves.

\begin{definition}
Given a family
\[
\begin{tikzcd}
  \tilde{C} \ar[dr, "\tilde{\pi}"] \ar[rr, "\nu"] & & C \ar[dl, "\pi"] \\
  & S &
\end{tikzcd}
\]
of equinormalized curves over $S$, we write $\Delta_{\tilde{C}/C}$ for the $\OO_C$-module $\nu_*\OO_{\tilde{C}} / \OO_C$, and write $\delta_{\tilde{C}/C}$ for the rank of $\pi_*\Delta_{\tilde{C}/C}$, which we call the \emphbf{total $\delta$-invariant} of $\nu$.
The \emphbf{conductor} of $\nu$ refers either to the quasicoherent sheaf of ideals
\[
  \cond_{\tilde{C}/C} = \AAnn_{\OO_C}(\Delta)
\]
of $\OO_C$ or the quasicoherent sheaf of ideals of $\OO_{\tilde{C}}$
\[
  \Cond_{\tilde{C}/C} = \nu^{-1}\cond_{\tilde{C}/C} \otimes_{\nu^{-1}\nu_*\OO_{\tilde{C}}} \OO_{\tilde{C}}.
\]

We will write $\Delta'_{\tilde{C}/C}$ for the $\OO_C$-module $\OO_C / \cond_{\tilde{C}/C}$. In the case that $\Delta'_{\tilde{C}/C}$ is locally free of finite rank, we write $\delta'$ for its rank, which we call the \emphbf{total $\delta'$-invariant} of $\nu$. If $\OO_{\tilde{C}} / \Cond_{\tilde{C}/C}$ is locally free of finite rank, we call its rank $c$ the \emphbf{total conductance} of $\nu$.

If the family $\nu : \tilde{C} \to C$ is clear from context, we will omit the subscript $\tilde{C}/C$.
\end{definition}

Exactly as in the affine case, we have the following two properties.

\begin{lemma} \label{lem:conductor_largest_ideal}
Let $\nu : \tilde{C} \to C$ be a family of equinormalized curves over $S$. The conductor $\Cond_{\tilde{C}/C}$ is characterized by the fact that it is the largest quasicoherent ideal sheaf of $\OO_{\tilde{C}}$ such that $\nu_*\Cond_{\tilde{C}/C}$ factors through $\OO_C.$
\end{lemma}

\begin{lemma}
Let $\nu : \tilde{C} \to C$ be a family of equinormalized curves over $S$, and write $\pi : C \to S$ for the map to $S$. The conductor $\cond_{\tilde{C}/C}$ is the kernel of the
morphism
\[
  \OO_C \to \EEnd_{\OO_C}(\Delta)
\]
\end{lemma}

We now give a careful definition of $\EE^{\delta, c}_{g,n}$.

\begin{definition} \label{def:E_gn_delta_c}
Let $g, n, \delta, c$ be non-negative integers such that $\delta \leq c \leq 2\delta$. Given a scheme $S$, a \emphbf{family of $n$-pointed equinormalized curves of genus $g$ over $S$ with invariants $\delta, c$}
consists of a family $(\nu : \tilde{C} \to C, q_1, \ldots, q_n) \in \EE_{g,n}(S)$ such that
\begin{enumerate}
  \item $\pi_*\Delta$ is locally free of rank $\delta$ on $S$
  \item the cokernel of the composite
  \[
    \pi_*(\Delta') \to \pi_*\EEnd_{\OO_{C}}(\Delta) \to \EEnd_{\OO_S}(\pi_*\Delta),
  \]
  taking $f$ to multiplication by $f$ is locally free of rank $\delta^2 - c + \delta$.
\end{enumerate}

The moduli stack parametrizing families of $n$-pointed equinormalized curves of genus $g$ with invariants $\delta, c$ is denoted by $\EE^{\delta,c}_{g,n}$.
\end{definition}

\begin{proposition}
Given a family of equinormalized curves
\[
  \begin{tikzcd}
     \tilde{C} \ar[rr, "\nu"] \ar[rd, "\tilde{\pi}"] &  & C \ar[dl, "\pi"] \\
     & S & 
  \end{tikzcd}
\]
of genus $g$ with invariants $\delta, c$, formation of the conductor ideal $\cond_{\tilde{C}/C}$ and $\Delta_{\tilde{C}/C}'$ commute with base change. Moreover $\pi_*\Delta'$ is locally free of rank $\delta'$ and $\tilde{\pi}_*(\OO_{\tilde{C}}/\Cond_{\tilde{C}/C})$ is locally free of rank $c$.
\end{proposition}

\begin{proof}
Since $\Delta$ is supported on a closed subscheme of $C$ finite over $S$, we may reduce to the affine case, and the result follows from Proposition \ref{prop:conductor_commutes_base_change}.
\end{proof}

\subsection{Construction of $\EE^{\delta,c}_{g,n}$}
\label{ssec:construction}

\begin{figure}
\[
\begin{tikzcd}
  \Ter^{\delta,c}_{g,n} \ar[r,"\sim"] \ar[d, hook] & \EE^{\delta,c}_{g,n} \ar[ddl,bend left] & {} \\
  \Ter^{\delta,c}_{Z/\tilde{C}/\calH^{\delta,c}_{g,n}} \ar[d] & & \ar[u, "\substack{\text{require connected,}\\\text{correct genus}}"'] \\
  \calH^{\delta,c}_{g,n} \ar[d] & & \ar[u, "\text{add subalgebra data}"'] \\
  \moduli^\delta_{g,n} & &  \ar[u, "\text{add branch subscheme}"' ]
\end{tikzcd}
\]
\caption{Intermediate spaces in the construction of $\EE^{\delta,c}_{g,n}$. The diagonal arrow is the one that we will later stratify into fiber bundles indexed by combinatorial type.}
\label{fig:e_delta_delta_prime_construction}
\end{figure}

Fix non-negative integers $g,n, \delta, c$ with $\delta \leq c \leq 2\delta$. We now give a construction of $\EE_{g,n}^{\delta,c}$ as an algebraic stack over a moduli space of curves. As a base space for this construction, we consider a moduli stack $\moduli^\delta_{g,n}$ parametrizing the possible $n$-marked source curves $(\tilde{C}, p_1, \ldots, p_n).$
Observe that if $\nu : \tilde{C} \to C \in \EE^{\delta,c}_{g,n}$ is an equinormalized curve and $K$ is the set of connected components of $\tilde{C}$, then the normalization sequence implies
\begin{equation} \label{eq:c_twiddle_components}
  1 \leq \#K = 1 + \delta - g + \sum_{k \in K} g(k) \leq 1 + \delta.
\end{equation}

\begin{definition} \label{def:m_gn_delta}
Let $\moduli^\delta_{g,n}$ be the moduli stack parametrizing smooth, proper, $n$-pointed, possibly disconnected curves $(\tilde{C}, p_1, \ldots, p_n)$ with distinct markings such that
\[
  \#K = 1 + \delta - g + \sum_{k \in K} g(k)
\]
where $K$ is the set of connected components of $\tilde{C}$.
\end{definition}

Observe that $\moduli^\delta_{g,n}$ is a finite union of stacks of the form
\[
  \left[\prod_{k \in K} \moduli_{g(k),S(k)} / G\right]
\]
where $g(k)$ is the genus of the component $k$, $S(k)$ is the set of markings on $k$, and $G$ is the group permuting the factors with equal genus and no markings. We next add the data of a length $c$ subscheme, which will be the conductor locus of $\nu : \tilde{C} \to C$.

\begin{definition} \label{def:H_gn_delta_c}
Let $\calH^{\delta,c}_{g,n}$ be the relative Hilbert scheme of $c$ points of the universal curve of $\moduli^\delta_{g,n}$. Write $\tilde{\mathcal{C}}$ for the universal curve over $\calH^{\delta,c}_{g,n}$ and $\mathcal{Z}$ for its closed subscheme.
\end{definition}

Since the universal curve of $\moduli^\delta_{g,n}$ is smooth, $\calH^{\delta,c}_{g,n}$ is represented by the $c$-fold symmetric product of the universal curve over $\moduli^\delta_{g,n}$.

\begin{definition} \label{def:ter_gn_delta_c}
Write $\Ter^{\delta,c}_{g,n}$ for the open and closed substack of
$\Ter^{\delta,c}_{\calZ/\tilde{\mathcal{C}}/\calH^{\delta,c}_{g,n}}$ where the curve $\mathcal{C}$ is connected of arithmetic genus $g$.
\end{definition}

Following the arrows $\Ter^{\delta,c}_{g,n} \to \calH^{\delta,c}_{g,n} \to \moduli^\delta_{g,n} \to \Spec \ZZ$, we see that $\Ter^{\delta,c}_{g,n}$ is a fibered category over $\Sch$ whose fiber on a scheme $S$ is the category parametrizing tuples of
\begin{enumerate}
  \item A smooth, proper, possibly disconnected family of curves $\tilde{\pi} : \tilde{C} \to S$ with $n$ sections $p_i : S \to \tilde{C}$,
  \item A closed subscheme $Z \subseteq \tilde{C}$, flat over $S$ of rank $c$,
  \item A flat and proper family of connected, reduced curves, $\pi : C \to S$ of arithmetic genus $g$;
  \item A morphism $\nu : \tilde{C} \to C$ over $S$
\end{enumerate}
such that
\begin{enumerate}
  \item $\nu$ is affine,
  \item $\nu^\sharp : \OO_{C} \to \nu_*\OO_{\tilde{C}}$ is injective,
  \item $\nu^\sharp(\OO_{C})$ contains $\nu_*\mathscr{I}_{Z/\tilde{C}}$,
  \item $\pi_*\Delta_{\tilde{C}/C}$ is locally free of rank $\delta$ over $S$,
  \item the cokernel of the composite
  \[
    \pi_*(\OO_{C}/\cond_{\tilde{C}/C}) \to \pi_*\EEnd_{\OO_{C}}(\Delta) \to \EEnd_{\OO_S}(\pi_*\Delta)
  \]
  is locally free of rank $\delta^2 - c + \delta$.
\end{enumerate}

\begin{theorem} \label{thm:algebraicity}
The natural map
\[
  \Ter^{\delta,c}_{g,n} \overset{\sim}{\longrightarrow} \EE^{\delta,c}_{g,n}
\]
forgetting the subscheme $Z$ is an equivalence of fibered categories. In particular, $\EE^{\delta,c}_{g,n}$ is an algebraic stack, with a separated, representable morphism to $\moduli^\delta_{g,n}$.
\end{theorem}
\begin{proof}
To obtain an inverse, we reconstruct the subscheme $Z$ from the conductor ideal.

Let $(\nu : \tilde{C} \to C, p_1, \ldots, p_n) \in \EE^{\delta,c}_{g,n}(S)$ be an $S$-point of $\EE^{\delta,c}_{g,n}$.
Let $Z = V(\Cond_{\tilde{C}/C}) \subseteq \tilde{C}$. By Lemma \ref{lem:conductor_largest_ideal}, $\nu_*{\Cond_{\tilde{C}/C}}$ factors through $\OO_C$, as required. 
Then $Z$ is locally free of rank $c$ over $S$. 
The formation of $Z$ is clearly invariant under isomorphism, and commutes with pullback by Proposition \ref{prop:conductor_commutes_base_change}. Therefore, we have a morphism of fibered categories $\EE^{\delta,c}_{g,n} \to \Ter^{\delta,c}_{g,n}$.
Moreover, since $\Cond_{\tilde{C}/C}$ is the largest quasicoherent sheaf of $\OO_{\tilde{C}}$-ideals such that $\nu_*\Cond_{\tilde{C}/C}$ factors through $\OO_{C}$, it is the unique such ideal whose vanishing has rank $c$ over $S$.
By the uniqueness, $\EE^{\delta,c}_{g,n} \to \Ter^{\delta,c}_{g,n}$ is inverse to the forgetful morphism. 
\end{proof}

\begin{proposition} \label{prop:stratification_by_delta_deltaprime}
The $\EE^{\delta,c}_{g,n}$s form a locally closed stratification of $\EE_{g,n}$ in the sense that for any morphism $f : S \to \EE_{g,n}$ there is a decomposition
$\sqcup_{\delta,c} S^{\delta,c} \to S$ of $S$ into disjoint locally closed subschemes, compatible with base change, such that $f|_{S^{\delta,c}}$ factors through $\EE^{\delta,c}_{g,n}$.

Moreover, provided $\delta > 0$, $\EE^{\delta,c}_{g,n}$ is non-empty only for $\delta < c \leq 2\delta$. In addition, when $c = 2\delta$, $\EE^{\delta,c}_{g,n}$ is open in
the sense that $S^{\delta,c}$ is open in $S$, and $\EE^{\delta,c}_{g,n}$ is the locus of Gorenstein curves of total delta invariant $\delta$ in $\EE_{g,n}.$
\end{proposition}
\begin{proof}
The rank of $\pi_*\Delta$ is locally constant, so we may pass to an open and closed subset of $S$ on which $\delta$ is constant.

We then apply \cite[Tag 05P9]{stacks-project} to the cokernel of $\pi_*(\OO_C/\cond_{\tilde{C}/C}) \to \EEnd_{S}(\pi_*\Delta)$ to obtain a locally closed stratification by $c$, compatible with base change.

The claims about non-emptiness and that $\EE^{\delta,c}_{g,n}$ is the Gorenstein locus follow directly from Lemma \ref{lem:conductor_bounds_and_gorenstein}. To see
that $\EE^{\delta,2\delta}_{g,n}$ is open, we may either observe that Gorenstein-ness is an open condition or that $\EE^{\delta,2\delta}_{g,n}$
is cut out by the minimal rank locus of the cokernel of $\pi_*(\OO_C/\cond_{\tilde{C}/C}) \to \EEnd_{S}(\pi_*\Delta)$.
\end{proof}

\begin{remark} \label{rem:e_bounded_conductance}
We can loosen our construction of $\EE^{\delta,c}_{g,n}$ by omitting the restriction to total conductance $c$: let $\Ter^{\delta, \leq c}_{g,n}$ be the open and closed substack of $\Ter^{\delta}_{Z/\tilde{C}/\calH^{\delta,c}_{g,n}}$ where the curve $C$ is connected of arithmetic genus $g$. Then $\Ter^{\delta,\leq c}_{g,n}$ is an algebraic stack, \emph{proper} over $\moduli^\delta_{g,n}$ containing $\EE^{\delta,c}_{g,n}$ as an open substack. Its $S$-points are tuples of
\begin{enumerate}
  \item A smooth, proper, possibly disconnected family of curves $\tilde{\pi} : \tilde{C} \to S$ with $n$ sections $p_i : S \to \tilde{C}$,
  \item A closed subscheme $Z \subseteq \tilde{C}$, flat over $S$ of rank $c$,
  \item A flat and proper family of connected, reduced curves, $\pi : C \to S$ of arithmetic genus $g$;
  \item A morphism $\nu : \tilde{C} \to C$ over $S$
\end{enumerate}
such that
\begin{enumerate}
  \item $\nu$ is affine,
  \item $\nu^\sharp : \OO_{C} \to \nu_*\OO_{\tilde{C}}$ is injective,
  \item $\nu^\sharp(\OO_{C})$ contains $\nu_*\mathscr{I}_{Z/\tilde{C}}$,
  \item $\pi_*\Delta_{\tilde{C}/C}$ is locally free of rank $\delta$ over $S$.
\end{enumerate}

As with $\Ter^{\delta,c}_{g,n}$, there is a natural morphism $\Ter^{\delta,\leq c}_{g,n} \to \EE^{\delta}_{g,n}$ given by taking $(\tilde{C} \to C, Z) \mapsto (\tilde{C} \to C)$. However, it is no longer an equivalence of categories since the subscheme $Z$ cannot be recovered from $\tilde{C} \to C$, as we will see in Remark \ref{rem:fixed_conductance_needed}.
\end{remark}

\begin{remark}
Given the properness of $\Ter^{\delta,\leq c}_{g,n}$ over $\moduli^\delta_{g,n}$, it is natural to wonder if a universally closed moduli stack containing $\EE_{g,n}$ could be constructed by allowing the curve $\tilde{C}$ to acquire nodal singularities. We leave this as a question for future work.
\end{remark}

We now tie up some loose ends from the introduction.

\begin{proposition} \label{prop:equinormalized_to_ugn}
Let $k$ be an algebraically closed field. The morphism of fibered categories $\EE_{g,n} \to \UU_{g,n}$ taking
\[
  (\nu : \tilde{C} \to C, p_1, \ldots, p_n) \mapsto (C, \nu(p_1),\ldots, \nu(p_n))
\]
is finite-to-one on isomorphism classes of $k$ points, and the restriction of $\EE_{g,n}(k) \to \UU_{g,n}(k)$ to the locus of smooth markings is an equivalence of categories.
\end{proposition}
\begin{proof}
Given a curve $(C, q_1, \ldots, q_n)$ in $\UU_{g,n}(k)$, let $\nu : \tilde{C} \to C$ be its normalization and let $p_1, \ldots, p_n$ be lifts of $q_1, \ldots, q_n$ to $\tilde{C}$, respectively. Then $(\nu : \tilde{C} \to C, p_1, \ldots, p_n)$
is a point of $\EE_{g,n}(k)$ mapping to $(C, q_1, \ldots, q_n)$, so $\EE_{g,n}(k) \to \UU_{g,n}(k)$ is essentially surjective. Since preimages of smooth markings remain smooth, the restriction of the functor to the subcategories of curves with smooth markings is also essentially surjective.

Now suppose that $(\nu' : \tilde{C}' \to C, p_1', \ldots, p_n')$ is a second preimage of $(C, q_1, \ldots, q_n)$. By the universal property of normalization, there is a unique isomorphism $\phi : \tilde{C}' \to \tilde{C}$ over $C$.
Note that $\phi(p_1'), \ldots, \phi(p_n')$ are lifts of $q_1, \ldots, q_n$ to $\tilde{C}$. Since $\nu$ is finite, there are at most finitely many choices of lifts. Therefore $(\nu' : \tilde{C}' \to C, p_1', \ldots, p_n')$ is isomorphic to one of the finitely many normalized curves of the form $(\nu : \tilde{C} \to C, p_1'', \ldots, p_n'')$ where $p_i''$ is a lift of $q_i$ for each $i$. We conclude $\EE_{g,n}(k) \to \UU_{g,n}(k)$ is finite-to-one on isomorphism classes.

Finally we check full faithfulness of the restriction of $\EE_{g,n}(k) \to \UU_{g,n}(k)$ to the subcategories in which markings are smooth. Suppose $(C, q_1, \ldots, q_n)$ and $(C', q_1', \ldots, q_n')$ lie in $\UU_{g,n}(k)$ and the markings $q_i$, $q_i'$ are smooth. Choose respective preimages $(\nu : \tilde{C} \to C, p_1, \ldots, p_n)$ and $(\nu' : \tilde{C}' \to C', p_1',\ldots,p_n')$
in $\EE_{g,n}(k)$. Since $\nu, \nu'$ restrict to an isomorphism in the smooth locus, the choice of lifts $p_1, \ldots, p_n$ and $p_1', \ldots, p_n'$ are unique. If there is an isomorphism $\psi : (C, q_1, \ldots, q_n) \to (C', q_1', \ldots, q_n')$, then by the universal property of normalization,
it admits a unique lift to a morphism $\tilde{\psi} : \tilde{C} \to \tilde{C}'$. Since the lifts of the $q_i$ and $q_i'$ are unique it is automatic that $\tilde{\psi}$ takes $p_1, \ldots, p_n$ to $p_1', \ldots, p_n'$, so $\tilde{\psi}$ defines a morphism $(\nu : \tilde{C} \to C, p_1, \ldots, p_n) \to (\nu' : \tilde{C}' \to C', p_1', \ldots, p_n')$. Thus, morphisms $\psi : C \to C'$ lift uniquely to $\EE_{g,n}(k)$, that is, the restricted functor is fully faithful. Since the restriction
of $\EE_{g,n}(k) \to \UU_{g,n}(k)$ to the locus of smooth markings is both fully faithful and essentially surjective, it is an equivalence of categories.
\end{proof}

\begin{example} \label{ex:egn_not_stratification_positive_char}
We show that the forgetful map $\EE_{g,n} \to \UU_{g,n}$ may be ramified in positive characteristic, following \cite[Example 1.76]{fred_thesis}.

Let $k$ be an algebraically closed field of characteristic 3. Let $\nu : \tilde{C} \to C$ be the normalization of the cuspidal cubic, given locally by $k[x,y]/(y^2 - x^3) \to k[t]$ with $x \mapsto t^2, y \mapsto t^3$. This is a normalized curve of genus 1 with $\delta = 1$ and $c = 2$. Let $\xi : \Spec k \to \UU_{1,0}$ be the classifying map for $C$. The fiber product $\mathfrak{X} = \EE^{1,2}_{1,0} \times_{\UU_{1,0},\xi} \Spec k$ is the algebraic stack whose category of $S$-points has objects the equinormalized curves $\nu_1 : \tilde{C}_1 \to C \times S$ and has morphisms the commutative diagrams
\[
\begin{tikzcd}
  \tilde{C}_1 \ar[r, "\beta"] \ar[d, "\nu_1"] & \tilde{C}_2 \ar[d, "\nu_2"] \\
  C \times_{\Spec k} S \ar[r, equals] & C \times_{\Spec k} S.
\end{tikzcd}
\]

By Proposition \ref{prop:equinormalized_to_ugn}, $\mathfrak{X}$ possesses only one $k$-point, namely $\nu : \tilde{C} \to C$. In order to show $\EE^{1,2}_{1,0} \to \UU_{1,0}$ is ramified, we find non-isomorphic tangent vectors to this point of $\mathfrak{X}$. The trivial tangent vector is represented by the pullback of $\nu$ to $\Spec k[\epsilon]/\epsilon^2$; call it $\nu_1 : \tilde{C}_1 \to C \times_{\Spec k} \Spec k[\epsilon]/\epsilon^2$. Now, observe that $C \times_k \Spec k[\epsilon]/\epsilon^2$ possesses an automorphism $\alpha$ given locally by $x \mapsto x + \epsilon, y \mapsto y$. Let $\nu_2 = \alpha \circ \nu_1$. It is clear that $\nu_2$ also defines a tangent vector to the unique $k$-point of $\mathfrak{X}$. It is not isomorphic to $\nu_1$ since an isomorphism $\beta$ as in the commutative diagram above would have to take $t^2 \mapsto t^2 + \epsilon$ and $t^3 \mapsto t^3$. But then $\beta^\sharp(t) = t - \epsilon t^{-1}$, which is not regular.
\end{example}

\section{Stratification of the Hilbert scheme of points of a smooth curve}
\label{sec:hilbert_stratification}
Let $\pi : C \to S$ be a family of smooth, possibly disconnected curves over a scheme $S$ and let $m$ be a positive integer. In this section we explain how to stratify the Hilbert scheme of $m$ points in $C/S$ according to the integer partitions of $m$. Later, taking $m$ equal to the total conductance, this stratification will allow us to fix the branch conductances of singular curves. In order for this stratification to behave well, we will later introduce the assumption that $S$ is characteristic 0. 

Let us introduce some notation. The $S$-scheme $C^m = C \times_S \cdots \times_S C$ parametrizes families of $m$ ordered points of $C$. It possesses a natural action of the symmetric group on $m$ letters $\mathfrak{S}_m$ by permuting the $m$ coordinates.\label{sigma_m_first_ref} The Hilbert scheme $\Hilb^m_{C/S}$ parametrizes closed subschemes of $C$ of rank $m$ over $S$; it is well-known that since $C$ is a smooth curve, it is isomorphic to the symmetric product $C^{(m)}$ parametrizing $m$ unordered points of $C$. The symmetric product is obtained affine locally by taking Spec of the $\mathfrak{S}_m$-invariant functions on $C^m$ and it is a geometric quotient of $C^m$ by $\mathfrak{S}_m$. The natural quotient map $\tau : C^m \to C^{(m)}$ is finite locally free of rank $m!$.

There are various diagonal closed subschemes $\{ \Diag_{\mathcal{P}} \}$ of $C^m$ indexed by set partitions $\mathcal{P}$ of $[m] = \{1, \ldots, m\}$, namely:
\begin{equation} \label{eq:d_tilde_p}
  \Diag_{\mathcal{P}} \coloneqq \{ (x_1, \ldots, x_m) \in C^m \mid i \sim_{\mathcal{P}} j \implies x_{i} = x_j \}.
\end{equation}
Observe that for any two partitions $\mathcal{Q}, \mathcal{P}$ of $[m]$, we have $\Diag_{\mathcal{Q}} \subseteq \Diag_{\mathcal{P}}$ if and only if $\mathcal{Q}$ is a coarsening of $\mathcal{P},$ which we denote by $\mathcal{Q} \preceq \mathcal{P}$. Thus, there is an induced locally closed stratification $\{ \Diag_{\mathcal{P}}^\circ \}_{\mathcal{P}}$ given by
\begin{equation} \label{eq:d_tilde_circ_p}
  \Diag_{\mathcal{P}}^\circ = \Diag_{\mathcal{P}} - \bigcup_{\mathcal{Q} \prec \mathcal{P}} \Diag_{\mathcal{Q}}.
\end{equation}

We find an analogous stratification in $C^{(m)}$ by taking images. Given a set partition $\mathcal{P} = \{ P_1, \ldots, P_k \}$ of $[m]$, denote by $[\mathcal{P}]$ the integer partition $|P_1| + \cdots + |P_k| = m$. We define symmetrized diagonal loci in $C^m$ by
\begin{equation} \label{eq:d_tilde_symmetrized}
  \Diag_{[\mathcal{P}]} = \bigcup_{\sigma \in \mathfrak{S}_m} \sigma(\Diag_{\mathcal{P}}) = \bigcup_{\sigma \in \mathfrak{S}_m} \Diag_{\sigma(\mathcal{P})}.
\end{equation}
For each such $[\mathcal{P}]$, we define a closed subscheme
\begin{equation} \label{eq:d_p}
  \imDiag_{[\mathcal{P}]} \coloneqq \tau(\Diag_{\mathcal{P}}) = \tau(\Diag_{[\mathcal{P}]}) \subseteq C^{(m)},
\end{equation}
of $C^{(m)}$ where $\tau(\Diag_{\mathcal{P}})$ denotes the scheme-theoretic image. As before, $\imDiag_{[\mathcal{Q}]} \subseteq \imDiag_{[\mathcal{P}]}$ if and only if $[\mathcal{Q}]$ is a coarsening of $[\mathcal{P}]$, which we denote by $[\mathcal{Q}] \preceq [\mathcal{P}]$. Hence, there is a locally closed stratification $\{ \imDiag_{[\mathcal{P}]}^\circ \}_{[\mathcal{P}]}$ of $C^{(m)}$ given by
\begin{equation} \label{eq:d_circ_p}
  \imDiag_{[\mathcal{P}]}^\circ = \imDiag_{[\mathcal{P}]} - \bigcup_{[\mathcal{Q}] \prec [\mathcal{P}]} \imDiag_{[\mathcal{Q}]}.
\end{equation}

Next, observe that $\Diag_{\mathcal{P}}$ is preserved by the subgroup
\begin{align*}
  G_{\mathcal{P}} &= \{ \sigma \in \mathfrak{S}_m \mid \sigma(\Diag_{\mathcal{P}}) \subseteq \Diag_{\mathcal{P}} \} \\
                 &= \{ \sigma \in \mathfrak{S}_m \mid i \sim_{\mathcal{P}} j \iff \sigma(i) \sim_{\mathcal{P}} \sigma(j) \}
\end{align*}
of $\mathfrak{S}_m$. There is an induced action of $G_{\mathcal{P}}$ on both $\Diag_{\mathcal{P}}$ and $\Diag^\circ_{\mathcal{P}}$. The stabilizer of any point in $\Diag^\circ_{\mathcal{P}}$ is
\begin{align*}
  K_{\mathcal{P}} = \{ \sigma \in \mathfrak{S}_m \mid \sigma(P) \subseteq P \text{ for }P \in \mathcal{P} \}.
\end{align*}
It follows that there is a faithful action on $\Diag^\circ_{\mathcal{P}}$ by $G_{\mathcal{P}} / K_{\mathcal{P}}$.
We can identify $G_{\mathcal{P}}/K_{\mathcal{P}}$ with the group
\begin{equation} \label{eq:sigma_p}
  \mathfrak{S}_{\mathcal{P}} = \{ \sigma \in \mathrm{Sym}(\mathcal{P}) : |\sigma(P_i)| = |P_i| \text{ for }P_i \in \mathcal{P}\}.
\end{equation}
which permutes the parts of $\mathcal{P}$ with the same size. Under this identification, $\sigma \in \mathfrak{S}_{\mathcal{P}}$ acts on $\Diag_{\mathcal{P}}$ by taking the coordinates with indices in a part $P$ to the coordinates with indices in $\sigma(P)$.

One might naively expect that $\Diag^{\circ}_{\mathcal{P}} \to \imDiag_{[\mathcal{P}]}^\circ$ is the geometric quotient by $\mathfrak{S}_{\mathcal{P}}$, as this is the case on closed points. However, this does not hold in positive characteristic. Another difficulty is that the formation of $\imDiag_{[\mathcal{P}]}$ may not commute with base change. The following example illustrates both difficulties.

\begin{example} \label{ex:positive_characteristic_bad}
Let $C = \AA^1$, $S = \Spec \ZZ$. Then $C^2 \cong \Spec \ZZ[x,y]$ and $C^{(2)} \cong \Spec \ZZ[b,c]$ and $\tau : C^2 \to C^{(2)}$ is induced by the homomorphism $\ZZ[b,c] \to \ZZ[x,y]$ sending $b \mapsto -(x+y)$ and $c \mapsto xy$. Write $\mathcal{P} = 12$ for the indiscrete partition of $\{1,2\}$. Then $\Diag_{12} = \Diag_{12}^\circ = V(x-y)$. One checks that the scheme-theoretic image is cut out by the degree two discriminant, i.e., $\imDiag_{[12]} = \imDiag_{[12]}^\circ = V(b^2-4c)$. Note that the base change to $\FF_2$ is $\imDiag_{[12]} \otimes \FF_2 \cong V(b^2)$.

Now consider $C' = C \times \Spec \FF_2$ and $S' = \Spec \FF_2$ and denote the data associated to $C'/S'$ by primes. Then we have $(C')^2 \cong \Spec \FF_2[x,y]$ and $(C')^{(2)} \cong \Spec \FF_2[b,c]$ and $\tau : (C')^2 \to (C')^{(2)}$ is induced by the homomorphism $\FF_2[b,c] \to \FF_2[x,y]$ sending $b \mapsto -(x+y)$ and $c \mapsto xy$. As before, $\Diag_{12}' = V(x-y)$. However, the scheme-theoretic image is $\imDiag_{[12]}' = V(b)$, not $V(b^2)$. This shows that we do not have commutativity with base change.

Next, consider the map $\Diag_{12}' \to \imDiag_{[12]}'$. Our expectation is that it is the quotient by the trivial group, i.e., an isomorphism. However, it is induced by the ring homomorphism $\FF_2[c] \to \FF_2[x]$ sending $c \mapsto x^2$. This is not an isomorphism.
\end{example}

\begin{theorem} \label{thm:hilbert_stratification}
Suppose $S$ is a scheme over $\QQ$. Then
\begin{enumerate}
  \item The induced map $\tau : \Diag_{\mathcal{P}}^\circ \to \imDiag_{[\mathcal{P}]}^\circ$ is finite, surjective, \'etale, and it is both the scheme-theoretic quotient by $\mathfrak{S}_{\mathcal{P}}$ and the stack-theoretic quotient by $\mathfrak{S}_{\mathcal{P}}$.
  \item Formation of the locally closed subschemes $\Diag_{\mathcal{P}}^\circ$ and $\imDiag_{[\mathcal{P}]}^\circ$ commutes with base change.
\end{enumerate}
\end{theorem}

\begin{proof}
The map $\tau$ is finite and surjective by construction. Since the invariant functions under the actions of $G_{\mathcal{P}}$ and $\mathfrak{S}_{\mathcal{P}}$ are the same, $\Diag_{\mathcal{P}}^\circ / \mathfrak{S}_{\mathcal{P}} = \Diag_{\mathcal{P}}^\circ / G_{\mathcal{P}}$.
By \cite[Proposition A.8]{mustata_zeta}, the natural map $\Diag_{\mathcal{P}}^\circ / G_{\mathcal{P}} \to \left(\bigsqcup_{\sigma \in \mathfrak{S}_m} \Diag_{\sigma(\mathcal{P})}^\circ \right) / \mathfrak{S}_m$
is an isomorphism.
By \cite[Remark A.27]{mustata_zeta}, since $S$ has characteristic 0, the natural map $\left(\bigsqcup_{\sigma \in \mathfrak{S}_m} \Diag_{\sigma(\mathcal{P})}^\circ \right) / \mathfrak{S}_m \to \imDiag_{[\mathcal{P}]}^\circ$
is an isomorphism. Composing the isomorphisms, we conclude that $\tau : \Diag_{\mathcal{P}}^\circ \to \imDiag_{[\mathcal{P}]}^\circ$ is the quotient by $\mathfrak{S}_{\mathcal{P}}$. Since the action of $\mathfrak{S}_{\mathcal{P}}$ is free, $\tau$ is \'etale \cite[Expos\'e V, Corollaire 2.3]{sga1} and agrees with the stack quotient \cite[Tag 07S7]{stacks-project}. We conclude that (i) holds.

We now turn to the proof of (ii). Let $f : S' \to S$ be a morphism of $\QQ$-schemes. Let $C' = C \times_S S'$ and denote with primes all the data formed from $C'/S'$. It is clear that $\Diag_{\mathcal{P}}' = \Diag_{\mathcal{P}} \times_S S'$, so $(\Diag_{\mathcal{P}}^\circ)' = \Diag_{\mathcal{P}}^\circ \times_S S'$. The natural map $(C')^{(m)} \to C^{(m)} \otimes_S S'$ is an isomorphism, since the corresponding statement for Hilbert schemes is obvious. To conclude the proof it suffices to show that $\imDiag_{[\mathcal{P}]}' \cong \imDiag_{[\mathcal{P}]} \times_S S'$ for $S, S'$ affine.

Let $S = \Spec B$, $S' = \Spec B'$. Let $U \subseteq C^{(m)}$ be an affine open subscheme. Let $A = \Gamma(\tau^{-1}(U), \mathscr{O}_{C^m})$ and write $G$ for $\mathfrak{S}_m$. Then $A^G = \Gamma(U, \mathscr{O}_{C^{(m)}})$. Write $I$ for the ideal of $\Diag_{[\mathcal{P}]}$ on $\tau^{-1}(U)$. We need to show that the natural map $(A/I)^G \otimes_B B' \to (A \otimes_B B' / I \otimes_B B')^G$ is an isomorphism. Since $G$ is finite and we are working in characteristic 0, taking $G$ invariants is exact \cite[Proposition 6.1.10]{weibel}. Then taking $G$ invariants commutes with tensor products \cite[Exercise 2.4.2]{weibel}, so we have the result.
\end{proof}

\begin{remark}
Example \ref{ex:positive_characteristic_bad} shows that the result above cannot be extended to $\ZZ.$ If it could be shown both that $\tau$ is finite flat and that (ii) holds for schemes over $\FF_p$, then our main results should extend to positive characteristic with the \'etale topology replaced by the fppf topology.
\end{remark}

\section{Combinatorial types}
\label{sec:combinatorial_type}
The notion of a dual graph plays an essential role in the study of moduli of nodal curves. We now give a full definition of the combinatorial type of an equinormalized curve which generalizes the notion of a dual graph
 to an arbitrary singular curve.

\begin{definition} \label{def:combinatorial_type}
An ($n$-pointed) \emphbf{combinatorial type} $\Gamma$ consists of
\begin{enumerate}
  \item A finite set $V = S \sqcup K$ of \emphbf{vertices}, made up of a set $S$ of \emphbf{singularities} and a non-empty set $K$ of \emphbf{irreducible components};
  \item A finite set $B$ of \emphbf{edges} or \emphbf{branches};
  \item A finite set $D$ of \emphbf{distinguished points};
  \item Incidence functions:
    \begin{enumerate}
      \item $\sigma : B \to S$, which is surjective;
      \item $\tau : B \to K$;
      \item $\epsilon : D \to K$
    \end{enumerate}
  \item A marking function $\mu : \{ 1, \ldots, n \} \to D \sqcup B$ with image containing $D$;
  \item A genus function $g : V \to \ZZ_{\geq 0}$;
  \item A branch conductance function $c : B \to \ZZ_{\geq 1}$
\end{enumerate}
\end{definition}

We remark that the ``distinguished points'' are intended to model the markings of $\tilde{C}$ away from the preimages of singularities. Markings $p_i$ of $\tilde{C}$ over singularities of $C$
are modeled by setting $\mu(i) \in B$ instead.

\begin{definition}
An \emphbf{isomorphism} of combinatorial types $\phi : \Gamma \to \Gamma'$ consists of (using primes for the data of $\Gamma'$)
bijections $\phi_S : S \to S'$, $\phi_K : K \to K'$, $\phi_B : B \to B'$, $\phi_D : D \to D'$ commuting
with $\sigma, \tau, \mu, \epsilon, g, c$.
\end{definition}

\begin{remark}
It is an interesting and non-obvious problem which combinatorial types are specializations of a fixed combinatorial type $\Gamma.$ Relatedly, can the definition above be generalized to a notion of
morphism of combinatorial types that would play the role of specializations of dual graphs?
\end{remark}

Each $n$-pointed equinormalized curve $\nu : (\tilde{C}; p_1, \ldots, p_n) \to C$ of genus $g$ with invariants $\delta,c$ has an associated $n$-pointed combinatorial type $\Gamma(\tilde{C} \to C)$, constructed as follows.

\begin{itemize}
  \item $S$ is the set of singular points of $C$;
  \item $K$ is the set of connected components of $\tilde{C}$;
  \item $B = \nu^{-1}(S)$;
  \item $\sigma$ is the restriction of $\nu$ from $B$ to $S$;
  \item $\tau$ is the map induced by the inclusion of $B$ in $\tilde{C}$;
  \item $D$ is the set $\{ p_1, \ldots, p_n \} - \nu^{-1}(S)$;
  \item $\mu$ is the map $i \mapsto p_i$;
  \item $\epsilon$ is the map taking $p \in D$ to the connected component containing $p$;
  \item $g|_S$ is given by taking a singularity to its genus;
  \item $c$ is given by taking $b \in B$ to the branch conductance $\dim_{k} (\OO_{\tilde{C}} / \Cond_{\tilde{C}/C})_b$.
  \item $g|_K$ is given by taking a component of $\tilde{C}$ to its arithmetic genus.
\end{itemize}

\begin{example} (Combinatorial type of a nodal curve) \label{ex:comb_type_of_prestable}
The combinatorial type of a nodal curve is essentially a subdivision of its dual graph at the midpoints of its edges. For example, consider an irreducible rational curve $C$ with two self-nodes. Its dual graph would consist of two loops attached to a single vertex. The combinatorial type $\Gamma$ of its normalization $\nu : \tilde{C} \to C$ is as depicted below:

\begin{center}
\begin{tikzpicture}[scale=1.5]
\draw (0,0) to[out=40,in=140]node[above]{\tiny$b_1$} (1,0)
      (0,0) to[out=-40,in=220]node[below]{\tiny$b_2$} (1,0)
      (1,0) to[out=40,in=140]node[above]{\tiny$b_3$} (2,0)
      (1,0) to[out=-40,in=220]node[below]{\tiny$b_4$} (2,0);
\draw[fill=white]
  (0,0) +(-.2,-.2) rectangle +(.2,.2)
  (1,0) circle (5pt)
  (2,0) +(-.2,-.2) rectangle +(.2,.2);
\node at (0,0) {\tiny$s_1$};
\node at (1,0) {\tiny$k_1$};
\node at (2,0) {\tiny$s_2$};
\end{tikzpicture}
\end{center}

Explicitly, the data of $\Gamma$ are
\[
  K = \{ k_1 \}, \quad S = \{ s_1, s_2 \}, \quad B = \{ b_1, b_2, b_3, b_4 \}, \quad D = \emptyset
\]
and
\begin{align*}
  \sigma(b_1) &= \sigma(b_2) = s_1, \quad \tau(b_i) = k_1, \quad g(k_1) = g(s_i) = 0, \\
  \sigma(b_3) &= \sigma(b_4) = s_2, \quad c(b_i) = 1.
\end{align*}

The automorphisms of $\Gamma$ are determined by their action on $B$ (see Lemma \ref{lem:isos_determined_by_branches}), so we may identify $\Aut(\Gamma)$
with a subgroup of $S_4$. As $\Aut(\Gamma)$ contains the automorphism taking $b_1 \mapsto b_3 \mapsto b_2 \mapsto b_4 \mapsto b_1$, and the automorphism
taking $b_1 \mapsto b_2 \mapsto b_1, b_3 \mapsto b_3, b_4 \mapsto b_4$, but not $b_1 \mapsto b_3 \mapsto b_1, b_2 \mapsto b_2, b_4, \mapsto b_4$, we may identify
$\Aut(\Gamma)$ with the subgroup $\langle (1324), (12) \rangle \leq S_4$, isomorphic to the dihedral group of order 8.
\end{example}

\begin{example} \label{ex:comb_type_markings} (Markings in combinatorial types)
Consider a marked normalized curve $(\nu : \tilde{C} \to C, p_1, p_2, p_3, p_4)$ as depicted below

\begin{center}
\includegraphics[width=2.5in]{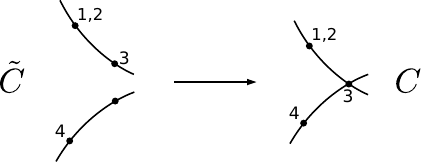}
\end{center}

The curve $\tilde{C}$ possesses 4 marked points, two of which coincide, and one of which maps to a node of $C$. The combinatorial type
$\Gamma$ of this normalized curve, with markings above edges, is:

\begin{center}
\begin{tikzpicture}[scale=1.5]
\draw (0,0) 
      --node[below]{\tiny $d_1$} node[above]{\tiny  1,2}
      (1,0)
      --node[above]{\tiny  3} node[below]{\tiny $b_1$}
      (2,0)
      --node[below]{\tiny $b_2$}
      (3,0)
      --node[below]{\tiny $d_2$} node[above]{\tiny  4}
      (4,0);
\draw[fill=white]
  (1,0) circle (5pt)
  (2,0) +(-.2,-.2) rectangle +(.2,.2)
  (3,0) circle (5pt);
\node at (1,0) {\tiny$k_1$};
\node at (2,0) {\tiny$s_1$};
\node at (3,0) {\tiny$k_2$};
\end{tikzpicture}
\end{center}

Explicitly, the data of $\Gamma$ related to the markings are
\begin{align*}
  K &= \{ k_1, k_2 \}, \quad S = \{ s_1 \}, \quad B = \{ b_1, b_2 \}, \quad D = \{ d_1, d_2 \} \\
\intertext{and}
  \mu(1) &= d_1, \quad \mu(2) = d_1, \quad \mu(3) = b_1, \quad \mu(4) = d_2 \\
  \epsilon(d_1) &= k_1, \quad \epsilon(d_2) = k_2.
\end{align*}
\end{example}

\begin{example} \label{ex:comb_type_fancy_sings} (A combinatorial type with more complex singularities)
Consider a curve $C$ with three rational components,
a ramphoid cusp, a cusp meeting a transverse smooth branch, and a planar 3-fold point, as depicted below:

\begin{center}
\includegraphics[width=2in]{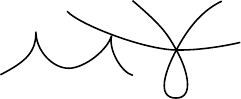}
\end{center}

The ramphoid cusp on the left-most component is unibranch, has genus 2, and conductance $4$. The cusp meeting a transverse
smooth branch has two branches, genus 1, and conductances 1 and 2. The planar 3-fold point has genus 1, three branches, and
conductances 2.

The combinatorial type $\Gamma$ of the normalization map $\nu : \tilde{C} \to C$, with genera above vertices and conductances above edges, is:

\begin{center}
\begin{tikzpicture}[scale=1.5]
\draw (0,0)
      --node[below]{\tiny $b_1$} node[above]{\tiny  $4$}
      (1,0) 
      -- node[below]{\tiny $b_2$} node[above]{\tiny  2}
      (2,0)
      --node[below]{\tiny $b_3$} node[above]{\tiny  1}
      (3,0)
      --node[below]{\tiny $b_4$} node[above]{\tiny  2}
      (4,0);
\draw (4,0) to[out=40,in=140]node[above]{\tiny 2} node[below]{\tiny $b_5$} (6,0)
      (4,0) to[out=-40,in=220]node[above]{\tiny 2} node[below]{\tiny $b_6$} (6,0);
\draw[fill=white]
  (0,0) +(-.2,-.2) rectangle +(.2,.2)
  (1,0) circle (5pt)
  (2,0) +(-.2,-.2) rectangle +(.2,.2)
  (3,0) circle (5pt)
  (4,0) +(-.2,-.2) rectangle +(.2,.2)
  (6,0) circle (5pt);
\node at (0,0) {\tiny$s_1$};
\node at (0,0.35) {\tiny  2};
\node at (1,0) {\tiny$k_1$};
\node at (1,0.35) {\tiny  0};
\node at (2,0) {\tiny$s_2$};
\node at (2,0.35) {\tiny  1};
\node at (3,0) {\tiny$k_2$};
\node at (3,0.35) {\tiny  0};
\node at (4,0) {\tiny$s_3$};
\node at (4,0.35) {\tiny  1};
\node at (6,0) {\tiny$k_3$};
\node at (6,0.35) {\tiny  0};
\end{tikzpicture}
\end{center}
\end{example}

The combinatorial type $\Gamma$ has sets
\[
  K = \{ k_1, k_2, k_3 \}, \quad S = \{ s_1, s_2, s_3 \}, \quad B = \{ b_1, \ldots, b_6 \}, \quad D = \emptyset.
\]
The data assigned to the branches of $\Gamma$ are
\begin{align*}
  \sigma(b_1) &= s_1, \quad \tau(b_1) = k_1, \quad c(b_1) = 4, \\
  \sigma(b_2) &= s_2, \quad \tau(b_2) = k_1, \quad c(b_2) = 2, \\
  \sigma(b_3) &= s_2, \quad \tau(b_3) = k_2, \quad c(b_3) = 1, \\
  \sigma(b_4) &= s_3, \quad \tau(b_4) = k_2, \quad c(b_4) = 2, \\
  \sigma(b_5) &= s_3, \quad \tau(b_5) = k_3, \quad c(b_5) = 2, \\
  \sigma(b_6) &= s_3, \quad \tau(b_6) = k_3, \quad c(b_6) = 2.
\end{align*}
The genera of singularities are
\[
  g(s_1) = 2, \quad g(s_2) = 1, \quad g(s_3) = 1.
\]
The genera of the components are all 0:
\[
  g(k_i) = 0.
\]

\begin{definition} \label{def:comb_type_addition_notions}
A combinatorial type is said to be \emphbf{connected} if its underlying bipartite graph is connected.

The \emphbf{genus} of a connected combinatorial type $\Gamma$, with notation as above, is
\[
  g(\Gamma) = \sum_{v \in V} g(v) + \#B - \#V + 1.
\]
Given a singularity $s \in S$, we say its \emphbf{$\delta$-invariant} is
\[
  \delta(s) = g(s) + \val(s) - 1,
\]
where the \emphbf{valence} $\val(s)$ is the number of branches $b$ with $\sigma(b) = s$. We can rewrite the formula for genus in terms of the delta invariant as
\[
  g(\Gamma) = \sum_{k \in K} g(k) + \sum_{s \in S} \delta(s) - \# K + 1.
\]
We say that the \emphbf{total $\delta$-invariant} of $\Gamma$ is $\delta(\Gamma) = \sum_{s \in S} \delta(s)$, and the  \emphbf{total conductance} of $\Gamma$ is $c(\Gamma) = \sum_{b \in B} c(b).$
\end{definition}

We will also have need of the part of the combinatorial type determined solely by $\tilde{C}$ and the vanishing of the conductor. This amounts to deleting the set $S$
and incidence function $\sigma$.

\begin{definition} \label{def:c_tilde_type}
An ($n$-pointed) \emphbf{$\tilde{C}$-type} $\Gtil$ consists of
\begin{enumerate}
  \item A non-empty set $K$ of \emphbf{vertices} or \emphbf{irreducible components};
  \item A finite set $B$ of \emphbf{edges} or \emphbf{branches};
  \item A finite set $D$ of \emphbf{distinguished points};
  \item Incidence functions:
    \begin{enumerate}
      \item $\tau : B \to K$;
      \item $\epsilon : D \to K$
    \end{enumerate}
  \item A marking function $\mu : \{ 1, \ldots, n \} \to D \sqcup B$ with image containing $D$;
  \item A genus function $g : K \to \ZZ_{\geq 0}$;
  \item A branch conductance function $c : B \to \ZZ_{\geq 1}$
\end{enumerate}

Given a combinatorial type $\Gamma$, the \emphbf{underlying $\tilde{C}$-type of $\Gamma$}, written $\Gtil$, is the $\tilde{C}$-type obtained by forgetting $S$ and $\sigma$.
\end{definition}

The $\tilde{C}$-type of a tuple $(C, (p_i), Z)$, where $(C, (p_i))$ is a smooth $n$-pointed, possibly disconnected curve over an algebraically closed field $k$ and $Z$ is a finite subscheme of $C$, is constructed in the same manner
as the combinatorial type of a normalized curve, except

\begin{itemize}
  \item $B$ is the set of points of $Z$;
  \item $c$ is given by taking $b \in B$ to the multiplicity of $Z$ at $b$, $\dim_k (\OO_Z)_b$.
\end{itemize}

\begin{definition}
An \emphbf{isomorphism} of $\tilde{C}$-types $\phi : \Gtil \to \Gtil'$ consists of (using primes for the data of $\Gtil'$)
bijections $\phi_K : K \to K'$, $\phi_B : B \to B'$, $\phi_D : D \to D'$ commuting
with $\tau, \mu, \epsilon, g, c$.
\end{definition}

\begin{example} \label{ex:c_tilde_type}
The combinatorial type $\Gamma$ of Example \ref{ex:comb_type_of_prestable} has $\tilde{C}$-type $\Gtil$ as below:

\begin{center}
\begin{tikzpicture}[scale=1.5]
\draw (1,0) --node[above]{\tiny$b_1$} (.2,.5)
      (1,0) --node[below]{\tiny$b_2$} (.2,-.5)
      (1,0) --node[above]{\tiny$b_3$} (1.8,.5)
      (1,0) --node[below]{\tiny$b_4$} (1.8,-.5);
\draw[fill=white]
  (1,0) circle (5pt);
\node at (1,0) {\tiny$k_1$};
\end{tikzpicture}
\end{center}

Explicitly, the data of $\Gtil$ are
\[
  K = \{ k_1 \}, \quad B = \{ b_1, b_2, b_3, b_4 \}, \quad D = \emptyset
\]
and
\[
  \tau(b_i) = k_1, \quad g(k_1) = 0, \quad c(b_i) = 1.
\]

The automorphisms of $\Gtil$ are determined by their action on $B$ (see Lemma \ref{lem:isos_determined_by_branches}), so we may identify $\Aut(\Gtil)$
with a subgroup of $S_4$. As any permutation of $B$ induces an automorphism, $\Aut(\Gtil)$ is identified with all of $S_4$.
\end{example}

The example above suggests that $\Aut(\Gamma) \leq \Aut(\Gtil)$. This is true in general.

\begin{lemma} \label{lem:aut_gamma_in_aut_gtil}
Let $\Gamma$ be an $n$-marked combinatorial type. Then there is a natural injective group homomorphism $\Aut(\Gamma) \to \Aut(\Gtil)$ induced by taking
$(\phi_K, \phi_S, \phi_B, \phi_D) \mapsto (\phi_K, \phi_B, \phi_D)$.
\end{lemma}
\begin{proof}
It is clear that the map respects composition. To show injectivity, suppose $\phi$ is an automorphism of $\Gamma$ mapping to the identity. Then $\phi_K, \phi_B,$ and $\phi_D$ are identity maps, and it only remains to show $\phi_S$ is the identity. Let $s \in S$. Then since $\sigma : B \to S$ is surjective, we may choose $b \in B$ such that $\sigma(b) = s$. Then $\phi_S(s) = \phi_S(\sigma(b)) = \sigma(\phi_B(b)) = \sigma(b) = s,$ as desired.
\end{proof}

We will also later consider situations in which the branches come with a numbering. In such cases, we may assign a more rigid notion of combinatorial type or $\tilde{C}$-type, where the branches cannot be exchanged by automorphisms. We will impose a connectedness hypothesis since the only combinatorial types considered in such situations are connected.

\begin{definition} \label{def:rigid_combinatorial_type}
A \emphbf{rigidified combinatorial type} $(\Gamma, \beta)$ \emphbf{with conductances $\vec{c}$} consists of a connected $n$-marked combinatorial type $\Gamma$ and a bijection $\beta : \{1, \ldots, m \} \to B_\Gamma$ such that $c(\beta(i)) = c_i$ for all $i$. An \emphbf{isomorphism of rigidified combinatorial types} $\phi : (\Gamma, \beta) \to (\Gamma', \beta')$
is an isomorphism of underlying combinatorial types such that $\phi_B \circ \beta = \beta'.$

Analogously, a \emphbf{rigidified $\tilde{C}$-type} $(\Gtil, \beta)$ \emphbf{with conductances $\vec{c}$} consists of an $n$-marked $\tilde{C}$-type and a bijection $\beta : \{ 1, \ldots, m \} \to B_{\Gtil}$ such that $c(\beta(i)) = c_i$ for all $i$. An \emphbf{isomorphism of rigidified $\tilde{C}$-types} $\phi : (\Gtil, \beta) \to (\Gtil', \beta')$
is an isomorphism of underlying $\tilde{C}$-types such that $\phi_B \circ \beta = \beta'.$
\end{definition}

\begin{lemma} \label{lem:isos_determined_by_branches}
Let $\Gamma$ be a connected $n$-marked combinatorial type. Then:
\begin{enumerate}
  \item An automorphism $\phi$ of $\Gamma$ is determined by its restriction $\phi_B : B \to B$.
  \item An automorphism $\phi$ of $\Gtil$ is determined by its restriction $\phi_B : B \to B$.
\end{enumerate}
\end{lemma}
\begin{proof}
It suffices to prove (ii) by Lemma \ref{lem:aut_gamma_in_aut_gtil} since an automorphism $\phi$ of $\Gamma$ has the same restriction to branches as the induced automorphism of $\Gtil$. Write $\Gamma = (V = S \sqcup K, B, D, \sigma, \tau, \mu, \epsilon, g, c)$.
Suppose $\phi$ and $\psi$ are automorphisms of its underlying $\tilde{C}$-type $\Gtil$ such that $\phi_B = \psi_B$. We must show $\phi_D = \psi_D$, and $\phi_K = \psi_K$.

($\phi_D = \psi_D$) Let $d \in D$ be a distinguished point. Since $D \subseteq \im(\mu)$, there is some $i \in \{1, \ldots, n\}$ such that $\mu(i) = d.$ Then by definition of isomorphism
of $\tilde{C}$-type, $\phi_D(d) = \phi_D(\mu(i)) = \mu(i) = \psi_D(\mu(i)) = \psi_D(d).$

($\phi_K = \psi_K$) If $K$ is a 1-element set, then $\phi_K = \psi_K$ is the unique function $K \to K$. Otherwise, let $k, k_2 \in K$ be distinct elements of $K$. Let $b \in B$
be the first edge on a path from $k$ to $k_2$ in $\Gamma$ (not $\Gtil$). Then $\tau(b) = k$ and $\phi_K(k) = \phi_K(\tau(b)) = \tau(\phi_B(b)) = \tau(\psi_B(b)) = \psi_K(\tau(b)) = \psi_K(k)$.
\end{proof}

\begin{corollary}
Let $(\Gamma, \beta)$ be a rigidified combinatorial type. Then both $\Aut((\Gamma,\beta))$ and $\Aut((\Gtil,\beta))$ are trivial groups.
\end{corollary}

\section{Stratification by combinatorial type}
\label{sec:stratification}

For the remainder of the paper, we restrict to the category of schemes over $\QQ$. Fix $\delta, c, g, n$ for the remainder of this section. We have two goals: first, to construct the stratum $\EE_\Gamma$ of curves of combinatorial type $\Gamma$ as a locally closed substack of $\EE^{\delta,c}_{g,n}$; and second, to prove that $\EE_\Gamma$ is a fiber bundle over an appropriate moduli stack of smooth curves
$\moduli_{\Gamma}$.

Consider the map $\EE^{\delta,c}_{g,n} \to \calH^{\delta,c}_{g,n}$ constructed in Section \ref{ssec:construction}. In order to cut down to $\EE_\Gamma$,
we need to impose various combinatorial conditions, such as which branches meet which components. For this it is necessary to label branches, impose the relevant combinatorial conditions on the rigidified stack, then forget the labels to arrive at $\EE_\Gamma$. Intuitively speaking the steps are the following:
\begin{enumerate}
  \item Restrict to the part of the relative Hilbert scheme $\calH^{\delta,c}_{g,n}$ where the $c$ points meet with fixed multiplicities $\vec{c} = (c_1, \ldots, c_m)$, which we call $\calH_{\vec{c}}$.
  \item Rigidify by numbering the $m$ distinct points parametrized, arriving at a space $\Hrig_{\vec{c}}$. 
  \item Restrict to the part of $\Hrig_{\vec{c}}$ with the right $\tilde{C}$-type, called $\calH_{(\Gtil, \beta)}$, and the part of the territory over it with the right combinatorial type, called $\Ter_{(\Gamma,\beta)}$.
  \item Take the image of $\Ter_{(\Gamma, \beta)}$ in $\EE^{\delta,c}_{g,n}$ to get rid of the rigidification and arrive at $\EE_{\Gamma}$ itself. (More precisely, $\EE_{\Gamma}$ is defined as a quotient of the orbit of $\Ter_{(\Gamma,\beta)}$.)
\end{enumerate}

We indicate in Figure \ref{fig:stratification_steps} the relationships of the intermediate spaces that we will construct.

\begin{figure}
\begin{center}
\[
\begin{tikzcd}
   & & \EE_\Gamma \ar[r,hook,"\text{loc. cl.}"] &  \EE^{\delta,c}_{g,n} \ar[d, hook,"\text{op. \& cl.}"] & {} \\
   \Ter_{(\Gamma,\beta)} \ar[rru,twoheadrightarrow,bend left=20] \ar[d] \ar[r,hook,"\text{loc. cl.}"] & \Terrig_{\vec{c}} \ar[r, "\mathfrak{S}_{\vec{c}}"] \arrow[dr, phantom, "\ulcorner", very near start] \ar[d] & \Terunrig_{\vec{c}} \ar[r,hook,"\text{loc. cl.}"] \arrow[dr, phantom, "\ulcorner", very near start] \ar[d] & \Ter^{\delta, c}_{\calZ/\tilde{\calC}/\calH^{\delta,c}_{g,n}} \ar[d] & \ar[u, "\substack{\text{require connected,}\\\text{correct genus}}"'] \\
   \calH_{(\Gtil,\beta)} \ar[r,hook, "\text{loc. cl.}"] & \Hrig_{\vec{c}} \ar[r, "\mathfrak{S}_{\vec{c}}"] & \calH_{\vec{c}} \ar[r, hook, "\text{loc. cl.}"] & \calH^{\delta,c}_{g,n} & \ar[u, "\text{add subalgebra data}"'] \\
   {} & \ar[l, "\text{fix comb. type}"] & \ar[l, "\text{add branch order}"] & \ar[l, "\substack{\text{fix branch}\\\text{incidences}}"]
\end{tikzcd}
\]
\caption{Spaces in the construction of the stratification by combinatorial type.}
\label{fig:stratification_steps}
\end{center}

\end{figure}

\bigskip

We denote an integer partition $c_1 + \cdots + c_m = c$ of $c$ by an $m$-vector $\vec{c}$, whose components $c_1, \ldots, c_m$ we make unambiguous by insisting $c_1 \geq c_2 \geq \cdots \geq c_m$.
As explained in Section \ref{sec:hilbert_stratification}, there is a locally closed stratification $\{ \imDiag^\circ_{\vec{c}} \}_{\vec{c}}$ of the Hilbert scheme of $c$ points for any family of curves $C \to S$, such that $\imDiag^\circ_{\vec{c}}$ is the locus parametrizing subschemes of $C$ whose points have multiplicities exactly $c_1, \ldots, c_m$. By Theorem \ref{thm:hilbert_stratification} and our characteristic 0 hypothesis, this locally closed stratification commutes with base change. Therefore it extends to a locally closed stratification $\bigsqcup_{\vec{c}} \calH_{\vec{c}}$ of the stack $\calH^{\delta,c}_{g,n}$.

\begin{definition} \label{def:H_vec_c}
Write $\tilde{\mathcal{C}}$ for the universal curve over $\mathcal{H}^{\delta,c}_{g,n}$ and write $\mathcal{Z}$ for the universal closed subscheme of $\tilde{\mathcal{C}}$.
Given an integer partition $\vec{c}$ of $c$, denote by $\calH_{\vec{c}}$ the locally closed substack on which the points of $\calZ$ have multiplicities exactly $\vec{c}$.
\end{definition}

We rigidify by ordering the points of the branch subscheme.

\begin{definition} \label{def:H_ord_vec_c}
Given an integer partition $\vec{c}$ of $c$, denote by $\Hrig_{\vec{c}}$ the stack parametrizing tuples $(C,  (p_1, \ldots, p_n), (q_1, \ldots, q_m) )$ where $(C, p_1, \ldots, p_n) \in \moduli^\delta_{g,n}$ and $q_1, \ldots, q_m$ are $m$ ordered branch markings 
that may coincide with the $p_i$s but not each other.
\end{definition}

Write $\mathfrak{S}_m$ for the symmetric group on $m$ letters and write
\begin{equation} \label{eq:sigma_vec_c}
  \mathfrak{S}_{\vec{c}} \coloneqq \{ \alpha \in \mathfrak{S}_m \mid c_{\alpha(i)} = c_i \text{ for all }i \}
\end{equation}
for the subgroup permuting indices $i$ with the same multiplicity $c_i$. It is isomorphic to a product of symmetric groups, one for each value of the $c_i$s. There is a faithful right action of the group $\mathfrak{S}_{\vec{c}}$ on $\Hrig_{\vec{c}}$ by $(\tilde{C}, (p_i), (q_i)) \cdot \alpha = (\tilde{C}, (p_i), (q_{\alpha(i)}))$. This coincides with the action of $\mathfrak{S}_{\mathcal{P}}$ considered in Section \ref{sec:hilbert_stratification}, so by Theorem \ref{thm:hilbert_stratification}, we have a finite \'etale quotient map $\Hrig_{\vec{c}} \to \calH_{\vec{c}}.$

\begin{definition} \label{def:ter_vec_c_and_friends}
We denote by $\Terunrig_{\vec{c}}$ the pullback of $\Ter^{\delta, c}_{\calZ/\tilde{\calC}/\calH^{\delta,c}_{g,n}}$ to $\calH_{\vec{c}}$ and by $\Terrig_{\vec{c}}$ the pullback of $\Ter^{\delta, c}_{\calZ/\tilde{\calC}/\calH^{\delta,c}_{g,n}}$ to $\Hrig_{\vec{c}}$. We denote the pullbacks of $\tilde{\calC}, \calC,$ and $\calZ$ by parallel notation.
\end{definition}

There is an induced right $\mathfrak{S}_{\vec{c}}$ action on $\Terrig_{\vec{c}}$ such that $[\Terrig_{\vec{c}} / \mathfrak{S}_{\vec{c}}] \cong \Terunrig_{\vec{c}}$ and $\Terrig_{\vec{c}} \to \Hrig_{\vec{c}}$ is $\mathfrak{S}_{\vec{c}}$-equivariant.

After pulling back to $\Hrig_{\vec{c}}$, the universal closed subscheme $\Zrig_{\vec{c}}$ of $\Ctilrig_{\vec{c}}$ decomposes into a sum
of divisors
\[
  \Zrig_{\vec{c}} = \sum_{i = 1}^m c_i q_i.
\]
Since $\Zrig_{\vec{c}}$ is the vanishing of the conductor ideal, the $c_i$s are the branch conductances of the various branches of singularities of $\Crig_{\vec{c}}$. We write
\[
  \Qrig_{\vec{c}} = \sum_{i = 1}^m q_i
\]
for a reduced version of the conductor $\Zrig_{\vec{c}}$. Observe that the image of $q_i$ in $\tilde{\calC}_{\vec{c}}$ is reduced and coincides with $q_j$ if and only if $c_i = c_j$. Accordingly,
we write
\begin{equation} \label{eq:q_vec_c}
  \calQ_{\vec{c}} = \sum_{c \in \{ c_1, \ldots, c_m \}} q_c
\end{equation}
where $q_c$ is the common image of the $q_i$ with $c_i = c$. This gives us access to a reduced version of $\calZ_{\vec{c}}$ on $\calH_{\vec{c}}$.

\subsection{Construction of $\EE_\Gamma$}

\begin{definition} \label{def:e_gamma_moduli_functor}
Let $\Gamma$ be a connected $n$-marked combinatorial type of genus $g$ and invariants $\delta, c$. Let $\vec{c}$ be the vector of branch conductances of $\Gamma$.
We define the \emphbf{moduli stack of curves of type $\Gamma$}, $\EE_\Gamma$, as the substack of $\EE^{\delta,c}_{g,n}$ such that a family $\nu : \tilde{C} \to C$ of $n$-marked equinormalized curves of genus $g$ with invariants $\delta, c$ factors through $\EE_{\Gamma}$ if and only if
  \begin{enumerate}[label={(B\arabic*)}]
    \item The vanishing of the conductor $V(\Cond_{\tilde{C}/C}) \subseteq \tilde{C}$ factors through $\calH_{\vec{c}}$;
    \item For each $i,j$ such that $\mu(i) = \mu(j)$, we have $p_i = p_j$;
    \item For each $i$ with $\mu(i) \in B_\Gamma$, $p_i(S)$ factors through the subscheme $q_{c(\mu(i))}$;
    \item Each geometric fiber of $\nu$ has combinatorial type $\Gamma.$
  \end{enumerate}
\end{definition}

The goal of this section is to prove the following theorem.

\begin{theorem} \label{thm:stratification_by_gamma}
The stack $\EE_{\Gamma}$ is a locally closed substack of $\EE^{\delta,c}_{g,n}$. Moreover, $\EE^{\delta, c}_{g,n}$ is a disjoint union of the substacks $\EE_{\Gamma}$ as $\Gamma$ varies among the 
connected $n$-marked combinatorial types of genus $g$ and invariants $\delta, c$.
\end{theorem}

Recall that in Section \ref{sec:territories_singularities}, we showed that the territory of $A_{\vec{c}}$ decomposes into a disjoint union according to which points are glued. In order to show that a similar decomposition holds for $\Terrig_{\vec{c}}$ and therefore $\EE_{g,n}^{\delta,c}$, we must prove the following lemmas.

\begin{lemma} \label{lem:universal_subscheme_as_a_bundle}
The morphism $\calZ_{\vec{c}} \to \calH_{\vec{c}}$ is an \'etale local $\Spec A_{\vec{c}}$-bundle, i.e.,
there exists an \'etale cover $S \to \calH_{\vec{c}}$ and a cartesian diagram
\[
\begin{tikzcd}
  \calZ_{\vec{c}}|_S \ar[r] \ar[d] & \Spec A_{\vec{c}} \ar[d] \\
  S \ar[r] & \Spec \ZZ
\end{tikzcd}
\]
Moreover, $\Zrig_{\vec{c}} \to \Hrig_{\vec{c}}$ is an \'etale local $\Spec A_{\vec{c}}$-bundle in the same sense.
\end{lemma}
\begin{proof}
Since $\Hrig_{\vec{c}} \to \calH_{\vec{c}}$ is an \'etale cover, it suffices to construct the required cartesian diagram on a cover of $\Hrig_{\vec{c}}$.
In addition, since smooth morphisms to algebraic stacks admit \'etale local sections, it suffices to do so on a smooth cover $S \to \Hrig_{\vec{c}}$. Begin by choosing $S$ an arbitrary smooth cover of $\Hrig_{\vec{c}}$ by a scheme. We may also reduce to finding, for each geometric point $s$ of $S$,
an \'etale neighborhood of $s$ on which $\calZ |_S \cong S \times \Spec A_{\vec{c}}$. Consider the points $q_1(s), \ldots, q_m(s)$. Since $q_i(s)$ is a point in the smooth locus of the family of curves $\tilde{\calC}|_S \to S$ it is standard that $\tilde{\calC}|_{S}$ is
\'etale locally isomorphic to $\AA^1_S$ around $q_i(s)$. More precisely, replacing $S$ by an affine \'etale neighborhood of $s$ if necessary, there exist maps
\[
\begin{tikzcd}
  & U_i \ar[dl,"g_i"'] \ar[dr, "f_i"] & \\
  \mathbb{A}^1_{S} \ar[dr] & & \tilde{\calC}|_S \ar[dl] \\
   & S &
\end{tikzcd}
\]
where $f_i : U_i \to \tilde{\calC}|_S$ is an \'etale neighborhood of $q_i(s)$ and the map $g_i$ is \'etale. Replacing $S$ by the \'etale neighborhood $U_i \times_{\tilde{\calC}|_S, q_i} S$, we may assume that the section $q_i$ of $\tilde{\calC}|_S \to S$ lifts to a section $\tilde{q}_i : S \to U_i$. Changing coordinates if necessary, we may assume that the induced section $g_i \circ \tilde{q}_i : S \to \AA^1_{S}$ is the zero section. Then it is clear that $c_i$ times the divisor of $(g_i \circ \tilde{q}_i)(S)$ is isomorphic to $\Spec \OO_S[t_i]/(t_i^{c_i})$. By the infinitesimal lifting property of \'etale maps, $c_i$ times the divisor of $(q_i)(S)$ is also isomorphic to $\Spec \OO_S[t_i]/(t_i^{c_i})$. Repeating the process for the other $q_i$, we obtain the result.
\end{proof}

\begin{corollary}
$\Terunrig_{\vec{c}} \to \calH_{\vec{c}}$ and $\Terrig_{\vec{c}} \to \Hrig_{\vec{c}}$ are \'etale local $\Ter^{\delta,c}_{A_{\vec{c}}}$-bundles.
\end{corollary}
\begin{proof}
This follows from Lemma \ref{lem:ter_base_change} and Theorem \ref{thm:territory_of_scheme_representable}.
\end{proof}

Recall that the territory $\Ter^{\delta,c}_{A_{\vec{c}}}$ decomposes into a disjoint
union of territories of gluing data $\Tergl(\calP, g, \vec{c})$ (Definition \ref{def:territory_gluing_datum}). We next show that this disjoint union decomposition extends from the fibers of
the bundle $\Terrig_{\vec{c}} \to \Hrig_{\vec{c}}$ to its total space.

\begin{lemma} \label{lem:territory_bundle_decomposition}
$\Terrig_{\vec{c}}$ admits a decomposition into disjoint open and closed substacks
\[
  \Terrig_{\vec{c}} = \bigsqcup_{(\calP, g)} T_{\calP, g}
\]
where the disjoint union is over all $m$-branch gluing data $(\calP, g)$ of corank $\delta$, and for each such $(\calP, g)$, the morphism $T_{\calP, g} \to \Hrig_{\vec{c}}$ is an \'etale local
$\Tergl(\calP, g, \vec{c})$-bundle.
\end{lemma}
\begin{proof}
Let $S \to \Hrig_{\vec{c}}$ be a cover as in Lemma \ref{lem:universal_subscheme_as_a_bundle}. To declutter the notation, we write $\tilde{\calC} \to \calC$ instead of $\Ctilrig_{\vec{c}} \to \Crig_{\vec{c}}$. Writing $\pi_1, \pi_2$ for the projections $S \times_{\Hrig_{\vec{c}}} S \to S$,
there is an isomorphism of equinormalized curves
\[
  \psi_{12} : \pi_1^*(\tilde{\calC}|_S \to \calC|_S) \to \pi_2^*(\tilde{\calC}|_S \to \calC|_S)
\]
coming from the fact that both are pullbacks of $\tilde{\calC} \to \calC$ along $S \times_{\Hrig_{\vec{c}}} S \to \Hrig_{\vec{c}}$.
This map preserves genera and branch multiplicities of singularities and fixes the $q_i$.

It follows that the induced cocycle
\[
 \phi_{12} : \pi_{1}^*\Ter^{\delta,c}_{A_{\vec{c}}} \to \pi_{2}^*\Ter^{\delta,c}_{A_{\vec{c}}}
\]
is compatible with the disjoint union decomposition of Definition \ref{def:territory_gluing_datum}, i.e., $\phi_{12}$ takes the summand indexed by $(\calP, g)$ to itself. Therefore this disjoint union decomposition descends to a
disjoint union decomposition of $\Terrig_{\vec{c}}$ into open and closed substacks $T_{\calP, g}$. By construction, we have an isomorphism over $S$
\[
  T_{\calP, g} \times_{\Hrig_{\vec{c}}} S \cong \Tergl(\calP, g, \vec{c}) \times S,
\]
so $T_{\calP, g}$ is an \'etale local $\Tergl(\calP, g, \vec{c})$-bundle.
\end{proof}

\begin{remark}
The same is \emph{not} true of $\Terunrig_{\vec{c}} \to \calH_{\vec{c}}$: the cocycle may interchange points $q_i$ with the same intersection multiplicity, which may interchange summands of the disjoint union. Consider for example, an irreducible rational curve
with two self nodes. The orbit of the curve under permuting the branches lands in three distinct summands: one with $q_1, q_2$ and $q_3, q_4$ glued, one with $q_1, q_3$ and $q_2, q_4$ glued, and one with $q_1, q_4$ and $q_2, q_3$ glued.
\end{remark}

\begin{definition} \label{def:rigid_type_of_point}
The \emphbf{rigidified combinatorial type of a geometric point $(\nu : \tilde{C} \to C, (p_i), (q_i))$ of $\Terrig_{\vec{c}}$} is the pair $(\Gamma, \beta)$ where $\Gamma$ is the combinatorial type of the underlying normalized curve and $\beta: \{ 1, \ldots, m \} \to B_\Gamma = \{ q_1, \ldots, q_m \}$ is the map $i \mapsto q_i$.

The \emphbf{rigidified $\tilde{C}$-type of a geometric point $(\tilde{C}, (p_i), (q_i))$ of $\Hrig_{\vec{c}}$} is the pair $(\Gtil, \beta)$ where $\Gtil$ is the $\tilde{C}$-type of $(\tilde{C}, (p_i), \sum_{i = 1}^m c_iq_i)$ and $\beta : \{1, \ldots, m\} \to B_{\Gtil} = \{ q_1, \ldots, q_m \}$ is the map $i \mapsto q_i$.
\end{definition}

Now we cut down to the part of $\Hrig_{\vec{c}}$ where $\tilde{C}$ and its markings are of the correct kind to have rigidified combinatorial type $(\Gamma, \beta)$.
Since rigidifying the branches also rigidifies the components (Lemma \ref{lem:isos_determined_by_branches}), it is enough to impose the conditions (A1)-(A5) below.

\begin{definition} \label{def:H_gamma_beta}
Let $\Gtil = (K, B, D, \tau, \mu, \epsilon, g, c)$ and let $(\Gtil, \beta)$ be a rigidified $\tilde{C}$-type. Write $\calH_{(\Gtil,\beta)}$ for the locally closed substack of $\Hrig_{\vec{c}}$ cut out by the conditions:
\begin{enumerate}[label={(A\arabic*)}]
  \item (Correct collisions of markings) For each $i,j \in \{1, \ldots, n\}$, if $\mu(i) = \mu(j)$, then $p_i = p_j$, otherwise $p_i \neq p_j$;
  \item (Correct collisions of markings and branches) For each $i \in \{1, \ldots, m\}, j \in \{1, \ldots, n\}$, if $\beta(i) = \mu(j)$, then $q_i = p_j$, otherwise $q_i \neq p_j$;
  \item (Correct adjacency of branches) For each $i,j \in \{1, \ldots, m\}$, $q_i$ and $q_j$ belong to the same connected component of each fiber of $\Ctilrig_{\vec{c}}$ if and only if $\tau(\beta(i)) = \tau(\beta(j))$;
  \item (Correct adjacency of markings and branches) For each $i \in \{1, \ldots, m\}, j \in \{1, \ldots, n\}$ such that $\mu(j) \in D$, we have that $q_i$ and $p_j$ belong to the same connected component of each geometric fiber of $\Ctilrig_{\vec{c}}$
  if and only if $\tau(\beta(i)) = \epsilon(\mu(j))$;
  \item (Correct genus of components) For each $i \in \{1, \ldots, m\}$, the connected component of $\Ctilrig_{\vec{c}}$ to which $q_i$ belongs has genus $g(\tau(\beta(i)))$. If $m = 0$, the unique connected component of $\tilde{C}$ has the same genus as $\Gtil$.
\end{enumerate}
\end{definition}

We may think of $\calH_{(\Gtil, \beta)}$ as a ``rigidified base space'' for a rigidification of $\EE_{\Gamma}$: it keeps track of the additional data of an ordering of the branch markings (and consequently how the components of $\tilde{C}$ match with $K(\Gamma)$), but is missing the subalgebra data encoded by territories which we need to construct the map $\tau : \tilde{C} \to C$.

Next, we consider the rigidification of $\EE_{\Gamma}$ itself. Like $\calH_{(\Gtil,\beta)}$ it keeps track of an ordering of the branch markings, but also possesses the subalgebra data. More concretely, using the disjoint union decomposition of Lemma \ref{lem:territory_bundle_decomposition}, we cut down to the part of $\Terrig_{\vec{c}}$ where the branches are glued to singularities as prescribed by $\Gamma$.

\begin{definition}
Let $\Gamma = (V = S \sqcup K, B, D, \sigma, \tau, \mu, \epsilon, g, c)$ and let $(\Gamma, \beta)$ be a rigidified combinatorial type. The \textbf{gluing datum $(\calP, g)$ associated to $(\Gamma, \beta)$} consists of
\begin{enumerate}
  \item $\calP$ is the partition of $\{1, \ldots, m\}$ whose parts are the preimages of singletons under the function $\sigma \circ \beta$;
  \item $g(P)$ is the genus of the singularity of $\Gamma$ corresponding to the part $P$ of $\calP$.
\end{enumerate}
\end{definition}

\begin{definition} \label{def:ter_gamma_beta}
Let $(\Gamma, \beta)$ be a rigidified combinatorial type and $(\calP, g)$ its associated gluing datum. Denote by $\Ter_{(\Gamma,\beta)}$ the locally closed substack $\calH_{(\Gtil, \beta)}|_{\Terrig_{\vec{c}}} \cap T_{\calP, g}$ of $\Terrig_{\vec{c}}$.
\end{definition}

\begin{remark} \label{rem:substacks_have_right_rigid_types} \hfill
\begin{enumerate}
  \item A geometric point of $\Hrig_{\vec{c}}$ has rigidified $\tilde{C}$-type isomorphic to $(\Gtil, \beta)$ if and only if it factors through $\calH_{(\Gtil,\beta)}$.
  \item A geometric point of $\Terrig_{\vec{c}}$ has rigidified combinatorial type isomorphic to $(\Gamma, \beta)$ if and only if it factors through $\Ter_{(\Gamma, \beta)}$.
\end{enumerate}
\end{remark}

At this point we would like to take the image of $\Ter_{(\Gamma,\beta)}$ in $\EE_{g,n}^{\delta,c}$. In order to get the correct stack structure on this image, we will construct it as a quotient stack. Thus, we consider the action of $\mathfrak{S}_{\vec{c}}$ on these substacks. Observe that given $\alpha \in \mathfrak{S}_{\vec{c}}$,
\[
  \calH_{(\Gtil, \beta)} \cdot \alpha = \calH_{(\Gtil, \beta \circ \alpha)}
\]
and
\[
  \Ter_{(\Gamma,\beta)} \cdot \alpha = \Ter_{(\Gamma, \beta \circ \alpha)}.
\]
Moreover, if the summand of the disjoint union decomposition of Lemma \ref{lem:territory_bundle_decomposition} containing $\Ter_{(\Gamma,\beta)}$ is indexed by a gluing datum $(\calP, g)$, then
the gluing datum of the summand to which $\Ter_{(\Gamma,\beta)} \cdot \alpha$ belongs is indexed by the gluing datum $(\calP, g) \cdot \alpha$ defined below.

\begin{definition}
We define an action of $\mathfrak{S}_{\vec{c}}$ on the set of $m$-branch gluing data of corank $\delta$ by $(\calP, g) \cdot \alpha = (\calQ, g')$ where $\calQ = \{ \alpha^{-1}(P) \mid P \in \calP \}$ and $g'(Q) = g(\alpha(Q))$ for each $Q \in \calQ$.
\end{definition}

We would like to think of the orbits
\[
  \calH_{(\Gtil, \beta)} \cdot \mathfrak{S}_{\vec{c}} = \bigcup_{\alpha \in \mathfrak{S}_{\vec{c}}} \calH_{(\Gtil, \beta)} \cdot \alpha \quad \text{and} \quad \Ter_{(\Gamma, \beta)} \cdot \mathfrak{S}_{\vec{c}} = \bigcup_{\alpha \in \mathfrak{S}_{\vec{c}}} \Ter_{(\Gamma, \beta)} \cdot \alpha
\]
as locally closed substacks of $\Hrig_{\vec{c}}$ and $\Terrig_{\vec{c}}$, respectively. In general, a finite union of locally closed
substacks need not be locally closed. Even when the locally closed substacks are disjoint with a locally closed union, their union need not
be isomorphic to the disjoint union of the substacks. We will now analyze the action of $\mathfrak{S}_{\vec{c}}$ to show that neither of these
pathologies occur in our situation.

\begin{lemma} \label{lem:action_on_rigid_types}
Let $\Gamma = (V = S \sqcup K, B, D, \sigma, \tau, \mu, \epsilon, g, c)$ and let $(\Gamma, \beta)$ be a rigidified combinatorial type.
\begin{enumerate}
  \item For $\alpha \in \mathfrak{S}_{\vec{c}}$, either $\calH_{(\Gtil, \beta)} \cdot \alpha = \calH_{(\Gtil, \beta)}$ or 
  $\calH_{(\Gtil, \beta)} \cdot \alpha \cap \calH_{(\Gtil, \beta)} = \emptyset.$
  \item For $\alpha \in \mathfrak{S}_{\vec{c}}$, either $\Ter_{(\Gamma, \beta)} \cdot \alpha = \Ter_{(\Gamma, \beta)}$ or 
  $\Ter_{(\Gamma, \beta)} \cdot \alpha \cap \Ter_{(\Gamma, \beta)} = \emptyset.$
  \item The group
  \[
    \tilde{G} = \{ \alpha \in \mathfrak{S}_{\vec{c}} \mid \calH_{(\Gtil, \beta)} \cdot \alpha = \calH_{(\Gtil, \beta)} \}
  \]
  is isomorphic to $\Aut(\Gtil).$
  \item The group
  \[
    G = \{ \alpha \in \mathfrak{S}_{\vec{c}} \mid \Ter_{(\Gamma, \beta)} \cdot \alpha = \Ter_{(\Gamma, \beta)} \}
  \]
  is isomorphic to $\Aut(\Gamma).$
\end{enumerate}
\end{lemma}

\begin{proof}
\textbf{(i):} Consider the conditions for membership in $\calH_{(\Gtil, \beta)} \cdot \alpha = \calH_{(\Gtil, \beta \circ \alpha)}$ assuming membership in $\calH_{(\Gtil,\beta)}.$ Whether each condition holds can be phrased solely in terms of $\Gamma$ and $\beta$. Explicitly:
\begin{enumerate}[label={(A\arabic*)}]
  \item Always holds.
  \item Holds if and only if $\beta(i) = \beta(\alpha(i))$ whenever $\beta(i)$ is in the image of $\mu$.
  \item Holds if and only if the equivalence relation induced by $\tau \circ \beta$ is the same as the equivalence relation induced by $\tau \circ \beta \circ \alpha$.
  \item Holds if and only if, for each $i$, $\tau(\beta(i)) = \tau(\beta(\alpha(i)))$ whenever $\tau(\beta(i))$ is in the image of $\epsilon$.
  \item Holds if and only if $g\circ \tau \circ \beta = g \circ \tau \circ \beta \circ \alpha$.
\end{enumerate}
It follows that the conditions for membership in $\calH_{(\Gtil, \beta)} \cdot \alpha$ and $\calH_{(\Gtil,\beta)}$ either agree for all curves or disagree for all curves.

\textbf{(ii):}
Let $(\calP, g)$ be the gluing datum associated to $(\Gamma, \beta)$. Then $\Ter_{(\Gamma, \beta)} \cdot \alpha = \Ter_{(\Gamma, \beta)}$ if and only if $\calH_{(\Gtil, \beta)} \cdot \alpha = \calH_{(\Gtil, \beta)}$ and $(\calP, g) \cdot \alpha = (\calP, g)$. Otherwise $\Ter_{(\Gamma, \beta)} \cdot \alpha \cap \Ter_{(\Gamma, \beta)} = \emptyset$.

\textbf{(iii):} If $B$ is empty, then $\Gamma$ has a single vertex, so both $\tilde{G}$ and $\Aut(\Gtil)$ are trivial groups, and the result holds trivially.

Assume that $B$ contains at least one element. Let $\alpha \in \tilde{G}$. We will define a morphism $\phi_{\alpha} : \Gtil \to \Gtil$ associated to $\alpha.$ Let $\phi_{\alpha,B} : B \to B$ be defined by $\beta(i) \mapsto \beta(\alpha(i))$,
which is well-defined since $\beta$ is a bijection.
Define $\phi_{\alpha,K} : K \to K$ by $\tau(\beta(i)) \mapsto \tau(\beta(\alpha(i)))$. This is well-defined and bijective by (A3) above. Define $\phi_{\alpha,D} : D \to D$ by the identity map. It remains to check that these choices commute with $\tau, \epsilon, \mu, g, c.$
We have commutativity with $\tau$ by construction. Commutativity with $\epsilon$ follows from (A4) above. Commutativity with $\mu$ is automatic. We have commutativity with $g$ by (A5) above. Finally, commutativity with $c$ holds since $\alpha$ preserves conductances.
Therefore $\phi_\alpha$ is an isomorphism of $\Gtil,$ as desired.

It is clear that the map $\xi : \tilde{G} \to \Aut(\Gtil)$ given by $\alpha \mapsto \phi_{\alpha}$ is a group homomorphism. We claim that the function $\eta : \Aut(\Gtil) \to \tilde{G}$ given by $\phi \mapsto \beta^{-1} \circ \phi_B \circ \beta$ is an inverse group homomorphism. It is well-defined since $\phi_B$ preserves branch conductances, and it is clear that $\eta$ is a group homomorphism. The composite $\eta \circ \xi : \tilde{G} \to \tilde{G}$ is clearly the identity. The other composite $\xi \circ \eta : \Aut(\Gtil) \to \Aut(\Gtil)$ is the identity since definition chasing shows $(\xi \circ \eta)(\phi)_B = \phi_B$ and automorphisms of $\tilde{C}$-types are determined by their restriction to $B$.

\textbf{(iv):} As before, if $B$ is empty the result is immediate.

Assume $B$ contains at least one element. Then as above, from $\alpha \in G$ we will construct a morphism $\phi_{\alpha} : \Gamma \to \Gamma$. Let $(\calP, g)$ be the gluing datum associated to $(\Gamma, \beta)$. Since $\alpha \in G$, we have $(\calP, g) \cdot \alpha = (\calP, g)$. This implies that $\sigma \circ \beta$ induces the same equivalence relation as $\sigma \circ \beta \circ \alpha$ and $g|_S \circ \sigma \circ \beta = g|_S \circ \sigma \circ \beta \circ \alpha$. We may construct $\phi_{\alpha,B}, \phi_{\alpha,K}, \phi_{\alpha,D}$ as before. We define $\phi_{\alpha,S} : S \to S$ by $\sigma(\beta(i)) \mapsto \sigma(\beta(\alpha(i)))$. This is well-defined since $\sigma \circ \beta$ induces the same equivalence relation as $\sigma \circ \beta \circ \alpha$. Commutativity of $\phi$ with $\tau, \epsilon, \mu, g|_K, c$
is as above. Commutativity of $\phi_{\alpha,S}$ with $\sigma$ holds by definition. Commutativity of $\phi_{\alpha,S}$ with $g|_S$ holds since $g|_S \circ \sigma \circ \beta = g|_S \circ \sigma \circ \beta \circ \alpha$.

The remainder of the proof proceeds as in the proof of (iii).
\end{proof}

\begin{corollary}
If $\alpha \in \mathfrak{S}_{\vec{c}}$ and $\calH_{(\Gtil, \beta)} \cdot \alpha \cap \calH_{(\Gtil, \beta)} = \emptyset$,
then the closures of $\calH_{(\Gtil, \beta)} \cdot \alpha$ and $\calH_{(\Gtil, \beta)}$ in $\Hrig_{\vec{c}}$ are also disjoint.

Similarly, if $\Ter_{(\Gamma, \beta)} \cdot \alpha \cap \Ter_{(\Gamma, \beta)} = \emptyset$, then the closures of
$\Ter_{(\Gamma, \beta)} \cdot \alpha$ and $\Ter_{(\Gamma, \beta)}$ in $\Terrig_{\vec{c}}$ are also disjoint.
\end{corollary}

\begin{proof}
We will just prove the result for $\Ter_{(\Gamma,\beta)}$, as the proof is analogous for $\calH_{(\Gtil, \beta)}$. To lighten the
notation, we will write $T = \Ter_{(\Gamma, \beta)}$.

Consider the closed substack $\tilde{T}$ of $\Terrig_{\vec{c}}$ given by conditions (A3) -- (A5), factorization through $T_{\calP, g}$, and the following weakenings of (A1) and (A2):
\begin{enumerate}[label={(A\arabic*$'$)}]
  \item For each $i,j \in \{1, \ldots, n\}$, if $\mu(i) = \mu(j)$, then $p_i = p_j$;
  \item For each $i \in \{1, \ldots, m\}, j \in \{1, \ldots, n\}$, if $\beta(i) = \mu(j)$, then $q_i = p_j$.
\end{enumerate}
Then clearly $\ol{T} \subseteq \tilde{T}$ and $\ol{T \cdot \alpha} \subseteq \tilde{T} \cdot \alpha,$ so it suffices to show that either
$T = T \cdot \alpha$ or $\tilde{T} \cap \tilde{T} \cdot \alpha = \emptyset.$ Suppose $(\nu : \tilde{C} \to C, (p_i), (q_j))$ is a point of $\tilde{T} \cap \tilde{T} \cdot \alpha.$ Then, since the construction of $\phi_{\alpha}$ in the proof of part (iv) above does not depend on axioms (A1) and (A2),
we can construct an associated map $\phi_{\alpha} : \Gamma \to \Gamma$ as in the proof of part (iv). Applying the map $\Aut(\Gamma) \to G$ to $\phi_\alpha$, we conclude that $\alpha \in G$. Therefore $T = T \cdot \alpha$.
\end{proof}

\begin{corollary}
The algebraic stacks
    \begin{align*}
      \calH_{(\Gtil, \beta)} \cdot \mathfrak{S}_{\vec{c}} &\coloneqq \bigsqcup_{[\alpha] \in \Aut(\Gtil)\setminus\mathfrak{S}_{\vec{c}}} \calH_{(\Gtil, \beta)} \cdot \alpha \\
      \text{ and } \Ter_{(\Gamma,\beta)} \cdot \mathfrak{S}_{\vec{c}} &\coloneqq \bigsqcup_{[\alpha] \in \Aut(\Gamma) \setminus \mathfrak{S}_{\vec{c}}} \Ter_{(\Gamma, \beta)} \cdot \alpha 
    \end{align*}
    have natural structures of locally closed substacks of $\Hrig_{\vec{c}}$ and $\Terrig_{\vec{c}}$, respectively.
\end{corollary}

\begin{proof}
By the previous lemmas and corollary, $\calH_{(\Gtil, \beta)} \cdot \mathfrak{S}_{\vec{c}}$ is a union of locally closed substacks with disjoint closures.
Given a class $[\alpha]$ in the set of right cosets $\Aut(\Gtil)\setminus\mathfrak{S}_{\vec{c}}$, write
\[
  U_{[\alpha]} = \Hrig_{\vec{c}} \setminus \bigcup_{[\gamma] \neq [\alpha]} \ol{\calH_{(\Gtil, \beta)} \cdot \gamma}.
\]
Then the $U_{[\alpha]}$s form an open cover of $\Hrig_{\vec{c}}$ on which the natural map $\calH_{(\Gtil, \beta)} \cdot \mathfrak{S}_{\vec{c}} \to \Hrig_{\vec{c}}$ restricts to an immersion, so by Zariski descent, $\calH_{(\Gtil, \beta)} \cdot \mathfrak{S}_{\vec{c}} \to \Hrig_{\vec{c}}$ is an immersion.

The proof for $\Ter_{(\Gamma, \beta)} \cdot \mathfrak{S}_{\vec{c}}$ is similar.
\end{proof}

We now have enough facts in place to define $\moduli_{\Gtil}$ and $\EE_\Gamma$ constructively. We will prove at the end of the section that this definition of $\EE_\Gamma$ agrees with Definition \ref{def:e_gamma_moduli_functor} given at the beginning
of the section. Until we have established the equivalence, $\EE_\Gamma$ will refer to the following definition:

\begin{definition} \label{def:strata_by_quotients}
Let $(\Gamma, \beta)$ be a rigidified combinatorial type. We define
\[
  \moduli_{\Gtil} = [ (\calH_{(\Gtil, \beta)} \cdot \mathfrak{S}_{\vec{c}}) / \mathfrak{S}_{\vec{c}}] \quad \text{and} \quad \EE_\Gamma = [ (\Ter_{(\Gamma, \beta)} \cdot \mathfrak{S}_{\vec{c}}) / \mathfrak{S}_{\vec{c}} ].
\]
\end{definition}

\begin{lemma} \label{lem:e_gamma_m_gamma_basic_facts}
The following basic facts hold.
\begin{enumerate}
  \item $\moduli_{\Gtil}$ and $\EE_\Gamma$ do not depend on the choice of $\beta$.
  \item $\moduli_{\Gtil}$ is naturally a locally closed substack of $\calH_{\vec{c}}$, and therefore of $\calH^{\delta,c}_{g,n}$.
  \item $\EE_\Gamma$ is naturally a locally closed substack of $\Terunrig_{\vec{c}}$.
  \item A geometric point of $\calH_{\vec{c}}$ factors through $\moduli_{\Gtil}$ if and only if its $\tilde{C}$-type is $\Gtil$.
  \item A geometric point of $\Terunrig_{\vec{c}}$ factors through $\EE_\Gamma$ if and only if its combinatorial type is $\Gamma$.
\end{enumerate}
\end{lemma}
\begin{proof}
We will provide the proofs for the claims about $\EE_\Gamma$. The proofs of the claims regarding $\moduli_{\Gtil}$ are similar.

\textbf{(i):} Suppose $\beta'$ is a second choice of bijection $\{1, \ldots, m \} \to B$ such that $c(\beta'(i)) = c_i$ for all $i$. Then $\alpha = \beta^{-1} \circ \beta' \in \mathfrak{S}_{\vec{c}}$. This implies that
the orbits of $\Ter_{(\Gamma, \beta)}$ and $\Ter_{(\Gamma, \beta')} = \Ter_{(\Gamma, \beta \circ \alpha)}$ coincide, so the definition of $\EE_\Gamma$ is unchanged.

\textbf{(iii):} The category of $S$-points of $\EE_\Gamma$ is the category of right $\mathfrak{S}_{\vec{c}}$ torsors $P \to S$ with an equivariant map $P \to \Ter_{(\Gamma, \beta)} \cdot \mathfrak{S}_{\vec{c}}$. We obtain a map from $\EE_\Gamma$ to $\Terunrig_{\vec{c}} = [\Terrig_{\vec{c}} / \mathfrak{S}_{\vec{c}}]$ by composing $P \to \Ter_{(\Gamma, \beta)} \cdot \mathfrak{S}_{\vec{c}}$ with the locally
closed immersion $\Ter_{(\Gamma, \beta)} \cdot \mathfrak{S}_{\vec{c}} \to \Terrig_{\vec{c}}.$

The base change of the resulting morphism $\EE_\Gamma \to \Terunrig_{\vec{c}}$ to $\Terrig_{\vec{c}}$ is the same locally closed immersion $\Ter_{(\Gamma, \beta)} \cdot \mathfrak{S}_{\vec{c}} \to \Terrig_{\vec{c}}$. Therefore, since $\Terrig_{\vec{c}} \to \Terunrig_{\vec{c}}$ is an \'etale cover, $\EE_\Gamma \to \Terunrig_{\vec{c}}$ is also a locally closed immersion.

\textbf{(v):} Suppose $(\nu : \tilde{C} \to C, (p_i), \sum_{i=1}^m c_iq_i)$ is a geometric point of $\Terunrig_{\vec{c}}$. Let $(\nu : \tilde{C} \to C, (p_i), (q_i))$ be a lift to $\Terrig_{\vec{c}}$. By Remark \ref{rem:substacks_have_right_rigid_types},
this lift has combinatorial type $\Gamma$ if and only if it factors through $\Ter_{(\Gamma, \beta')}$ for some $\beta'$, or equivalently, if it factors through $\Ter_{(\Gamma, \beta)} \cdot \mathfrak{S}_{\vec{c}}$. This is equivalent to saying
that the original geometric point of $\Terunrig_{\vec{c}}$ factors through $\EE_\Gamma$, as desired.
\end{proof}

Finally, we give the proof of Theorem \ref{thm:stratification_by_gamma}, by first showing that $\EE_\Gamma$ represents the moduli functor with which we originally defined $\EE_\Gamma$, then deducing from our construction that $\EE_\Gamma$ is a locally closed substack
of $\EE^{\delta,c}_{g,n}$.

\begin{proof}
Consider $\EE_\Gamma$ as a locally closed substack of $\Terunrig_{\vec{c}}$ by part (iii) of Lemma \ref{lem:e_gamma_m_gamma_basic_facts}. Write $\mathcal{F}$ for the stack of Definition \ref{def:e_gamma_moduli_functor}.
Given a family in $\EE_\Gamma$, it is clear that the conditions (B1)--(B4) hold. Conversely, given a family in $\mathcal{F}$, then by axiom (B1) it factors through $\Terunrig_{\vec{c}}$.
Conditions (B2) and (B3) cut out a closed substack of $\Terunrig_{\vec{c}}$ in which $\EE_{\Gamma}$ is open, since these pull back to the only closed conditions used to cut out $\Ter_{(\Gamma,\beta)}$ in $\Terrig_{\vec{c}}$, namely the conditions (A1$'$) that if $\mu(i) = \mu(j)$, then $p_i = p_j$, and (A2$'$) if $\beta(i) = \mu(j)$, then $q_i = p_j$.
Further factorization through $\EE_{\Gamma}$ is characterized solely by membership of geometric points, and we conclude $\EE_\Gamma = \mathcal{F}$ by Lemma \ref{lem:e_gamma_m_gamma_basic_facts}(v).

That $\EE_\Gamma$ is locally closed in $\EE^{\delta,c}_{g,n}$ is immediate from parts (iii) and (v) of Lemma \ref{lem:e_gamma_m_gamma_basic_facts} and the fact that $\Terunrig_{\vec{c}}$ is locally closed
in $\EE^{\delta,c}_{g,n}$. The statement about $\EE^{\delta,c}_{g,n}$ being a disjoint union of the substacks $\EE_\Gamma$ is clear, as each geometric point possesses exactly one combinatorial type.
\end{proof}

\subsection{$\EE_\Gamma$ is a fiber bundle}

The goal of this section is to prove the following.

\begin{theorem} \label{thm:e_gamma_fiber_bundle}
Let $\Gamma$ be a connected combinatorial type with conductances $\vec{c}$. Choose a rigidification $(\Gamma, \beta)$, and let $(\mathcal{P},g)$ be the gluing datum associated to $(\Gamma, \beta)$. Then there is an \'etale fiber bundle $\EE_{\Gamma} \to \moduli_\Gamma$ with fiber $\Tergl(\mathcal{P}, g, \vec{c})$
over
\[
  \moduli_\Gamma \coloneqq \left[ \left(\prod_{v \in K} \moduli_{g(v), n(v) + m(v)}\right) / \Aut(\Gamma) \right]
\]
where $n(v)$ denotes the number of distinguished points adjacent to $v$ and $m(v)$ denotes the number of branches incident to $v$. 
\end{theorem}

\begin{remark}
We have $\sum_{v \in K} n(v) \leq n$ with equality precisely when none of the markings lie on singularity branches. The sum $\sum_{v \in K} m(v)$ is always the number of branches.
\end{remark}

We will also show that the $\EE_{\Gamma}$s form \'etale fiber bundles over
the spaces $\moduli_{\Gtil}$. Unlike the spaces $\moduli_{\Gamma}$, the spaces $\moduli_{\Gtil}$ are disjoint locally closed substacks of the common base space $\calH^{\delta,c}_{g,n}$. In exchange for living in a common space, the fiber of the bundle is more complex.

\begin{theorem} \label{thm:e_gamma_meh_fiber_bundle}
With notation as in the previous theorem, there is an \'etale fiber bundle $\EE_{\Gamma} \to \moduli_{\Gtil}$ with fiber
\[
  \bigsqcup_{\Aut(\Gamma) \setminus \Aut(\Gtil)} \Tergl(\mathcal{P}, g, \vec{c})
\]
over
\[
  \moduli_{\Gtil} \cong \left[ \left(\prod_{v \in K} \moduli_{g(v), n(v) + m(v)}\right) / \Aut(\Gtil) \right]
\]
where $n(v)$ denotes the number of distinguished points adjacent to $v$ and $m(v)$ denotes the number of branches incident to $v$. 
\end{theorem}

Before we prove the theorems, we will show that $\moduli_{\Gtil} = [ (\calH_{(\Gtil, \beta)} \cdot \mathfrak{S}_{\vec{c}}) / \mathfrak{S}_{\vec{c}}]$ admits Theorem \ref{thm:e_gamma_meh_fiber_bundle}'s simplified description and that $\EE_{\Gamma} = [ \Ter_{(\Gamma, \beta)} \cdot \mathfrak{S}_{\vec{c}} / \mathfrak{S}_{\vec{c}}]$ admits a similarly simplified description as a quotient by $\Aut(\Gamma)$.

\begin{lemma} \label{lem:moduli_gamma_as_prod_of_mgns} \hfill
\begin{enumerate}
  \item We have
  \[
    \moduli_{\Gtil} \cong \left[\left(\prod_{v \in K_\Gamma} \moduli_{g(v), n(v) + m(v)}\right) / \Aut(\Gtil) \right].
  \]
  where $n(v)$ denotes the number of distinguished points adjacent to $v$ and $m(v)$ denotes the number of branches incident to $v$.
  \item Similarly,
  \[
    \EE_{\Gamma} \cong [ \Ter_{(\Gamma, \beta)} / \Aut(\Gamma) ].
  \]
\end{enumerate}

\end{lemma}

\begin{proof}
By definition,
\[
  \calH_{(\Gtil, \beta)} \cdot \mathfrak{S}_{\vec{c}} = \bigsqcup_{[\alpha] \in \Aut(\Gtil)\setminus\mathfrak{S}_{\vec{c}}} \calH_{(\Gtil, \beta)} \cdot \alpha.
\]
Let $\moduli_0$ denote the summand $\calH_{(\Gtil, \beta)} \cdot \id$. Observe that
\[
  \moduli_{0} \cong \prod_{v \in K_\Gamma} \moduli_{g(v), n(v) + m(v)}.
\]
By Lemma \ref{lem:action_on_rigid_types}, there is a natural action
of $\Aut(\Gtil)$ on $\calH_{(\Gtil, \beta)}$.
We will define an equivalence of fibered categories $\moduli_{\Gtil} \cong [ (\calH_{(\Gtil, \beta)} \cdot \mathfrak{S}_{\vec{c}}) / \mathfrak{S}_{\vec{c}}] \to [ \moduli_0 / \Aut(\Gtil) ]$.

By definition, an $S$-point of $\moduli_{\Gtil}$ consists of an $\mathfrak{S}_{\vec{c}}$ torsor $p : P \to S$ and an equivariant map
$f : P \to \calH_{(\Gtil, \beta)} \cdot \mathfrak{S}_{\vec{c}}$. Let $P_0 = f^{-1}(\moduli_0)$. By Lemma \ref{lem:action_on_rigid_types}(iii), $p|_{P_0} : P_0 \to S$ is an $\Aut(\Gtil)$-torsor. The restriction $f|_{P_0} : P_0 \to \moduli_0$ is clearly equivariant. Thus, $S \leftarrow P_0 \to \moduli_0$ is a point of $[\moduli_0 / \Aut(\Gtil) ]$. The assignment taking $S \leftarrow P \to \calH_{(\Gtil, \beta)} \cdot \mathfrak{S}_{\vec{c}}$ to $S \leftarrow P_0 \to \moduli_0$ is clearly functorial and respects pullbacks in $S$, so defines a morphism of fibered
categories.

Next, we construct a quasi-inverse. Suppose $S \overset{p_0}{\leftarrow} P_0 \overset{f_0}{\to} \moduli_0$ is a point of $[\moduli_0 / \Aut(\Gtil)]$. Form the contracted products
\[
  P_0 \wedge^{\Aut(\Gtil)} \mathfrak{S}_{\vec{c}} \coloneqq P_0 \times \mathfrak{S}_{\vec{c}} / (p\alpha, \gamma) \sim (p, \alpha\gamma)
\]
and
\[
  \moduli_0 \wedge^{\Aut(\Gtil)} \mathfrak{S}_{\vec{c}} \coloneqq \moduli_0 \times \mathfrak{S}_{\vec{c}} / (x\alpha, \gamma) \sim (x, \alpha\gamma).
\]
The space $P_0 \wedge^{\Aut(\Gtil)} \mathfrak{S}_{\vec{c}}$ has a natural right $\mathfrak{S}_{\vec{c}}$-torsor structure given by $[(p,\gamma)] \cdot \alpha = [(p, \gamma\alpha)]$. Similarly, $\moduli_0 \wedge^{\Aut(\Gtil)} \mathfrak{S}_{\vec{c}}$ has a natural right $\mathfrak{S}_{\vec{c}}$-action given by the same formula.
There is a natural equivariant isomorphism $\moduli_0 \wedge^{\Aut(\Gtil)} \mathfrak{S}_{\vec{c}} \to \calH_{(\Gtil, \beta)} \cdot \mathfrak{S}_{\vec{c}}$ given by $[(x,\alpha)] \mapsto x\cdot \alpha$ and a natural equivariant map $P_0 \wedge^{\Aut(\Gtil)} \mathfrak{S}_{\vec{c}} \to \moduli_0 \wedge^{\Aut(\Gtil)} \mathfrak{S}_{\vec{c}}$ given by $[(x, \alpha)] \mapsto [(f(x), \alpha)].$
As contracted products are functorial, the assignment
\[
\begin{tikzcd}
  P_0 \ar[d] \ar[r] & \moduli_0 \\
  S
\end{tikzcd} \mapsto \begin{tikzcd}
  P_0 \wedge^{\Aut(\Gtil)} \mathfrak{S}_{\vec{c}} \ar[r] \ar[d] & \calH_{(\Gtil,\beta)} \cdot \mathfrak{S}_{\vec{c}} \\
  S
\end{tikzcd}
\]
extends to a morphism of fibered categories $[\moduli_0 / \Aut(\Gtil)] \to \moduli_{\Gtil}$. We leave it to the reader to check that the two
given maps are quasi-inverses to each other.

The claim about $\EE_\Gamma$ follows by identical reasoning.
\end{proof}

Now we may prove the main theorems of the section.

\begin{proof}[Proof of Theorem \ref{thm:e_gamma_fiber_bundle}.]
By Lemma \ref{lem:action_on_rigid_types}, $\calH_{(\Gtil,\beta)}$ admits an action by $\Aut(\Gtil)$, hence also an action by $\Aut(\Gamma)$, since this is a subgroup. By the same Lemma, $\Ter_{(\Gamma,\beta)}$ admits an action by $\Aut(\Gamma)$. In both cases the $\Aut(\Gamma)$ action is induced by considering $\Aut(\Gamma)$ as a subgroup of $\mathfrak{S}_{\vec{c}}$ in a way determined by the rigidification $\beta$. Since the map $\Terrig_{\vec{c}} \to \Hrig_{\vec{c}}$ is $\mathfrak{S}_{\vec{c}}$ equivariant, the induced map $\Ter_{(\Gamma, \beta)} \to \calH_{(\Gtil, \beta)}$ is $\Aut(\Gamma)$-equivariant.

Thus, there is a cartesian square \cite[Tag 04ZN]{stacks-project}
\[
\begin{tikzcd}
 \Ter_{(\Gamma, \beta)} \ar[r] \ar[d] & {}[\Ter_{(\Gamma, \beta)} / \Aut(\Gamma)] \ar[d] \\
  \calH_{(\Gtil, \beta)} \ar[r] & {}[\calH_{(\Gtil, \beta)} / \Aut(\Gamma)].
\end{tikzcd}
\]
By Definition \ref{def:ter_gamma_beta} and Lemma \ref{lem:territory_bundle_decomposition}, the left arrow is a $\Tergl(\calP, g, \vec{c})$-bundle. As the bottom arrow is \'etale, the right arrow is also a $\Tergl(\calP, g, \vec{c})$-bundle.
By Lemma \ref{lem:moduli_gamma_as_prod_of_mgns}, $\EE_{\Gamma} \cong [\Ter_{(\Gamma, \beta)} / \Aut(\Gamma)]$. Clearly, $\calH_{(\Gtil, \beta)} \cong \prod_{v \in K_\Gamma} \moduli_{g(v), n(v) + m(v)},$ so
\[
  [\calH_{(\Gtil, \beta)} / \Aut(\Gamma)] \cong \left[\prod_{v \in K_\Gamma} \moduli_{g(v), n(v) + m(v)} / \Aut(\Gamma)\right] = \moduli_\Gamma.
\]
The result follows.
\end{proof}

Now we prove the second main theorem of this section.

\begin{proof}[Proof of Theorem \ref{thm:e_gamma_meh_fiber_bundle}.]
We must show that the natural map $\EE_\Gamma \to \moduli_{\Gtil}$ is an \'etale
\[
  \bigsqcup_{\Aut(\Gamma) \setminus \Aut(\Gtil)} \Tergl(\calP, g, \vec{c})
\]
bundle.

It suffices to show the claim \'etale locally, so we may reduce to showing the same fact for the base change of the map $\EE_\Gamma \to \moduli_{\Gtil}$ along $\calH_{(\Gtil, \beta)} \cdot \mathfrak{S}_{\vec{c}} \to \moduli_{\Gtil}$, namely $\Ter_{(\Gamma, \beta)} \cdot \mathfrak{S}_{\vec{c}} \to \calH_{(\Gtil, \beta)} \cdot \mathfrak{S}_{\vec{c}}.$ By definition, the source and target of this map can be rewritten as disjoint unions over right cosets
\[
  \bigsqcup_{[\alpha] \in \Aut(\Gamma) \setminus \mathfrak{S}_{\vec{c}}} \Ter_{(\Gamma, \beta \circ \alpha)} \longrightarrow \bigsqcup_{[\alpha] \in \Aut(\Gtil) \setminus \mathfrak{S}_{\vec{c}}} \calH_{(\Gtil, \beta \circ \alpha)}.
\]

Since $\mathfrak{S}_{\vec{c}}$ acts equivariantly on the source and target and transitively on $\Aut(\Gtil) \setminus \mathfrak{S}_{\vec{c}}$, it suffices to check the result over the summand indexed by $\alpha = \id$. The pullback over this summand is
\[
  \bigsqcup_{[\alpha] \in \Aut(\Gamma) \setminus \Aut(\Gtil)} \Ter_{(\Gamma, \beta \circ \alpha)} \to \calH_{(\Gtil, \beta)}.
\]
Finally, recall from Definition \ref{def:ter_gamma_beta} and Lemma \ref{lem:territory_bundle_decomposition} that each $\Ter_{(\Gamma, \beta \circ \alpha)}$ is a $\Tergl(\calP, g, \vec{c})$-bundle over $\calH_{(\Gtil, \beta)}.$
\end{proof}

\begin{example}
Consider the combinatorial type $\Gamma$ of Example \ref{ex:comb_type_of_prestable}. By Example \ref{ex:comb_type_of_prestable} and \ref{ex:c_tilde_type}, we may identify the inclusion $\Aut(\Gamma) \leq \Aut(\Gtil)$ with the inclusion $\langle (1324), (12) \rangle \leq S_4$.
This subgroup has 3 cosets. It is not hard to check that $\Tersing(0, (1,1)) \cong \Spec \ZZ$. Applying Theorem \ref{thm:e_gamma_meh_fiber_bundle}, the geometric fibers of $\EE_\Gamma \to \moduli_{\Gtil}$ are a disjoint union of 3 points. This agrees with our computation of the fiber over $H_{1,1,1,1}$ in the extended example of Section \ref{sec:territory_colliding_sings}. 
\end{example}

The next lemma is a computation to support Example \ref{ex:stable_equinormalized_curves} below. We can think of it as a moduli-theoretic enhancement of Lemma \ref{lem:subalges_of_kn}.

\begin{lemma} \label{lem:territory_of_points}
There is an isomorphism
\[
  \Ter^\delta_{\ZZ^m} \cong \bigsqcup_{\mathcal{P}} \Spec \ZZ
\]
where the disjoint union is over partitions $\mathcal{P}$ of $\{1, \ldots, m\}$ such that $\sum_{P \in \calP} (|P| - 1) = \delta.$
In particular, $\Tergl(\mathcal{P}, 0,(1,\ldots,1)) = \Spec \ZZ$.
\end{lemma}
\begin{proof}
We perform a computation in local coordinates using the description of $\Ter^{\delta}_{\ZZ^m}$ as a closed subscheme of $\Grass(m - \delta, \OO_{\Spec \ZZ}^m)$ from the proof of Theorem \ref{thm:territory_representable}.

Take $e_1, \ldots, e_m$ as the standard basis of $\OO^m$. Work in the affine patch $U = \Spec \ZZ[a_{i,j} : i = 1, \ldots, \delta; j = 1, \ldots, m - \delta]$ of $\Grass(m - \delta, \OO_{\Spec \ZZ}^m)$ representing the rank $m - \delta$ subsheaves of $\OO^m$ with basis $\vec{b}_1, \ldots, \vec{b}_{m - \delta}$ given by the columns of the $m \times (m - \delta)$ matrix
\[
  \begin{bmatrix} \vec{b}_1 & \cdots & \vec{b}_{m - \delta}\end{bmatrix} = \begin{bmatrix}
      \mathbb{I}_{m - \delta} \\
      A
  \end{bmatrix}
\]
where $A = [a_{i,j}]_{1 \leq i \leq \delta, 1 \leq j \leq m - \delta}$.
The cokernel of the inclusion of the universal sub-bundle $\mathcal{U} \to \OO^m$ over this locus is represented by the $\delta \times m$ matrix
\[
  \begin{bmatrix}
       A & -\mathbb{I}_{\delta} 
  \end{bmatrix}
\]

Now, the algebra structure on $\OO_{\Spec \ZZ}^m$ is represented by the linear map
\begin{align*}
  \OO_{\Spec \ZZ}^m \otimes \OO_{\Spec \ZZ}^m &\to \OO_{\Spec \ZZ}^m \\
    e_i \otimes e_j &\mapsto \begin{cases}
        e_i & \text{ if } i = j \\
        0 & \text{ if } i \neq j.
    \end{cases}
\end{align*}

It follows that the induced multiplication on the subbundle $\mathcal{U}$ is represented by
\begin{align*}
  \mathcal{U} \otimes \mathcal{U} &\to \OO_{U}^m \\
   \vec{b}_i \otimes \vec{b}_j &\mapsto \begin{cases}
       e_i + \sum_{k = 1}^{\delta} a_{k,i}^2 e_{k + m - \delta} & \text{ if } i = j\\
       \sum_{k = 1}^{\delta} a_{k,i}a_{k,j} e_{k + m - \delta} & \text{ if }i \neq j.
   \end{cases}
\end{align*}

In order for this multiplication to factor through $\mathcal{U}$, we need the composite with the quotient map $\OO^m \to \OO^m/\mathcal{U}$ to vanish.
That is, multiplying with the matrix $\begin{bmatrix} A & -\mathbb{I}_{\delta} \end{bmatrix}$, we need the equations
\[
  a_{k,i} - a_{k,i}^2 = 0
\]
to hold for $k = 1, \ldots, \delta$, $i = 1, \ldots, m - \delta$ and the equations
\[
  a_{k,i}a_{k,j} = 0
\]
to hold for $k = 1, \ldots, \delta,$ $i = 1, \ldots, m - \delta$, $j = 1, \ldots, m - \delta$ with $i \neq j$.

We must also impose that $(1, \ldots, 1)$ belongs to $\mathcal{U}$, for which we need to impose the equations
\[
  \sum_{j = 1}^{m - \delta} a_{k,j} - 1 = 0
\]
for each $k = 1, \ldots, \delta$.

Altogether, $\Ter^{\delta}_{\ZZ^m} \cap U$ is representable by the scheme
\[
  \Spec \ZZ[a_{i,j}] / I
\]
where $I$ is the ideal generated by the equations above. We observe that the factors $a_{k,i}$ and $1 - a_{k,i}$ of $a_{k,i} - a_{k,i}^2$ are relatively prime, so
we may break $\ZZ[a_{i,j}] / I$ into a product with $2^{\delta(m - \delta)}$ terms using the Chinese remainder theorem.
This amounts to imposing $a_{k,i} = 0$ or $a_{k,i} = 1$ for each $k,i$. The equations $a_{k,i}a_{k,j} = 0$ imply that at most
one of the $a_{k,i}$s for fixed $k$ can be nonzero. Similarly, considering the equations $\sum_{j = 1}^{m - \delta} a_{k,j} - 1 = 0$ we learn that at least one of the $a_{k,j}$s for a fixed $k$ must be nonzero. Altogether, we can enumerate the nonzero terms, each isomorphic to $\ZZ$, by picking for each $k = 1, \ldots, \delta$ for which $j = 1, \ldots, m - \delta$ the coefficient $a_{k,j}$ is 1.

It follows that 
\[
  \Ter^{\delta}_{\ZZ^m}|_{U} \cong \bigsqcup_{(m - \delta)^\delta} \Spec \ZZ
\]
where the disjoint union is over the functions $f : \{1, \ldots, \delta\} \to \{1, \ldots, m - \delta\}$, and the $\Spec \ZZ$ corresponding to $f$ corresponds to the algebra
\[
  \{ (a_1, \ldots, a_m) \in \OO_S^m \mid a_{m - \delta + i} = a_{f(i)} \text{ for each }i = 1,\ldots,\delta\}.
\]
Observe that the sets of indices whose coordinates have equal values form a partition $\calP$ of $\{1, \ldots, m\}$ into $m - \delta$ parts such that $1, \ldots, m - \delta$ belong to distinct parts. A dimension count shows $\sum_{P \in \calP} (|P| - 1) = \delta.$ 

The result follows by taking the union over the remaining open sets of the standard affine cover of $\Grass(m - \delta, \OO_{\Spec \ZZ}^m)$.
\end{proof}

\begin{corollary} \label{cor:e_gamma_is_m_gamma}
If $\Gamma$ is a combinatorial type of a curve, all of whose singularities have genus 0, then
\[
  \EE_\Gamma \cong \moduli_\Gamma.
\]
\end{corollary}
\begin{proof}
Since $\Tergl(\calP, 0, (1,\ldots,1)) \cong \Spec \ZZ$, the projection $\EE_\Gamma \to \moduli_\Gamma$ is an isomorphism.
\end{proof}

\begin{example}\label{ex:stable_equinormalized_curves} (The pre-image of $\Moduli_{g,n}$ under $\EE_{g,n} \to \UU_{g,n}$ is the disjoint union of the usual strata of $\Moduli_{g,n}$)
Let $\Gamma$ be a combinatorial type of a stable curve of genus $g$ with $n$ markings, and let $G$ be the corresponding dual graph. Denote by $\moduli_G$ the usual stratum in $\Moduli_{g,n}$ of stable curves with dual graph $G$. Then $\Aut(\Gamma) = \Aut(G)$, so
\begin{align*}
  \moduli_{\Gamma} &\cong \left[ \prod_{v \in K(\Gamma)} \moduli_{g(v), n(v) + m(v)} / \Aut(\Gamma) \right] \\
    &\cong \left[ \prod_{v \in V(G)} \moduli_{g(v), n(v) + m(v)} / \Aut(G) \right] \\
    &\cong \moduli_G.
\end{align*}

By Corollary \ref{cor:e_gamma_is_m_gamma}, $\EE_{\Gamma} \to \moduli_{\Gamma}$ is an isomorphism. The induced map $\moduli_{G} \cong \EE_{\Gamma} \to \Moduli_{g,n}$ is clearly identified
with the clutching morphism $\moduli_G \to \Moduli_{g,n}$. Since stable curves are Gorenstein, the total conductance $c$ of $\Gamma$ is $2\delta$. Finally, examining the defining equations of $\EE^{\delta,2\delta}_{g,n}$ in $\EE_{g,n}$ and $\EE_{\Gamma}$ inside $\EE^{\delta,2\delta}_{g,n}$ shows that the inclusion $\EE_{\Gamma} \to \EE_{g,n}$ is open.

Therefore, as $\Gamma$ varies over the combinatorial types of $n$-marked stable curves of genus $g$, the images of the $\EE_{\Gamma}$ in $\EE_{g,n}$ form an open substack isomorphic to $\bigsqcup_G \moduli_{G}$. As this clearly agrees with the preimage of $\Moduli_{g,n}$ in $\EE_{g,n}$ on geometric points, and an open substack is uniquely determined by its geometric points, we conclude that the pre-image of $\Moduli_{g,n}$ in $\EE_{g,n}$ is isomorphic to $\bigsqcup_{G} \moduli_G$, as claimed.
\end{example}

\section{Some computations in $\EE^{2,4}_{2,0}$}
\label{sec:territory_colliding_sings}

We conclude the paper with an extended example, studying the structure of the space $\EE^{2,4}_{2,0}$ parametrizing equinormalized curves of arithmetic genus 2 with total delta invariant 2 and only Gorenstein singularities. We continue to work in the category of $\QQ$-schemes, so for brevity, we will write $A_{\vec{c}}$
instead of $A_{\vec{c}} \otimes \QQ$.

\subsection{Computation of relevant territories} \label{ssec:territory_computations}
We start by computing the various relevant territories, namely $\Ter^{1,2}_{A}$ for $A = A_{(2)}$, $A_{(1,1)}$ as well as $\Ter^{2,4}_{A}$ for $A = A_{(1,1,1,1)}$, $A_{(2,1,1)}$, $A_{(2,2)}$, $A_{(3,1)}$, $A_{(4)}$. Taking appropriate connected components, we will find the territories of gluing data that appear as the fibers of strata. Recall that, given a choice of subalgebra $B \subseteq A \otimes k$, and a closed subscheme $Z \cong \Spec A \otimes k$ inside of a smooth curve $\tilde{C}$ one finds the corresponding curve singularities by taking the pushout
\[
\begin{tikzcd}
    \Spec A \otimes k \ar[r] \ar[d] & \tilde{C} \ar[d,"\nu"] \\
    \Spec B \ar[r] & C
\end{tikzcd}
\]
as in the proof of Theorem \ref{thm:territory_of_scheme_representable}. 

\bigskip

The computations with $\delta = 1$ are easy.

\medskip

$\vec{\mathbf{c}} = \mathbf{(2)}$: There is only one $\OO_S$-subalgebra of $A_{(2)} \otimes \OO_S = \OO_S[t]/(t^2)$, namely $\OO_S$, so $\Ter^{1,2}_{A_{(2)}} \cong \Spec \QQ$. Its unique point corresponds to a cusp singularity.

\medskip

$\vec{\mathbf{c}} = \mathbf{(1,1)}$: There is only one $\OO_S$-subalgebra of $A_{(1,1)} \otimes \OO_S = \OO_S \times \OO_S$, namely $\OO_S \cdot (1,1)$, so $\Ter^{1,2}_{A_{(1,1)}} \cong \Spec \QQ$. Its unique point corresponds to a node. 

\medskip

For the computations with $\delta = 2$, our strategy is first to find equations for subalgebras $B$ with a basis of the form $1, ae_1 + be_2 + ce_3$ where $1, e_1, e_2, e_3$ is a basis for the containing algebra $A$. Then, taking $\Proj$, we quotient by the scaling action of $\GG_m$ on the second basis element to forget the choice of basis. The results are illustrated in Figure \ref{fig:proper_territory_example}.

\begin{figure}
\begin{center}
{\renewcommand{\arraystretch}{2}
\setlength{\tabcolsep}{12pt}
\begin{tabular}{m{1.7cm}|m{2cm}m{1.8cm}m{1.3cm}m{2.3cm}m{1.8cm}}
 $\vec{c}$ & $(1,1,1,1)$ & $(2,1,1)$ & $(3,1)$ & $(2,2)$ & $(4)$ \\ \hline \\[-.5cm]
  \begin{minipage}{1.7cm}$\Ter^2_{A_{\vec{c}}}$\end{minipage} &
  \includegraphics[scale=.45]{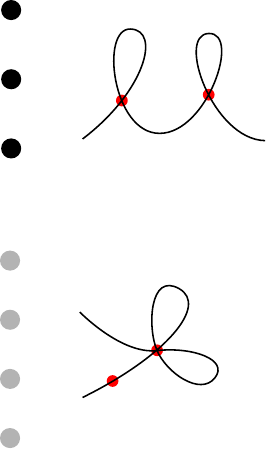} & 
  \includegraphics[scale=.45]{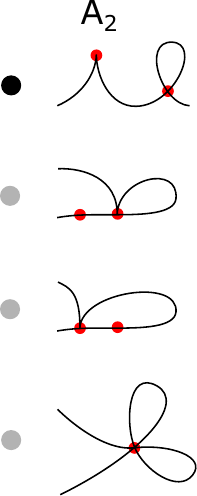} &
  \includegraphics[scale=.45]{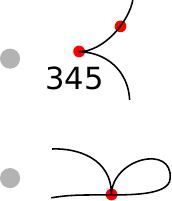} &
  \includegraphics[scale=.45]{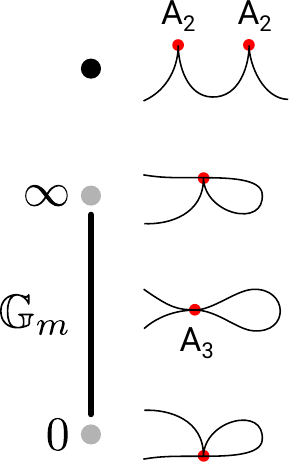} &
  \includegraphics[scale=.45]{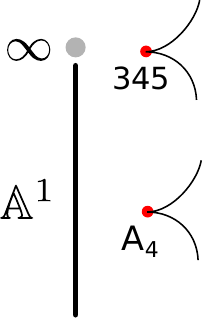} 
\end{tabular}
}
\end{center}
\caption{Sketches of the territory $\Ter^2_{A_{\vec{c}}}$ together with sketches of the associated singular curves $C$. The singularity isomorphic to $k\llbracket t^3, t^4, t^5\rrbracket$ is labeled ``345.'' Black points of the territory indicate singularities with conductance 4, i.e. points of $\Ter^{2,4}_{A_{\vec{c}}}$, while lighter grey points of the territory correspond to singularities with conductance 3, i.e. points of $\Ter^{2,3}_{A_{\vec{c}}}$.}
\label{fig:proper_territory_example}
\end{figure}

\medskip

$\vec{\mathbf{c}} = \mathbf{(1,1,1,1)}$: A basis of $A_{(1,1,1,1)} = \QQ^4$ is
\[
  (1,1,1,1), (0,1,0,0), (0,0,1,0), (0,0,0,1).
\]
Consider an $\OO_S$-submodule of $\OO_S \otimes_\QQ \QQ^4$ with basis $1, a(0,1,0,0) + b(0,0,1,0) + c(0,0,0,1)$. Now,
\[
  (a(0,1,0,0) + b(0,0,1,0) + c(0,0,0,1))^2 = a^2(0,1,0,0) + b^2(0,0,1,0) + c^2(0,0,0,1),
\]
so the submodule is an $\OO_S$-subalgebra if and only if the maximal minors of
\[
  \begin{pmatrix}
     0 & a^2 & b^2 & c^2 \\
     0 & a & b & c \\
     1 & 0 & 0 & 0
  \end{pmatrix}
\]
vanish. Equivalently, the maximal minors of
\[
  \begin{pmatrix}
    a^2 & b^2 & c^2 \\
    a & b & c
  \end{pmatrix}
\]
vanish. We conclude that
\[
  \Ter^2_{A_{(1,1,1,1)}} \cong \Proj \QQ[a, b, c]/(a^2b - ab^2, a^2c - ac^2, b^2c - bc^2).
\]
It has seven reduced $\Spec \QQ$ points. Three, corresponding to the subalgebras
\begin{align*}
  \QQ(1,1,1,1) + \QQ(0,1,1,0) & \quad (a = 1, b = 1, c = 0) \\
  \QQ(1,1,1,1) + \QQ(0,1,0,1) & \quad (a = 1, b = 0, c = 1) \\
  \QQ(1,1,1,1) + \QQ(0,0,1,1) & \quad (a = 0, b = 1, c = 1)
\end{align*}
correspond to curves with two nodes: two pairs of the points of $Z$ are identified under $\nu : \tilde{C} \to C$. For these subalgebras the conductances are all 1, so the total conductance is 4. This agrees with the fact that nodes are Gorenstein, see Lemma \ref{lem:conductor_bounds_and_gorenstein}.
The remaining four points of $\Ter^2_{A_{(1,1,1,1)}}$ correspond to the subalgebras
\begin{align*}
  \QQ(1,1,1,1) + \QQ(0,1,0,0) & \quad (a = 1, b = 0, c = 0), \\
  \QQ(1,1,1,1) + \QQ(0,0,1,0) & \quad (a = 0, b = 1, c = 0), \\
  \QQ(1,1,1,1) + \QQ(0,0,0,1) & \quad (a = 0, b = 0, c = 1), \\
  \QQ(1,1,1,1) + \QQ(0,1,1,1) & \quad (a = 1, b = 1, c = 1),
\end{align*}
which correspond in turn to curves with a single ordinary 3-fold point: three of the points of $Z$ are identified by $\nu : \tilde{C} \to C$, while the fourth point is ``spurious,'' as it is not crimped to a singularity. For each of these subalgebras exactly one of the conductances is 0, and accordingly the corresponding singularities are non-Gorenstein.

Therefore $\Ter^{2,4}_{A_{(1,1,1,1)}}$ consists of only the first three points corresponding to pairs of nodes. In terms of gluing data, the first point is $\Tergl(14/23, (0,0), (1,1,1,1))$, the second is $\Tergl(13/24, (0,0), (1,1,1,1))$, and the third is $\Tergl(12/34, (0,0), (1,1,1,1))$.

\medskip

$\vec{\mathbf{c}} = \mathbf{(2,1,1)}$: A basis of $A_{(2,1,1)} \cong \QQ[t]/(t^2) \times \QQ \times \QQ$ is given by
\[
  (1,1,1), (t,0,0), (0,1,0), (0,0,1).
\]
Since
\[
  (a(t,0,0) + b(0,1,0) + c(0,0,1))^2 = b^2(0,1,0) + c^2(0,0,1),
\]
we compute maximal minors of
\[
  \begin{pmatrix}
    0 & b^2 & c^2 \\
    a & b & c
  \end{pmatrix}
\]
and conclude
\[
  \Ter^2_{A_{(2,1,1)}} \cong \Proj \QQ[a,b,c]/(ab^2,ac^2,b^2c - bc^2).
\]

The $a = 1$ chart, isomorphic to $\Spec \QQ[b,c]/(b^2,c^2)$, possesses a single non-reduced point. It has residue field $\QQ$, and the unique $\QQ$-point corresponds to the algebra
\[
  \QQ(1,1,1) + \QQ(t,0,0) \quad (a = 1, b = 0, c = 0).
\]
The corresponding curve singularity is an ordinary triple point, with conductances $1,1,1$. Accordingly this singularity is non-Gorenstein. 

There are three remaining points. They are reduced, corresponding to the algebras
\begin{align*}
  k(1,1,1) + k(0,1,0) & \quad (a = 0, b = 1, c = 0)\\
  k(1,1,1) + k(0,0,1) & \quad (a = 0, b = 0, c = 1) \\
  k(1,1,1) + k(0,1,1) & \quad (a = 0, b = 1, c = 1).
\end{align*}
The first two of these correspond to curves with a cusp glued transversely to a smooth branch. Their conductances are respectively $2,0,1$ and $2,1,0$. Neither is Gorenstein. The last corresponds to a curve with a cusp at the first point and a node joining the last two points. Both singularities are Gorenstein. The cusp has genus 1 and conductance 2; the node has genus 0 and conductances 1, 1.

We conclude that $\Ter^{2,4}_{A_{(2,1,1)}} \cong \Spec \QQ$. In terms of gluing data, the unique point is equal to $\Tergl(1/23, (1,0), (2,1,1))$.

\medskip

$\vec{\mathbf{c}} = \mathbf{(3,1)}$: A basis of $A_{(3,1)} = \QQ[t]/(t^3) \times \QQ$ is given by
\[
  (1,1), (t,0), (t^2,0), (0,1).
\]
Since
\[
  (a(t,0) + b(t^2,0) + c(0,1))^2 = a^2(t^2,0) + c^2(0,1),
\]
we compute the maximal minors of
\[
  \begin{pmatrix}
    0 & a^2 & c^2 \\
    a & b & c
  \end{pmatrix}
\]
and conclude
\[
  \Ter^2_{A_{(3,1)}} \cong \Proj \QQ[a, b, c]/(a^3, ac^2, a^2c - bc^2).
\]

The $a = 1$ chart is empty. The $b = 1$ chart, isomorphic to $\Spec \QQ[a,c]/(a^3,a^2c,c^2)$, possesses a single non-reduced point with residue field $\QQ$. The non-reduced point corresponds to the subalgebra
\[
  \QQ(1,1) + \QQ(t^2,0) \quad (a = 0, b = 1, c = 0). 
\]
This subalgebra corresponds in turn to a singularity with a cusp glued transversely to a smooth branch. This is a non-Gorenstein singularity of genus 1, with conductances 2,1.

The $c = 1$ chart possesses a single reduced point, corresponding to the subalgebra
\[
  k(1,1) + k(0,1) \quad (a = 0, b = 0, c = 1).
\]
A corresponding equinormalized curve $\nu : \tilde{C} \to C$ sends the non-reduced point of $Z$ to a $\QQ[t^3,t^4,t^5]$-singularity, which is non-Gorenstein of genus 2 with a single branch conductance equal to 3. Meanwhile, $\nu : \tilde{C} \to C$ sends the reduced point of $Z$ to a smooth point.

We conclude that $\Ter^{2,4}_{A_{(3,1)}}$ is empty: there are no Gorenstein singularities with conductances $3,1$.

\medskip

$\vec{\mathbf{c}} = \mathbf{(2,2)}$:
A basis for $A_{(2,2)} = \QQ[t_1]/(t_1^2) \times \QQ[t_2]/(t_2^2)$ is given by
\[
  (1,1), (t_1, 0), (0,1), (0,t_2).
\]
Since
\[
  (a(t_1,0) + b(0,1) + c(0,t_2))^2 = b^2(0,1) + 2bc(0,t_2),
\]
we take the maximal minors of
\[
  \begin{pmatrix}
    0 & b^2 & 2bc \\
    a & b & c
  \end{pmatrix}
\]
and conclude
\[
  \Ter^2_{A_{(2,2)}} \cong \Proj \QQ[a,b,c]/(ab^2,2abc,b^2c).
\]

The $b = 1$ chart possesses a single reduced point, parametrizing the subalgebra
\[
  k(1,1) + k(0,1) \quad (a = 0, b = 1, c = 0).
\]
A corresponding equinormalized curve $\nu : \tilde{C} \to C$ crimps each of the points of $Z$ to a cusp on $C$. Each of the two cusps is a Gorenstein singularity of genus 1 with a single branch conductance of 2.

The $a = 1$ chart is isomorphic to $\Spec \QQ[b,c]/(b^2,2bc)$. Since 2 is invertible in $\QQ$, then this is $\AA^1$ together with an embedded point at the origin. The $\QQ$-points parametrize
subalgebras of the form
\[
  \QQ(1,1) + \QQ(t_1,ct_2) \quad (a = 1, b = 0, c \text{ free}).
\]
When $c \neq 0$, the corresponding curve has a tacnode, a Gorenstein singularity of genus 1. When $c = 0,$ the corresponding curve has a cusp glued transversely to a smooth branch.

The $c = 1$ chart is symmetric to the $a = 1$ chart and parametrizes algebras of the form
\[
  \QQ(1,1) + \QQ(at_1, t_2) \quad (a \text{ free}, b = 0, c = 1).
\]

Observe that the reduction of $\Ter^2_{A_{(2,2)}}$ consists of the union of a point with a $\PP^1$. Restricting to Gorenstein singularities, we have $\Ter^{2,4}_{A_{(2,2)}} \cong \Spec \QQ \sqcup \GG_m$. In terms of gluing data, $\Tergl(1/2, (1,1), (2,2))$ is the $\Spec \QQ$, while $\Tergl(12, 1, (2,2))$ is the $\GG_m$.

We remark that the $a = 1$ and $b = 1$ charts together show that $\Tergl(12, 1, (2,2))$ is an infinitesimal thickening of $\Ter^{1}_{A^+_{(2,2)}}$, which is isomorphic to $\PP^1$ by Example \ref{ex:crimping_all_2s}. Thus, Theorem \ref{thm:decomposition_into_territories_of_sings_no_conductance} cannot be strengthened to an isomorphism.

$\vec{\mathbf{c}} = \mathbf{(4)}:$ The algebra $A_{(4)} = \QQ[t]/t^4$ has the basis $1, t, t^2, t^3$. Since
\[
  (at + bt^2 + ct^3)^2 = a^2t^2 + 2abt^3,
\]
we take the maximal minors of the matrix
\[
  \begin{pmatrix}
     0 & a^2 & 2ab \\
     a & b & c
  \end{pmatrix}
\]
and conclude that $\Ter^2_{A_{(4)}} \cong \Proj \QQ[a,b,c]/(a^3,2a^2b, a^2c - 2ab^2)$. This is non-reduced, but its reduction is $\Proj \QQ[b,c] \cong \PP^1$. A point $[b : c] \in \PP^1(k)$ corresponds to the subalgebra $k \cdot 1 + k \cdot (bt^2 + ct^3).$ The corresponding singularity on $C$ is a ramphoid cusp for $b \neq 0$ and a $k[t^3,t^4,t^5]$-singularity for $b = 0$. (Cf. \cite[Lemma 2.1]{battistella_m_stable}.) Ramphoid cusps are Gorenstein, while the $k[t^3,t^4,t^5]$-singularity is not. Thus, since $\QQ$ has characteristic 0,
\[
  \Ter^{2,4}_{A_{(4)}} = \Tergl(1,2,(4)) \cong \Spec \QQ[a,c] / (a^3, a^2, a^2c - 2a)  \cong \AA^1. 
\]

\subsection{Combinatorial types and strata of $\EE^{2,4}_{2,0}$}
Next, we consider the possible combinatorial types $\Gamma$ for $\EE^{2,4}_{2,0}$. Since
\begin{align*}
  g(\Gamma) &= \sum_{k \in K} g(k) + \sum_{s \in S} \delta(s) - \# K + 1\\
  \implies \# K &= \sum_{k \in K} g(k) + 1,
\end{align*}
we conclude that $\Gamma$ can have 1 component of genus 0, one component of genus 1 and one of genus 0, one component of genus 2 and two of genus 0, or two components of genus 1 and one of genus 0. We know from our computations of territories above that the singularities of $\Gamma$ can only have the following invariants:
\begin{center}
\begin{tabular}{|c|c|c|c|} \hline
  $g$ & $m$ & $\delta$ & $\vec{c}$ \\ \hline
  0 & 2 & 1 & $(1,1)$ \\ \hline
  1 & 1 & 1 & $(2)$ \\ \hline
  1 & 2 & 2 & $(2,2)$ \\ \hline
  2 & 1 & 2 & $(4)$ \\ \hline
\end{tabular}
\end{center}

With these constraints, we find that there are 15 possible combinatorial types, which we display in Figure \ref{fig:E_24_20_combinatorial_types}. For four of these combinatorial types $\moduli_{\Gamma}$ and $\moduli_{\Gtil}$ differ, see Figure \ref{fig:E_24_20_combinatorial_types_circ}. 

\begin{figure}
\begin{center}
\begin{tabular}{| >{\centering\arraybackslash}m{.5in} | >{\centering\arraybackslash}m{2in}| c | c |} \hline
Name & Combinatorial type & $\moduli_{\Gamma}$ & $\EE_{\Gamma} \to \moduli_{\Gamma}$ fiber \\ \hline
$\Gamma_1$ & \begin{tikzpicture}[scale=1]
\draw (0,0) to[out=40,in=140] (1,0)
      (0,0) to[out=-40,in=220] (1,0)
      (1,0) to[out=40,in=140] (2,0)
      (1,0) to[out=-40,in=220] (2,0);
\draw[fill=white]
  (0,0) +(-.15,-.15) rectangle +(.15,.15)
  (1,0) circle (5pt)
  (2,0) +(-.15,-.15) rectangle +(.15,.15);
\end{tikzpicture} 
& $[\moduli_{0,4} / D_8]$ & $\Spec \QQ$ \\ \hline
$\Gamma_2$ & \begin{tikzpicture}[scale=1]
\draw (0,0) to[out=40,in=140] (1,0)
      (0,0) to[out=-40,in=220] (1,0)
      (1,0) --node[above]{\tiny $2$} (2,0);
\draw[fill=white]
  (0,0) +(-.15,-.15) rectangle +(.15,.15)
  (1,0) circle (5pt)
  (2,0) +(-.2,-.2) rectangle +(.2,.2);
\node at (2,0) {\tiny $1$};
\end{tikzpicture} 
& $[\moduli_{0,3} / \mathfrak{S}_2]$ & $\Spec \QQ$ \\ \hline
$\Gamma_3$ & \begin{tikzpicture}[scale=1]
\draw (0,0) --node[above]{\tiny $2$} (1,0)
      (1,0) --node[above]{\tiny $2$} (2,0);
\draw[fill=white]
  (0,0) +(-.2,-.2) rectangle +(.2,.2)
  (1,0) circle (5pt)
  (2,0) +(-.2,-.2) rectangle +(.2,.2);
\node at (0,0) {\tiny $1$};
\node at (2,0) {\tiny $1$};
\end{tikzpicture} 
& $[\moduli_{0,2} / \mathfrak{S}_2]$ & $\Spec \QQ$ \\ \hline
$\Gamma_4$ & \begin{tikzpicture}[scale=1]
\draw (0,0) to[out=40,in=140]node[above]{\tiny $2$} (1,0)
      (0,0) to[out=-40,in=220]node[below]{\tiny $2$} (1,0);
\draw[fill=white]
  (0,0) circle (5pt)
  (1,0) +(-.2,-.2) rectangle +(.2,.2);
\node at (1,0) {\tiny $1$};
\end{tikzpicture} 
& $[\moduli_{0,2} / \mathfrak{S}_2]$ & $\GG_m$ \\ \hline
$\Gamma_5$ & \begin{tikzpicture}[scale=1]
\draw (0,0) --node[above]{\tiny $4$} (1,0);
\draw[fill=white]
  (0,0) circle (5pt)
  (1,0) +(-.2,-.2) rectangle +(.2,.2);
\node at (1,0) {\tiny $2$};
\end{tikzpicture} 
& $\moduli_{0,1}$ & $\AA^1$ \\ \hline
$\Gamma_6$ & \begin{tikzpicture}[scale=1]
\draw (0,0) to[out=40,in=140] (1,0)
      (0,0) to[out=-40,in=220] (1,0)
      (1,0) -- (2,0)
      (2,0) -- (3,0);
\draw[fill=white]
  (0,0) +(-.15,-.15) rectangle +(.15,.15)
  (1,0) circle (5pt)
  (2,0) +(-.15,-.15) rectangle +(.15,.15)
  (3,0) circle (5pt);
\node at (1,0) {\tiny $1$};
\end{tikzpicture} 
& $[\moduli_{1,3} \times \moduli_{0,1} / \mathfrak{S}_2]$ & $\Spec \QQ$ \\ \hline
$\Gamma_7$ & \begin{tikzpicture}[scale=1]
\draw (0,0) -- (1,.5)
      (0,0) -- (1,-.5)
      (1,.5) -- (2,0)
      (1,-.5) -- (2,0);
\draw[fill=white]
  (0,0) circle (5pt)
  (1,.5) +(-.15,-.15) rectangle +(.15,.15)
  (1,-.5) +(-.15,-.15) rectangle +(.15,.15)
  (2,0) circle (5pt);
\node at (0,0) {\tiny $1$};
\end{tikzpicture} 
& $[\moduli_{1,2} \times \moduli_{0,2} / \mathfrak{S}_2]$ & $\Spec \QQ$ \\ \hline
$\Gamma_8$ & \begin{tikzpicture}[scale=1]
\draw (0,0) -- (2,0)
      (2,0) to[out=40,in=140] (3,0)
      (2,0) to[out=-40,in=220] (3,0);
\draw[fill=white]
  (0,0) circle (5pt)
  (1,0) +(-.15,-.15) rectangle +(.15,.15)
  (2,0) circle (5pt)
  (3,0) +(-.15,-.15) rectangle +(.15,.15);
\node at (0,0) {\tiny $1$};
\end{tikzpicture} 
& $[\moduli_{1,1} \times \moduli_{0,3} / \mathfrak{S}_2]$ & $\Spec \QQ$ \\ \hline
$\Gamma_9$ & \begin{tikzpicture}[scale=1]
\draw (0,0) --node[above] {\tiny $2$} (1,0)
      (1,0) -- (3,0);
\draw[fill=white]
  (0,0) +(-.2,-.2) rectangle +(.2,.2)
  (1,0) circle (5pt)
  (2,0) +(-.15,-.15) rectangle +(.15,.15)
  (3,0) circle (5pt);
\node at (0,0) {\tiny $1$};
\node at (1,0) {\tiny $1$};
\end{tikzpicture} 
& $\moduli_{1,2} \times \moduli_{0,1}$ & $\Spec \QQ$ \\ \hline
$\Gamma_{10}$ & \begin{tikzpicture}[scale=1]
\draw (0,0) -- (2,0)
      (2,0) --node[above] {\tiny $2$} (3,0);
\draw[fill=white]
  (0,0) circle (5pt)
  (1,0) +(-.15,-.15) rectangle +(.15,.15)
  (2,0) circle (5pt)
  (3,0) +(-.2,-.2) rectangle +(.2,.2);
\node at (0,0) {\tiny $1$};
\node at (3,0) {\tiny $1$};
\end{tikzpicture} 
& $\moduli_{1,1} \times \moduli_{0,2}$ & $\Spec \QQ$ \\ \hline
$\Gamma_{11}$ & \begin{tikzpicture}[scale=1]
\draw (0,0) --node[above] {\tiny $2$} (1,0)
      (1,0) --node[above] {\tiny $2$} (2,0);
\draw[fill=white]
  (0,0) circle (5pt)
  (1,0) +(-.2,-.2) rectangle +(.2,.2)
  (2,0) circle (5pt);
\node at (0,0) {\tiny $1$};
\node at (1,0) {\tiny $1$};
\end{tikzpicture} 
& $\moduli_{1,1} \times \moduli_{0,1}$ & $\GG_m$ \\ \hline
$\Gamma_{12}$ & \begin{tikzpicture}[scale=1]
\draw (0,0) -- (4,0);
\draw[fill=white]
  (0,0) circle (5pt)
  (1,0) +(-.15,-.15) rectangle +(.15,.15)
  (2,0) circle (5pt)
  (3,0) +(-.15,-.15) rectangle +(.15,.15)
  (4,0) circle (5pt);
\node at (0,0) {\tiny $2$};
\end{tikzpicture} 
& $\moduli_{2,0} \times \moduli_{0,1}^2$ & $\Spec \QQ$ \\ \hline
$\Gamma_{13}$ & \begin{tikzpicture}[scale=1]
\draw (0,0) -- (4,0);
\draw[fill=white]
  (0,0) circle (5pt)
  (1,0) +(-.15,-.15) rectangle +(.15,.15)
  (2,0) circle (5pt)
  (3,0) +(-.15,-.15) rectangle +(.15,.15)
  (4,0) circle (5pt);
\node at (2,0) {\tiny $2$};
\end{tikzpicture} 
& $[\moduli_{2,0} \times \moduli_{0,1}^2 / \mathfrak{S}_2]$ & $\Spec \QQ$ \\ \hline
$\Gamma_{14}$ & \begin{tikzpicture}[scale=1]
\draw (0,0) -- (4,0);
\draw[fill=white]
  (0,0) circle (5pt)
  (1,0) +(-.15,-.15) rectangle +(.15,.15)
  (2,0) circle (5pt)
  (3,0) +(-.15,-.15) rectangle +(.15,.15)
  (4,0) circle (5pt);
\node at (0,0) {\tiny $1$};
\node at (2,0) {\tiny $1$};
\end{tikzpicture} 
& $\moduli_{1,1} \times \moduli_{1,1} \times \moduli_{0,1}$ & $\Spec \QQ$ \\ \hline
$\Gamma_{15}$ & \begin{tikzpicture}[scale=1]
\draw (0,0) -- (4,0);
\draw[fill=white]
  (0,0) circle (5pt)
  (1,0) +(-.15,-.15) rectangle +(.15,.15)
  (2,0) circle (5pt)
  (3,0) +(-.15,-.15) rectangle +(.15,.15)
  (4,0) circle (5pt);
\node at (0,0) {\tiny $1$};
\node at (4,0) {\tiny $1$};
\end{tikzpicture} 
& $[\moduli_{1,1}^2 \times \moduli_{0,2} / \mathfrak{S}_2]$ & $\Spec \QQ$ \\ \hline
\end{tabular}
\end{center}
\caption{Combinatorial types in $\EE^{2,4}_{2,0}$. We label edges with their conductance and vertices with their genus. We omit conductances equal to 1 and genera equal to 0.}
\label{fig:E_24_20_combinatorial_types}
\end{figure}

\begin{figure}
\begin{center}
\begin{tabular}{| >{\centering\arraybackslash}m{.5in} | >{\centering\arraybackslash}m{2in}| c | c |} \hline
Name & Combinatorial type & $\moduli_{\Gtil}$ & $\EE_{\Gamma} \to \moduli_{\Gtil}$ fiber \\ \hline
$\Gamma_1$ & \begin{tikzpicture}[scale=1]
\draw (0,0) to[out=40,in=140] (1,0)
      (0,0) to[out=-40,in=220] (1,0)
      (1,0) to[out=40,in=140] (2,0)
      (1,0) to[out=-40,in=220] (2,0);
\draw[fill=white]
  (0,0) +(-.15,-.15) rectangle +(.15,.15)
  (1,0) circle (5pt)
  (2,0) +(-.15,-.15) rectangle +(.15,.15);
\end{tikzpicture} 
& $[\moduli_{0,4} / \mathfrak{S}_4]$ & $\Spec \QQ \sqcup \Spec \QQ \sqcup \Spec \QQ$ \\ \hline
$\Gamma_7$ & \begin{tikzpicture}[scale=1]
\draw (0,0) -- (1,.5)
      (0,0) -- (1,-.5)
      (1,.5) -- (2,0)
      (1,-.5) -- (2,0);
\draw[fill=white]
  (0,0) circle (5pt)
  (1,.5) +(-.15,-.15) rectangle +(.15,.15)
  (1,-.5) +(-.15,-.15) rectangle +(.15,.15)
  (2,0) circle (5pt);
\node at (0,0) {\tiny $1$};
\end{tikzpicture} 
& $[\moduli_{1,2} \times \moduli_{0,2} / \mathfrak{S}_2 \times \mathfrak{S}_2 ]$ & $\Spec \QQ \sqcup \Spec \QQ$ \\ \hline
$\Gamma_{13}$ & \begin{tikzpicture}[scale=1]
\draw (0,0) -- (4,0);
\draw[fill=white]
  (0,0) circle (5pt)
  (1,0) +(-.15,-.15) rectangle +(.15,.15)
  (2,0) circle (5pt)
  (3,0) +(-.15,-.15) rectangle +(.15,.15)
  (4,0) circle (5pt);
\node at (2,0) {\tiny $2$};
\end{tikzpicture} 
& $[\moduli_{2,0} \times \moduli_{0,1}^2 / \mathfrak{S}_2 \times \mathfrak{S}_2]$ & $\Spec \QQ \sqcup \Spec \QQ$ \\ \hline
$\Gamma_{15}$ & \begin{tikzpicture}[scale=1]
\draw (0,0) -- (4,0);
\draw[fill=white]
  (0,0) circle (5pt)
  (1,0) +(-.15,-.15) rectangle +(.15,.15)
  (2,0) circle (5pt)
  (3,0) +(-.15,-.15) rectangle +(.15,.15)
  (4,0) circle (5pt);
\node at (0,0) {\tiny $1$};
\node at (4,0) {\tiny $1$};
\end{tikzpicture} 
& $[\moduli_{1,1}^2 \times \moduli_{0,2} / \mathfrak{S}_2 \times \mathfrak{S}_2]$ & $\Spec \QQ \sqcup \Spec \QQ$ \\ \hline
\end{tabular}
\end{center}
\caption{Combinatorial types in $\EE^{2,4}_{2,0}$ for which $\moduli_\Gamma$ differs from $\moduli_{\Gtil}$.}
\label{fig:E_24_20_combinatorial_types_circ}
\end{figure}

\subsection{A closer look at curves normalized by $\PP^1$} \label{ssec:curves_normalized_by_p1}
For $\delta = 2, c =4$, the construction of Remark \ref{rem:e_bounded_conductance} goes as follows: We let $H$ be the relative Hilbert scheme of $4$ points in the universal curve of $\moduli^{2}_{2,0}$. We write $\tilde{C}_H$ for the universal curve over $H$ and $Z_H$ for its universal closed subscheme. Then $\Ter^2_{Z_H/\tilde{C}_H/H}$ parametrizes equinormalized curves $\nu : \tilde{C}_H \to C_H$ formed by crimping and gluing the length 4 subscheme $Z_H$ to a length $c - \delta = 2$ subscheme $\nu(Z_H)$ of $C_H$. We then set $\Ter^{2,\leq 4}_{2,0}$ to be the open and closed substack of $\Ter^{2}_{Z_H/\tilde{C}_H/H}$
on which the crimped curve $C$ is connected of arithmetic genus $2$. Finally, by Theorem \ref{thm:algebraicity}, $\EE^{2,4}_{2,0}$ is identified with the open substack of $\Ter^{2,\leq 4}_{2,0}$ on which the total conductance is exactly 4. We have a sequence of maps
\begin{equation} \label{eq:the_construction}
  \EE^{2,4}_{2,0} \subseteq \Ter^{2,\leq 4}_{2,0} \hookrightarrow \Ter^{2}_{Z_H/\tilde{C}_H/H} \to H \to \moduli^{2}_{2,0},
\end{equation}
where the first is an open immersion, and all others are proper.

We narrow our focus to the connected component $\moduli_{0,0} = [ \Spec \QQ / PGL(2) ]$ of $\moduli^{2}_{2,0}$ over which $\EE_{\Gamma}$ lives for $\Gamma_1, \ldots, \Gamma_5$. Consider the smooth cover $x = \Spec \QQ \to \moduli_{0,0}$ corresponding to the unmarked curve $\PP^1_{\QQ}$. We take the fiber of the whole construction \eqref{eq:the_construction} over $x$ to find a smooth-local description of $\EE^{2,4}_{2,0}$.
Write $\tilde{C} = \PP^1_\QQ$ for the pullback of the universal curve to $x$ and $Z \to \tilde{C}$ for the universal closed subscheme of degree $4$. To avoid clutter, we abuse notation by referring henceforth to $H|_x$ as $H$. Observe that $H$ is $\Hilb^{4}_{\PP^1_\QQ} \cong (\PP^1_\QQ)^{[4]} \cong \PP^4_\QQ$. Observe also that any curve normalized by $\tilde{C}$ is connected, so $\Ter^{2,\leq 4}_{2,0}|_x = \Ter^{2}_{Z/\tilde{C}/H}$.

In principle $\Ter^{2}_{Z/\tilde{C}/H}$ can be computed explicitly on sufficiently small open subsets of $H$, but the results are not particularly human-interpretable. For example, consider the chart $U = \Spec \QQ[b,c,d,e] \subseteq H$ parametrizing the subschemes of $\PP^1$ not meeting the point at infinity. Then $Z|_U \cong \Spec \QQ[b,c,d,e][t]/(t^4 + bt^3 + ct^2 + dt + e)$. The ring $\QQ[b,c,d,e][t]/(t^4 + bt^3 + ct^2 + dt + e)$ admits a $\QQ[b,c,d,e]$-basis $1, t, t^2, t^3$. Then, as in our computations above, we can compute the expression for $(xt + yt^2 + zt^3)^2$ in terms of this basis, then take Proj of the maximal minors of an appropriate matrix to check that it is in the span of 1 and $xt + yt^2 + zt^3$. With the aid of Macaulay2, we find
\[
  \Ter^{2}_{Z/\tilde{C}/H|_x}|_U \cong \Proj \QQ[b,c,d,e,x,y,z] / (f_1,f_2,f_3)
\]
where $b,c,d,e$ have degree 0, $x,y,z$ have degree 1, and the generators of the ideal are
\begin{align*}
f_1 &= b^{3}yz^{2}-b^{2}cz^{3}-2b^{2}y^{2}z+c^{2}z^{3}+bdz^{3}+by^{3} \\
   & \quad\quad +2bxyz+cy^{2}z-2cxz^{2} -dyz^{2}-ez^{3}-2xy^{2}+x^{2}z, \\
f_2 &= b^{2}cxz^{2}-b^{2}dyz^{2}-2bcxyz+2bdy^{2}z-c^{2}xz^{2}-bdxz^{2} \\
    & \quad \quad +cdyz^{2}+beyz^{2}+cxy^{2}-dy^{3}+2cx^{2}z-2ey^{2}z+exz^{2}-x^{3}, \text{ and } \\
f_3 &= b^{3}xz^{2}-b^{2}dz^{3}-2b^{2}xyz-2bcxz^{2}+2bdyz^{2}+cdz^{3} \\
  & \quad \quad +bez^{3}+bxy^{2}+2bx^{2}z+2cxyz-dy^{2}z-dxz^{2}-2eyz^{2}-2x^{2}y.
\end{align*}
According to Macaulay2, $\Ter^{2}_{Z/\tilde{C}/H}|_U$ has two irreducible components, both of dimension 4, and both dominating $U$.

A better approach for human-readability is to restrict to strata of $H \cong \PP^4$ over which $Z$ is an $\Spec A_{\vec{c}}$ bundle. These will be the pullbacks of the strata $\moduli_{\Gtil}$ to $x$. Then $\Ter^{2,\leq 4}_{2,0}|_x = \Ter^2_{Z/\tilde{C}/H}$ and $\EE^{2,4}_{2,0}|_x = \Ter^{2,4}_{Z/\tilde{C}/H}$ will restrict to bundles with fibers given by the computations in Section \ref{ssec:territory_computations}. Following the notation of Section \ref{sec:hilbert_stratification}, let $\tau : (\PP^1)^4 \to H$ be the map taking an ordered tuple of points $(q_1, \ldots, q_4)$ to the same points without order, let $\Diag_{\mathcal{P}} \subseteq (\PP^1)^4$ be the locus in which the $q_i$ coincide whose indices belong to the same part of the partition $\mathcal{P}$, and let $\Diag_{\mathcal{P}}^\circ$ be the induced locally closed loci.

If $\mathcal{P} = \{ P_1, \ldots, P_k \}$ and $c_i = |P_i|$ for each $i$, we claim $Z|_{\Diag_{\mathcal{P}}^\circ}$ is a $\Spec A_{\vec{c}}$ bundle. This can be seen explicitly as follows. The restriction of $\tau$ to $U$ can be written in coordinates as the map $\AA^4 \to U$ sending $(x_1, \ldots, x_4) \mapsto (-\sigma_1, \sigma_2, -\sigma_3, \sigma_4)$ where $\sigma_i$ is the degree $i$ elementary symmetric polynomial. Then we have
\[
 \OO_{Z} |_U\cong \QQ[x_1, \ldots, x_4][t]/((t - x_1) \cdots (t - x_4)).
\]
Restricting further to $\Diag_{\mathcal{P}}^\circ \times_H U$ sets $t - x_i$ equal to $t - x_j$ for $i \sim_{\mathcal{P}} j$ and makes $t - x_i$ relatively prime to $t - x_j$ for $i \not\sim_{\mathcal{P}} j$. Applying the Chinese Remainder Theorem,
\[
  \OO_{Z}|_{\Diag_{\mathcal{P}}^\circ \times_H U} \cong \OO_{\Diag_{\mathcal{P}}^\circ \times_H U} \otimes A_{\vec{c}}.
\]
The claim follows by applying the symmetries of $\PP^1$ to $U$ to obtain a cover of $\Diag_{\mathcal{P}}$.

At this point, we know that $\Ter^2_{Z/\tilde{C}/H}$ restricts to a Zariski-local $\Ter^2_{A_{\vec{c}}}$-bundle and $\EE^{2,4}_{2,0}$ restricts to a Zariski-local $\Ter^{2,4}_{A_{\vec{c}}}$-bundle on the loci $\Diag_{\mathcal{P}}^\circ$ of $(\PP^1)^4$. Theorem \ref{thm:hilbert_stratification} tells us that by taking images of the loci $\Diag_{\mathcal{P}}$ in $H$, we obtain a locally closed stratification of $H$ on which these will be \'etale local bundles.
Explicitly, there is a $4$-dimensional open stratum $H_{1,1,1,1} \subseteq \PP^4$ where the four points of $Z$ are all distinct,
a $3$-dimensional stratum $H_{2,1,1}$ where two points coincide, and similarly, a $2$-dimensional stratum $H_{3,1}$, a $2$-dimensional stratum $H_{2,2}$ and a $1$-dimensional stratum $H_{4}$, defined by
\[
\begin{array}{rclcrcl}
  H_{1,1,1,1} &=& H - \tau(\Diag_{12/3/4}) & \quad & H_{2,1,1} &=& \tau(\Diag_{12/3/4}) - \tau(\Diag_{123/4} \cup \Diag_{12/34}) \\
  H_{3,1} &=& \tau(\Diag_{123/4}) - \tau(\Diag_{1234}) & \quad & H_{2,2} &=& \tau(\Diag_{12/34}) - \tau(\Diag_{1234}) \\
 & & & & H_{4} &=& \tau(\Diag_{1234}).
\end{array}
\]
Since $\tau$ is proper, each stratum has a natural locally closed scheme structure. Explicit equations are unpleasant: for example, $\ol{H}_{2,1,1}$ is cut out inside of $\PP^4$ by the homogeneous degree four discriminant
\begin{align*}
& D_4 = 256a^3e^3-192a^2bde^2-128a^2c^2e^2+144a^2cd^2e -27a^2d^4+144ab^2ce^2-6ab^2d^2e \\
&\quad -80abc^2de+18abcd^3+16ac^4e-4ac^3d^2-27b^4e^2+18b^3cde -4b^3d^3-4b^2c^3e+b^2c^2d^2.
\end{align*}
In the notation of Section \ref{sec:hilbert_stratification}, using Theorem \ref{thm:hilbert_stratification} and the construction of $\moduli_{\Gtil}$, we have
\[
\begin{array}{rcccl}
  H_{1,1,1,1} &\cong &  \moduli_{\Gtil[1]}|_x & \cong & [\Diag_{1/2/3/4}^\circ / \mathfrak{S}_4] \cong \PP^4 - V(D_4) \\
  H_{2,1,1} &\cong & \moduli_{\Gtil[2]}|_x &\cong&  [\Diag_{12/3/4}^\circ / \mathfrak{S}_2] \\
  H_{2,2} &\cong & \moduli_{\Gtil[3]}|_x = \moduli_{\Gtil[4]}|_x &\cong& [\Diag_{12/34}^\circ / \mathfrak{S}_2] \cong \PP^2 - V(D_2) \\
  H_{4} & \cong& \moduli_{\Gtil[5]}|_x & \cong & \Diag_{1234} \cong \PP^1.
\end{array}
\]
where $D_k$ denotes the homogeneous degree $k$ discriminant. Upstairs, the pullbacks of $\EE^{2,4}_{2,0}|_x = \Ter^{2,4}_{Z/\tilde{C}/H}$ to the various strata of $H$ are:

\begin{tabular}{| c | c | c | c |} \hline
 locus in $H$ & restriction of $\EE^{2,4}_{2,0}$ & fiber of $\Ter^{2}_{Z/\tilde{C}/H}$ & fiber of $\EE^{2,4}_{2,0}$ to locus \\ \hline
 $H_{1,1,1,1}$ & $\EE_{\Gamma_1}|_x$ & $\Ter^2_{A_{(1,1,1,1)}}$ & $\Spec \QQ \sqcup \Spec \QQ \sqcup \Spec \QQ$ \\
 $H_{2,1,1}$ & $\EE_{\Gamma_2}|_x$ & $\Ter^2_{A_{(2,1,1)}}$ & $\Spec \QQ$ \\
 $H_{2,2}$ & $\EE_{\Gamma_3}|_x \sqcup \EE_{\Gamma_4}|_x$ & $\Ter^2_{A_{(2,2)}}$ & $\Spec \QQ \sqcup \GG_m$ \\
 $H_{3,1}$ & $\emptyset$ & $\Ter^2_{A_{(3,1)}}$ & $\emptyset$ \\
 $H_4$ & $\EE_{\Gamma_5}|_x$ & $\Ter^2_{A_{(4)}}$ & $\AA^1$ \\ \hline
\end{tabular}

\bigskip

We can use our arrangement to draw some conclusions about closures of strata of $\EE^{2,4}_{2,0}$. If the closure of $\EE_{\Gamma_i}$ intersects some $\EE_{\Gamma_j}$, for $1\leq i,j \leq 5$, then, because $x \to \moduli_{0,0}$ is a smooth cover, there must be a one-parameter family in $\EE^{2,4}_{2,0}|_x$ witnessing it. Projecting to $H$, the family must have generic member in one of $H_{1,1,1,1},\ldots, H_4$ and a limit in another. Because of the closure relationships of $H_{1,1,1,1},\ldots, H_4$ in $H$, we have containments:
\begin{align*}
  \ol{\EE_{\Gamma_1}} &\subseteq \EE_{\Gamma_1} \sqcup \EE_{\Gamma_2} \sqcup \EE_{\Gamma_3} \sqcup \EE_{\Gamma_4} \sqcup \EE_{\Gamma_5} \\
  \ol{\EE_{\Gamma_2}} &\subseteq \EE_{\Gamma_2} \sqcup \EE_{\Gamma_3} \sqcup \EE_{\Gamma_4} \sqcup \EE_{\Gamma_5} \\
  \ol{\EE_{\Gamma_3}} &\subseteq \EE_{\Gamma_3} \sqcup \EE_{\Gamma_5} \\
  \ol{\EE_{\Gamma_4}} &\subseteq \EE_{\Gamma_4} \sqcup \EE_{\Gamma_5} \\
  \ol{\EE_{\Gamma_5}} &= \EE_{\Gamma_5}.
\end{align*}

By the deformation theory of ADE singularities (see for example, \cite[Section 2.3]{yano_laza_simultaneous_reduction}):
\begin{itemize}
  \item A cusp ($A_2$) may deform into a single node ($A_1$).
  \item A tacnode ($A_3$) may deform into a single cusp ($A_2$), two nodes ($A_1 + A_1$), or one node ($A_1$).
  \item A ramphoid cusp ($A_4$) may deform into a single tacnode ($A_3$), a cusp and a node ($A_2 + A_1$), a single cusp ($A_2$), two nodes ($A_1 + A_1$), or one node ($A_1$).
\end{itemize}
In particular, a tacnode cannot deform to a cusp and a node, and a ramphoid cusp cannot deform to two cusps. This constrains the containments further:
\begin{align*}
  \ol{\EE_{\Gamma_1}} &\subseteq \EE_{\Gamma_1} \sqcup \EE_{\Gamma_2} \sqcup \EE_{\Gamma_3} \sqcup \EE_{\Gamma_4} \sqcup \EE_{\Gamma_5} \\
  \ol{\EE_{\Gamma_2}} &\subseteq \EE_{\Gamma_2} \sqcup \EE_{\Gamma_3} \sqcup \EE_{\Gamma_5} \\
  \ol{\EE_{\Gamma_3}} &\subseteq \EE_{\Gamma_3} \\
  \ol{\EE_{\Gamma_4}} &\subseteq \EE_{\Gamma_4} \sqcup \EE_{\Gamma_5} \\
  \ol{\EE_{\Gamma_5}} &= \EE_{\Gamma_5}.
\end{align*}

Using families of nodes degenerating to cusps, like those of Example \ref{ex:cusp_is_limit_of_nodes}, it is not hard to see that $\EE_{\Gamma_2} \subseteq \ol{\EE_{\Gamma_1}}$, $\EE_{\Gamma_3} \subseteq \ol{\EE_{\Gamma_1}}$, and $\EE_{\Gamma_3} \subseteq \ol{\EE_{\Gamma_2}}$. Similarly, considering the family of curves defined by $(y - x^2 - t)(y + x^2 + t)$ shows that $\EE_{\Gamma_4}$ intersects the closure of $\EE_{\Gamma_1}$. Computations with Macaulay2 verify that all of the remaining containments are equalities:
\begin{align*}
  \ol{\EE_{\Gamma_1}} &= \EE_{\Gamma_1} \sqcup \EE_{\Gamma_2} \sqcup \EE_{\Gamma_3} \sqcup \EE_{\Gamma_4} \sqcup \EE_{\Gamma_5} \\
  \ol{\EE_{\Gamma_2}} &= \EE_{\Gamma_2} \sqcup \EE_{\Gamma_3} \sqcup \EE_{\Gamma_5} \\
  \ol{\EE_{\Gamma_3}} &= \EE_{\Gamma_3} \\
  \ol{\EE_{\Gamma_4}} &= \EE_{\Gamma_4} \sqcup \EE_{\Gamma_5} \\
  \ol{\EE_{\Gamma_5}} &= \EE_{\Gamma_5}.
\end{align*}

\begin{remark} \label{rem:fixed_conductance_needed}
We make one final remark that supports the necessity of restricting to fixed conductance. Let $\nu_1: \PP^1 \to C_1$ be a general equinormalized curve gluing two pairs of points to two nodes,
and let $\nu_2 : \PP^1 \to C_2$ be an equinormalized curve gluing three points to an ordinary 3-fold point.
By our computation of $\Ter^{2}_{A_{\vec{c}}}$ in the $\vec{c} = (1,1,1,1)$ case, the preimage under the forgetful morphism $\Ter^{2,\leq 4}_{2,0} \to \EE^{2}_{2,0}$ of $\nu_1$ is a single point, while the preimage of $\nu_2$ is one-dimensional since the spurious point can vary arbitrarily in $\tilde{C}$ without changing $C$. Since $H|_x$ is connected, it follows that the dimension of fibers of $\Ter^{2,\leq 4}_{2,0} \to \EE^2_{2,0}$ is not locally constant, so this morphism is not an equivalence of fibered categories or even smooth. (In particular, it cannot be used to show algebraicity of $\EE^{2}_{2,0}$.)
\end{remark}

\section*{Index of notation}

We provide an index of notation to help the reader. Entries are organized by first appearance.

{\small
\begin{longtable}[H]{@{}>{$}l<{$} p{0.58\textwidth} l@{}}
\toprule
\textnormal{Symbol} & \textnormal{Meaning} & \textnormal{Defined} \\
\midrule
\endfirsthead
\toprule
\textnormal{Symbol} & \textnormal{Meaning} & \textnormal{Defined} \\
\midrule
\endhead
\bottomrule
\endfoot

\multicolumn{3}{@{}l}{\itshape\S1. Introduction and main objects}\\[2pt]
\UU_{g,n} & Stack of all connected, reduced, proper $n$-pointed curves of arithmetic genus $g$. & Def. \ref{def:U_gn} \\
\EE_{g,n} & Stack parametrizing families of $n$-pointed equinormalized curves $\nu : \tilde{C} \to C$ of arithmetic genus $g$. & Def. \ref{def:E_gn} \\
\EE^{\delta,c}_{g,n} & Locally closed substack of $\EE_{g,n}$: total $\delta$-invariant $\delta$, total conductance $c$. & Def. \ref{def:E_gn_delta_c} \\
\EE_\Gamma & Stratum in $\EE_{g,n}$ of equinormalized curves of combinatorial type $\Gamma$. & Def. \ref{def:e_gamma_moduli_functor}, \ref{def:strata_by_quotients} \\

\addlinespace[3pt]
\multicolumn{3}{@{}l}{\itshape\S\ref{sec:territories_algebras}. Moduli of subsheaves of algebras}\\[2pt]
\Ter^\delta_{\mathscr{A}} & The \emph{territory} of a finite locally free $\OO_S$-algebra $\mathscr{A}$; the moduli scheme parametrizing $\OO_S$-subalgebras of $\mathscr{A}$ of corank $\delta$. & Def. \ref{def:territory}, Thm. \ref{thm:territory_representable}\\
\Delta & Cokernel $\mathscr{A}_T/\mathscr{B}$ of a subalgebra inclusion. & Def. \ref{def:territory} \\
\cond_{\mathscr{A}_T/\mathscr{B}} & Conductor ideal sheaf of $\mathscr{B}$ in $\mathscr{A}_T$; equals $\AAnn_{\mathscr{B}}(\Delta)$. & Def. \ref{def:conductor} \\
\Cond_{\mathscr{A}_T/\mathscr{B}} & The same conductor, regarded as an ideal sheaf of $\mathscr{A}_T$ rather than of $\mathscr{B}$. & below Lem. \ref{lem:affine_conductor_largest_ideal} \\
\Delta' & The quotient $\mathscr{B}/\cond_{\mathscr{A}_T/\mathscr{B}}$. & Def. \ref{def:conductor} \\
\delta' & Rank of $\Delta'$. & Def. \ref{def:conductor} \\
c & \emph{Conductance} of a subalgebra; the rank of the conductor locus. & Def. \ref{def:conductance} \\
\Ter^{\delta,c}_{\mathscr{A}} & Locally closed subscheme of $\Ter^\delta_{\mathscr{A}}$ parametrizing subalgebras of conductance exactly $c$. & Def. \ref{def:ter_delta_c}, Thm. \ref{thm:delta_prime_stratification_affine} \\

\addlinespace[3pt]
\multicolumn{3}{@{}l}{\itshape\S\ref{sec:territories_singularities}. Territories for curve singularities}\\[2pt]
\delta(p) & The \emph{delta invariant} of a curve singularity $p$. & Def. \ref{def:sing_basics}\\
m(p) & The \emph{number of branches} of a curve singularity $p$. & Def. \ref{def:sing_basics}\\
g(p) & The \emph{genus} of a curve singularity $p$: $\delta(p) - m(p) + 1$. & Def. \ref{def:sing_basics}\\
\cond(p) & The \emph{conductor ideal} of a curve singularity $p \in C$, as an ideal sheaf of $C$. & Def. \ref{def:sing_basics} \\
\Cond(p) & The \emph{conductor ideal} of a curve singularity $p \in C$, as an ideal sheaf of $\tilde{C}$. & Def. \ref{def:sing_basics} \\
c(p) & The \emph{conductance} of a curve singularity $p$. & Def. \ref{def:sing_basics} \\
c(b) & The \emph{branch conductance} of a branch $b$ of a curve singularity. & Def. \ref{def:sing_basics} \\
\vec{c} & Vector of branch conductances $(c_1, \ldots, c_m)$; usually ordered so that $c_1 \geq \cdots \geq c_m$ with sum $c$. & Def. \ref{def:algebra_for_curve_sing} \\
A_{\vec{c}} & $\prod_{i=1}^m \ZZ[t_i]/(t_i^{c_i})$; germ of the normalization of an $m$-branch singularity with branch conductances $\vec c$. & Def. \ref{def:algebra_for_curve_sing}\\
A^+_{\vec{c}} & Subring of $A_{\vec c}$ with equal constant terms; germ of the seminormalization. & Def. \ref{def:algebra_for_curve_sing}\\
\Tersing(g,\vec{c}) & \emph{Territory of singularities} of genus $g$, conductances $\vec c$: $\Ter^{\delta,c}_{A_{\vec c}} \cap \Ter^{g}_{A^+_{\vec c}}$. & Def. \ref{def:territory_of_singularity}\\
(\mathcal{P}, g) & A \emph{gluing datum}: a partition $\mathcal P$ of $\{1,\dots,m\}$ and a genus function $g:\mathcal P \to \ZZ_{\ge 0}$. & Def. \ref{def:gluing_data} \\
\Tergl(\mathcal{P},g,\vec{c}) & \emph{Territory of a gluing datum}: open and closed locus of $\Ter^{\delta,c}_{A_{\vec{c}}}$ with branches grouped into singularities according to $\mathcal P$ with genera prescribed by $g$. & Def. \ref{def:territory_gluing_datum} \\

\addlinespace[3pt]
\multicolumn{3}{@{}l}{\itshape\S\ref{ssec:crimping}. Relationship to crimping spaces}\\[2pt]
\Aut^{[N]}_{k\to(A,J)} & Group scheme of level-$N$ automorphisms of $A/J^N$ preserving $J/J^N$. & Def. \ref{def:crimping_scheme} \\
\Aut^{[N]}_{k\to B\to(A,J)} & Subgroup additionally preserving $B/J^N$. & Def. \ref{def:crimping_scheme} \\
\Cr & Crimping scheme of a singularity: the fppf quotient $\Aut^{[N]}_{k\to(A,J)}/\Aut^{[N]}_{k\to B\to(A,J)}$. & Def. \ref{def:crimping_scheme}\\
\Aut^{[N]}_{k\to(A,J),0} & Subgroup of $\Aut^{[N]}_{k \to (A,J)}$ preserving the conductor $\cond_{A/B}/J^N$ (no branch-swapping across conductances). & Def. \ref{def:strict_crimping_scheme} \\
\Aut^{[N]}_{k \to B \to (A,J), 0} & Subgroup of $\Aut^{[N]}_{k \to B \to (A,J)}$ preserving the conductor $\cond_{A/B}/J^N$. & Def. \ref{def:strict_crimping_scheme} \\
\Cr_0 & Strict crimping scheme: the quotient by the conductor-preserving subgroups. & Def. \ref{def:strict_crimping_scheme} \\

\addlinespace[3pt]
\multicolumn{3}{@{}l}{\itshape\S\ref{sec:territories_schemes}. Territories of schemes concentrated on subschemes}\\[2pt]
\Ter^\delta_{Z/\tilde{X}/S} & $\delta$-territory of the $S$-scheme $\tilde{X}$ concentrated on the finite closed subscheme $Z$; parametrizes crimpings of $\tilde{X}$ within $Z$ of corank $\delta$. & Def. \ref{def:global_territory}, Def. \ref{def:territory_of_scheme} \\
\Ter^\delta_{\calZ/\tilde{\mathcal{X}}/\mathcal{S}} & Territory of the $\mathcal{S}$-algebraic stack $\tilde{X}$ concentrated on the finite closed substack $\calZ$. & Cor. \ref{cor:territory_of_stack} \\
\Ter^{\delta,c}_{\calZ/\tilde{\mathcal{X}}/\mathcal{S}} & Locally closed substack of $\Ter^\delta_{\calZ/\tilde{\mathcal{X}}/\mathcal{S}}$ with fixed conductance $c$. & Cor. \ref{cor:stacky_conductor_stratification} \\

\addlinespace[3pt]
\multicolumn{3}{@{}l}{\itshape\S\ref{sec:equinormalized_curves}. Equinormalizable curves of fixed conductance}\\[2pt]
\moduli^{\delta}_{g,n} & Moduli of possible source curves for $\EE_{g,n}^{\delta,c}$: smooth, proper, possibly disconnected $n$-marked curves $(\tilde C,(p_i))$ satisfying $\#K = 1 + \delta - g + \sum_{k \in K} g(k)$, where $K$ is the set of connected components. & Def. \ref{def:m_gn_delta} \\
\calH^{\delta,c}_{g,n} & Relative Hilbert scheme of $c$ points of the universal curve over $\moduli^{\delta}_{g,n}$ (a $c$-fold symmetric product). & Def. \ref{def:H_gn_delta_c} \\
\tilde{\mathcal{C}} & Universal curve over $\calH^{\delta,c}_{g,n}$. & Def. \ref{def:H_gn_delta_c} \\
\calZ & Universal closed subscheme of $\tilde{\mathcal{C}}$; to be used as conductor locus, supported on branches of singularities. & Def. \ref{def:H_gn_delta_c} \\
\Ter^{\delta,c}_{g,n} & Open-and-closed substack of $\Ter^{\delta,c}_{\calZ/\tilde{\calC}/\calH^{\delta,c}_{g,n}}$ where $\calC$ is connected of genus $g$; represents $\EE^{\delta,c}_{g,n}$. & Def. \ref{def:ter_gn_delta_c} \\
\Ter^{\delta,\le c}_{g,n} & Bounded-conductance enlargement; proper over $\moduli^{\delta}_{g,n}$, contains $\EE^{\delta,c}_{g,n}$ as an open substack. & Remark \ref{rem:e_bounded_conductance}\\

\addlinespace[3pt]
\multicolumn{3}{@{}l}{\itshape\S\ref{sec:hilbert_stratification}. Stratification of the Hilbert scheme of points}\\[2pt]
\mathfrak{S}_m  & Symmetric group on $\{ 1, \ldots, m \}$. & Page \pageref{sigma_m_first_ref} \\
\calP & A set partition of $\{1, \ldots, m\}$. & above Eq. \eqref{eq:d_tilde_p} \\
\Diag_{\mathcal{P}} & Diagonal in the $m$-fold product $C^m$ for $\mathcal P$: the points $\vec{x}$ such that $x_i=x_j$ when $i\sim_{\mathcal P} j$. & Eq. \eqref{eq:d_tilde_p} \\
\Diag^\circ_{\mathcal{P}} & Open stratum of $\Diag_{\mathcal P}$ (collisions exactly $\mathcal P$). & Eq. \eqref{eq:d_tilde_circ_p} \\
{}[\mathcal{P}] & Integer partition $|P_1|+\cdots+|P_k|=m$ underlying $\mathcal P$. & above Eq. \eqref{eq:d_p} \\
\imDiag_{[\mathcal{P}]} & Scheme-theoretic image of $\Diag_{\mathcal P}$ in $C^{(m)}$. & Eq. \eqref{eq:d_p} \\
\imDiag^\circ_{[\mathcal{P}]} & Open stratum of $\imDiag_{[\mathcal P]}$ (collisions exactly $[\mathcal{P}]$). & Eq. \eqref{eq:d_circ_p} \\
\mathfrak{S}_{\mathcal{P}} & Group of permutations of the parts of $\mathcal{P}$ which preserve the size of the parts. & Eq. \eqref{eq:sigma_p} \\

\addlinespace[3pt]
\multicolumn{3}{@{}l}{\itshape\S\ref{sec:combinatorial_type}. Combinatorial types}\\[2pt]
\Gamma & Combinatorial type $(V=S\sqcup K, B, D, \sigma,\tau,\mu,\epsilon,g,c)$ of an equinormalized curve. & Def. \ref{def:combinatorial_type} \\
\quad\quad S & Singularity vertices. &  \\
\quad\quad K & Component vertices. &  \\
\quad\quad V & Vertices; $V = S \sqcup K$. &  \\
\quad\quad B & Branches/edges. & \\
\quad\quad D & Distinguished points. &  \\
\quad\quad \sigma & Branch$\to$singularity function. &  \\
\quad\quad \tau & Branch$\to$component function. &  \\
\quad\quad \mu & Marking function. &  \\
\quad\quad \epsilon & Distinguished-point$\to$component function. & \\
\quad\quad g & Vertex$\to$genus function. &  \\
\quad\quad c & Branch$\to$branch-conductance function. & \\
g(\Gamma) & Genus $\sum_{v \in V} g(v) + \#B - \#V + 1$ of a combinatorial type $\Gamma$. & Def. \ref{def:comb_type_addition_notions} \\
\val(s) & Valence of a singularity vertex $s$. & Def. \ref{def:comb_type_addition_notions}\\
\delta(s) & $\delta$-invariant $g(s)+\val(s)-1$ of a singularity vertex $s$. & Def. \ref{def:comb_type_addition_notions} \\
\Gtil & Underlying $\tilde C$-type of $\Gamma$ (forgetting $S$ and $\sigma$). & Def. \ref{def:c_tilde_type} \\

\addlinespace[3pt]
\multicolumn{3}{@{}l}{\itshape\S\ref{sec:stratification}. Stratification by combinatorial type}\\[2pt]
\calH_{\vec{c}} & Locally closed stratum of $\calH^{\delta,c}_{g,n}$ where the points of $\calZ$ have multiplicities exactly $\vec c$; branches \emph{unordered}. & Def. \ref{def:H_vec_c}\\
\Hrig_{\vec{c}} & Cover of $\calH_{\vec c}$ with ordered branch markings $q_1, \ldots, q_m$. & Def. \ref{def:H_ord_vec_c} \\
\mathfrak{S}_{\vec{c}} & Subgroup of $\mathfrak S_m$ permuting equal-conductance indices. & Eq. \eqref{eq:sigma_vec_c} \\
\Terunrig_{\vec{c}} & Pullback of $\Ter^{\delta,c}_{\calZ/\tilde\calC/\calH^{\delta,c}_{g,n}}$ to $\calH_{\vec c}$ (unordered branches). & Def. \ref{def:ter_vec_c_and_friends} \\
\Terrig_{\vec{c}} & Pullback $\Ter^{\delta,c}_{\calZ/\tilde\calC/\calH^{\delta,c}_{g,n}}$ to $\Hrig_{\vec c}$ (ordered branches). & Def. \ref{def:ter_vec_c_and_friends}\\
\calZ_{\vec{c}},\ \Zrig_{\vec{c}} & Pullbacks of $\calZ$ to $\calH_{\vec c}$ and $\Hrig_{\vec c}$ (likewise $\Crig_{\vec c},\Ctilrig_{\vec c}$). & Def. \ref{def:ter_vec_c_and_friends}\\

(\Gamma,\beta),\ (\Gtil,\beta) & Rigidified combinatorial / $\tilde C$-type: adds a bijection $\beta:\{1,\dots,m\}\to B$ with $c(\beta(i))=c_i$. & Def. \ref{def:rigid_combinatorial_type}\\

\calH_{(\Gtil,\beta)} & Locally closed locus of $\Hrig_{\vec c}$ with rigidified $\tilde C$-type $(\Gtil,\beta)$. & Def. \ref{def:H_gamma_beta} \\

\Ter_{(\Gamma,\beta)} & Locally closed locus of $\Terrig_{\vec c}$ with combinatorial type $(\Gamma,\beta)$; the rigidified stratum territory. & Def. \ref{def:ter_gamma_beta} \\

T_{\mathcal{P},g} & Open-and-closed piece of $\Terrig_{\vec c}$ indexed by a gluing datum $(\mathcal P,g)$. & Lem. \ref{lem:territory_bundle_decomposition} \\
\moduli_{\Gtil} & $[(\calH_{(\Gtil,\beta)}\cdot\mathfrak S_{\vec c})/\mathfrak S_{\vec c}]$; base of the $\tilde C$-type bundle, a substack of $\calH^{\delta,c}_{g,n}$. & Def. \ref{def:strata_by_quotients}, Thm. \ref{thm:e_gamma_meh_fiber_bundle} \\
\moduli_\Gamma & $[\prod_{v\in K}\moduli_{g(v),n(v)+m(v)}/\Aut(\Gamma)]$; base of the combinatorial-type fiber bundle (Thm.~IV). & Thm. \ref{thm:e_gamma_fiber_bundle}\\
n(v),\ m(v) & Number of distinguished points adjacent to $v$; number of branches incident to $v$. & Thm. \ref{thm:e_gamma_fiber_bundle}\\

\end{longtable}
}

\bibliographystyle{amsalpha}
\bibliography{stratification}

\end{document}

%% file: macros.tex
\newtheorem{theorem}{Theorem}
\newtheorem{proposition}[theorem]{Proposition}
\newtheorem{lemma}[theorem]{Lemma}
\newtheorem{corollary}[theorem]{Corollary}

\theoremstyle{definition}
\newtheorem{definition}[theorem]{Definition}
\newtheorem{example}[theorem]{Example}
\newtheorem{remark}[theorem]{Remark}

\numberwithin{theorem}{section}

\makeatletter
\providecommand{\leftsquigarrow}{%
  \mathrel{\mathpalette\reflect@squig\relax}%
}
\newcommand{\reflect@squig}[2]{%
  \reflectbox{$\m@th#1\rightsquigarrow$}%
}
\makeatother

\renewcommand{\phi}{\varphi} 

\newcommand{\id}{\mathrm{id}}

\newcommand{\ol}[1]{\overline{#1}}
\newcommand{\emphbf}[1]{\textbf{#1}}
\newcommand{\RN}[1]{  
  \textup{\uppercase\expandafter{\romannumeral#1}}%
}

\definecolor{sebgreen1}{rgb}{0.019,0.317,0.149}
\definecolor{sebgreen2}{rgb}{0.784,0.952,0.780}

\definecolor{chrisyellow1}{rgb}{1,0.73,0}
\definecolor{chrisyellow2}{rgb}{0.99,0.99,0.75}

\definecolor{sebdarkgreen}{rgb}{0.019,0.317,0.149}
\definecolor{sebgreen}{RGB}{36,157,55}
\definecolor{seblightgreen}{rgb}{0.784,0.952,0.780}
\definecolor{sebblue}{RGB}{0,123,194}
\definecolor{seblightblue}{RGB}{194,242,255}

\usepackage[pagebackref]{hyperref}
\hypersetup{
  colorlinks   = true,          
  urlcolor     = sebblue,          
  linkcolor    = purple,          
  citecolor   = sebblue             
}

\newcommand{\mm}{\mathfrak{m}}

\newcommand{\pp}{\mathfrak{p}}
\newcommand{\qq}{\mathfrak{q}}

\renewcommand{\AA}{\mathbb{A}}
\newcommand{\CC}{\mathbb{C}}
\newcommand{\FF}{\mathbb{F}}
\newcommand{\GG}{\mathbb{G}}

\newcommand{\PP}{\mathbb{P}}

\newcommand{\QQ}{\mathbb{Q}}
\newcommand{\ZZ}{\mathbb{Z}}

\newcommand{\OO}{\mathscr{O}}

\newcommand{\UU}{\mathcal{U}}

\newcommand{\calC}{\mathcal{C}}
\newcommand{\calH}{\mathcal{H}}
\newcommand{\calP}{\mathcal{P}}
\newcommand{\calZ}{\mathcal{Z}}
\newcommand{\calQ}{\mathcal{Q}}

\newcommand{\Ann}{\mathrm{Ann}}
\newcommand{\AAnn}{\underline{\mathrm{Ann}}}
\newcommand{\Aut}{\mathrm{Aut}}

\newcommand{\cond}{\mathrm{cond}}
\newcommand{\Cond}{\mathrm{Cond}}
\newcommand{\Cr}{\mathrm{Cr}}
\newcommand{\Grass}{\mathrm{Gr}}
\newcommand{\Hilb}{\mathrm{Hilb}}
\newcommand{\Hom}{\mathrm{Hom}}

\newcommand{\EEnd}{\underline{\mathrm{End}}}
\newcommand{\im}{\mathrm{im}}

\newcommand{\EE}{\mathscr{E}}
\newcommand{\moduli}{\mathcal{M}}
\newcommand{\Moduli}{\ol{\mathcal{M}}}

\newcommand{\Proj}{\mathrm{Proj}\,}

\newcommand{\Span}{\mathrm{Span}\,}
\newcommand{\Spec}{\mathrm{Spec}\,}
\newcommand{\Stab}{\mathrm{Stab}}
\renewcommand{\tilde}{\widetilde}

\newcommand{\Ter}{\mathrm{Ter}}
\newcommand{\val}{\mathrm{val}}

\newcommand{\Tersing}{\mathrm{Ter}^{\mathrm{sing}}}   
\newcommand{\Tergl}{\mathrm{Ter}^{\mathrm{gl}}}       
\newcommand{\Terrig}{\mathrm{Ter}^{\mathrm{ord}}}     
\newcommand{\Hrig}{\calH^{\mathrm{ord}}}
\newcommand{\Zrig}{\calZ^{\mathrm{ord}}}
\newcommand{\Crig}{\mathcal{C}^{\mathrm{ord}}}
\newcommand{\Ctilrig}{\tilde{\mathcal{C}}^{\mathrm{ord}}}
\newcommand{\Qrig}{\calQ^{\mathrm{ord}}}
\newcommand{\Terunrig}{\mathrm{Ter}} 
\newcommand{\Gtil}[1][]{\widetilde{\Gamma}%
  \ifx\relax#1\relax\else_{#1}\fi}
\newcommand{\Diag}{\tilde{D}}          
\newcommand{\imDiag}{D}  

\newcommand{\Sch}{\mathrm{Sch}}
\newcommand{\Set}{\mathrm{Set}}

\theoremstyle{definition}
\newtheorem{innercustomthm}{Theorem}
\newenvironment{customthm}[1]
  {\renewcommand\theinnercustomthm{#1}\innercustomthm}
  {\endinnercustomthm}